\documentclass{amsart}
\usepackage{amssymb}
\usepackage{amsmath}
\usepackage{amsfonts}

\newtheorem{theorem}{Theorem}[section]
\theoremstyle{plain}

\newtheorem{corollary}[theorem]{Corollary}

\newtheorem{definition}[theorem]{Definition}

\newtheorem{lemma}[theorem]{Lemma}

\newtheorem{proposition}[theorem]{Proposition}
\newtheorem{remark}[theorem]{Remark}

\numberwithin{equation}{section}

\begin{document}
\title[On Kato-Sobolev spaces]{On Kato-Sobolev spaces}
\author{Gruia Arsu}
\address{Institute of Mathematics of The Romanian Academy\\
P.O. Box 1-174\\
RO-70700\\
Bucharest \\
Romania}
\email{Gruia.Arsu@imar.ro, agmilro@yahoo.com, agmilro@hotmail.com}
\begin{abstract}
We investigate some multiplication properties of Kato-Sobolev spaces by
adapting the techniques used in the study of Beurling algebras by Coifman
and Meyer \cite{Meyer}. We develop an analytic functional calculus for
Kato-Sobolev algebras based on an integral representation formula belonging
A. P. Calder\'{o}n. Also we study the Schatten-von Neumann properties of
pseudo-differential operators with symbols in the Kato-Sobolev spaces.
\end{abstract}
\maketitle

\tableofcontents

\section{Introduction}

In this paper we study some multiplication properties of Kato-Sobolev
spaces, we develop an analytic functional calculus for Kato-Sobolev algebras
and we prove Schatten-von Neumann class properties for pseudo-differential
operators with symbols in the Kato-Sobolev spaces. Kato-Sobolev spaces $%
\mathcal{H}_{\mathtt{ul}}^{s}$ were introduced in \cite{Kato} by Tosio Kato
and are known as uniformly local Sobolev spaces. The uniformly local Sobolev
spaces can be seen as a convenient class of functions with the local Sobolev
property and certain boundedness at infinity. We mention that $\mathcal{H}_{%
\mathtt{ul}}^{s}$ were defined only for integers $s\geq 0$ and play an
essential part in the paper. In this paper, Kato-Sobolev spaces are defined
for arbitrary orders and are proved some embedding theorems (in the spirit
of the \cite{Kato}) which expresses the multiplication properties of the
Kato-Sobolev spaces. The techniques we use in establishing these results are
inspired by techniques used in the study of Beurling algebras by Coifman and
Meyer \cite{Meyer}. Also we develop an analytic functional calculus for
Kato-Sobolev algebras based on an integral representation formula of A. P.
Calder\'{o}n. This part corresponds to the section of \cite{Kato} where the
invertible elements of the algebra $\mathcal{H}_{\mathtt{ul}}^{s}$ are
determined and which has as main result a Wiener type lemma for $\mathcal{H}%
_{\mathtt{ul}}^{s}$. In our case, the main result is the Wiener-L\'{e}vy
theorem for Kato-Sobolev algebras. This theorem allows a spectral analysis
of these algebras. In Section 2 we define Sobolev spaces of multiple order.
Uniformly local Sobolev spaces of multiple order were used as spaces of
symbols of pseudo-differential operators in many papers \cite{Boulkhemair 1}%
, \cite{Boulkhemair 2},.... By adapting the techniques of Coifman and Meyer,
used in the study of Beurling algebras, we prove a result that allows us to
extend the embedding theorems of Kato in the case when the order of $%
\mathcal{H}_{\mathtt{ul}}^{s}$ is not an integer $\geq 0$. In Section 3 we
study an increasing family of spaces $\left\{ \mathcal{K}_{p}^{\mathbf{s}%
}\right\} _{1\leq p\leq \infty }$ for which $\mathcal{K}_{\infty }^{\mathbf{s%
}}=\mathcal{H}_{\mathtt{ul}}^{\mathbf{s}}$. The Wiener-L\'{e}vy theorem for
Kato-Sobolev algebras is established in Section 4. Using this theorem we
build an analytic functional calculus for Kato-Sobolev algebras. The
Schatten-von Neumann class properties for pseudo-differential operators with
symbols in the Kato-Sobolev spaces are presented in the last section.

\section{Sobolev spaces of multiple order}

Let $j\in \left\{ 1,...,n\right\} $. Suppose that $%
\mathbb{R}
^{n}=%
\mathbb{R}
^{n_{1}}\times ...\times 
\mathbb{R}
^{n_{j}}$, where $n_{1},...,n_{j}\in 
\mathbb{N}
^{\ast }$. We have a partition of variables corresponding to this orthogonal
decomposition, $\left\{ 1,...,n\right\} =\bigcup_{l=1}^{j}N_{l}$, where $%
N_{l}=\left\{ k:n_{0}+...+n_{l-1}<k\leq n_{0}+...+n_{l}\right\} $. Here $%
n_{0}=0$ such that $N_{1}=\left\{ 1,...,n_{1}\right\} $.

Let $\mathbf{s}=\left( s_{1},...,s_{j}\right) \in 
\mathbb{R}
^{j}$. We find it convenient to introduce the following space 
\begin{gather*}
\mathcal{H}^{\mathbf{s}}\left( 
\mathbb{R}
^{n}\right) =\left\{ u\in \mathcal{S}^{\prime }\left( 
\mathbb{R}
^{n}\right) :\left( 1-\triangle _{%
\mathbb{R}
^{n_{1}}}\right) ^{s_{1}/2}\otimes ...\otimes \left( 1-\triangle _{%
\mathbb{R}
^{n_{j}}}\right) ^{s_{j}/2}u\in L^{2}\left( 
\mathbb{R}
^{n}\right) \right\} , \\
\left\Vert u\right\Vert _{\mathcal{H}^{\mathbf{s}}}=\left\Vert \left(
1-\triangle _{%
\mathbb{R}
^{n_{1}}}\right) ^{s_{1}/2}\otimes ...\otimes \left( 1-\triangle _{%
\mathbb{R}
^{n_{j}}}\right) ^{s_{j}/2}u\right\Vert _{L^{2}},\quad u\in \mathcal{H}^{%
\mathbf{s}}.
\end{gather*}%
For $\mathbf{s}=\left( s_{1},...,s_{j}\right) \in 
\mathbb{R}
^{j}$ we define the function%
\begin{gather*}
\left\langle \left\langle \cdot \right\rangle \right\rangle ^{\mathbf{s}}:%
\mathbb{R}^{n}=\mathbb{R}^{n_{1}}\times ...\times \mathbb{R}%
^{n_{j}}\rightarrow \mathbb{R}, \\
\left\langle \left\langle \cdot \right\rangle \right\rangle ^{\mathbf{s}%
}=\left\langle \cdot \right\rangle _{\mathbb{R}^{n_{1}}}^{s_{1}}\otimes
...\otimes \left\langle \cdot \right\rangle _{\mathbb{R}^{n_{j}}}^{s_{j}},
\end{gather*}%
where $\left\langle \cdot \right\rangle _{%
\mathbb{R}
^{m}}=\left( 1+\left\vert \cdot \right\vert _{%
\mathbb{R}
^{m}}^{2}\right) ^{1/2}$.Then 
\begin{equation*}
\left\langle \left\langle D\right\rangle \right\rangle ^{\mathbf{s}}=\left(
1-\triangle _{\mathbb{R}^{n_{1}}}\right) ^{s_{1}/2}\otimes ...\otimes \left(
1-\triangle _{\mathbb{R}^{n_{j}}}\right) ^{s_{j}/2},
\end{equation*}%
and 
\begin{gather*}
\mathcal{H}^{\mathbf{s}}=\left\{ u\in \mathcal{S}^{\prime }\left( 
\mathbb{R}
^{n}\right) :\left\langle \left\langle D\right\rangle \right\rangle ^{%
\mathbf{s}}u\in L^{2}\left( 
\mathbb{R}
^{n}\right) \right\} , \\
\left\Vert u\right\Vert _{\mathcal{H}^{\mathbf{s}}}=\left\Vert \left\langle
\left\langle D\right\rangle \right\rangle ^{\mathbf{s}}u\right\Vert
_{L^{2}},\quad u\in \mathcal{H}^{\mathbf{s}}.
\end{gather*}%
Let us note an immediate consequence of Peetre's inequality:%
\begin{equation*}
\left\langle \left\langle \xi +\eta \right\rangle \right\rangle ^{\mathbf{s}%
}\leq 2^{\left\vert \mathbf{s}\right\vert _{1}/2}\left\langle \left\langle
\xi \right\rangle \right\rangle ^{\mathbf{s}}\left\langle \left\langle \eta
\right\rangle \right\rangle ^{\left\vert \mathbf{s}\right\vert },\quad \xi
,\eta \in \mathbb{R}^{n}
\end{equation*}%
where $\left\vert \mathbf{s}\right\vert _{1}=\left\vert s_{1}\right\vert
+...+\left\vert s_{j}\right\vert $ and $\left\vert \mathbf{s}\right\vert
=\left( \left\vert s_{1}\right\vert ,...,\left\vert s_{j}\right\vert \right)
\in 
\mathbb{R}
^{j}$. Also we have 
\begin{equation*}
\left\langle \left\langle \xi \right\rangle \right\rangle ^{\mathbf{s}}\leq
\left\langle \xi \right\rangle ^{\left\vert \mathbf{s}\right\vert
_{1}},\quad \xi \in 
\mathbb{R}
^{n}.
\end{equation*}

Let $k$ be an integer $\geq 0$ or $k=\infty $. We shall use the following
standard notations:%
\begin{gather*}
\mathcal{BC}^{k}\left( \mathbb{R}^{n}\right) =\left\{ f\in \mathcal{C}%
^{k}\left( \mathbb{R}^{n}\right) :f\text{ \textit{and its derivatives of
order }}\leq k\text{ \textit{are bounded}}\right\} , \\
\left\Vert f\right\Vert _{\mathcal{BC}^{l}}=\max_{m\leq l}\sup_{x\in
X}\left\Vert f^{\left( m\right) }\left( x\right) \right\Vert <\infty ,\quad
l<k+1.
\end{gather*}

\begin{proposition}
\label{ks3}Suppose that $%
\mathbb{R}
^{n}=%
\mathbb{R}
^{n_{1}}\times ...\times 
\mathbb{R}
^{n_{j}}$. Let $\mathbf{s}\in 
\mathbb{R}
^{j}$, $k_{s}=\left[ \left\vert \mathbf{s}\right\vert _{1}\right] +n+2$ and $%
m_{s}=\left[ \left\vert \mathbf{s}\right\vert _{1}+\frac{n+1}{2}\right] +1$.

$\left( \mathtt{a}\right) $ $u\in \mathcal{H}^{\mathbf{s}}\left( 
\mathbb{R}
^{n}\right) $ if and only if there is $l\in \left\{ 1,...,j\right\} $ such
that $u,\partial _{k}u\in \mathcal{H}^{\mathbf{s}-\mathbf{\delta }%
_{l}}\left( 
\mathbb{R}
^{n}\right) $ for any $k\in N_{l}$, where $\mathbf{\delta }_{l}=\left(
\delta _{l1},...,\delta _{lj}\right) $.

$\left( \mathtt{b}\right) $ If $\chi \in H^{\left\vert \mathbf{s}\right\vert
_{1}+\frac{n+1}{2}}\left( 
\mathbb{R}
^{n}\right) $, then for every $u\in \mathcal{H}^{\mathbf{s}}\left( 
\mathbb{R}
^{n}\right) $ we have $\chi u\in \mathcal{H}^{\mathbf{s}}\left( 
\mathbb{R}
^{n}\right) $ and%
\begin{eqnarray*}
\left\Vert \chi u\right\Vert _{\mathcal{H}^{\mathbf{s}}} &\leq &C\left(
s,n,\chi \right) \left\Vert u\right\Vert _{\mathcal{H}^{\mathbf{s}}} \\
&\leq &C\left( s,n\right) \left\Vert \chi \right\Vert _{H^{\left\vert 
\mathbf{s}\right\vert _{1}+\frac{n+1}{2}}}\left\Vert u\right\Vert _{\mathcal{%
H}^{\mathbf{s}}},
\end{eqnarray*}%
where 
\begin{eqnarray*}
C\left( s,n,\chi \right) &=&\left( 2\pi \right) ^{-n}2^{\left\vert \mathbf{s}%
\right\vert _{1}/2}\left( \int \left\langle \eta \right\rangle ^{\left\vert 
\mathbf{s}\right\vert _{1}}\left\vert \widehat{\chi }\left( \eta \right)
\right\vert \mathtt{d}\eta \right) \\
&\leq &\left( 2\pi \right) ^{-n}2^{\left\vert \mathbf{s}\right\vert
_{1}/2}\left\Vert \left\langle \cdot \right\rangle ^{-n-1}\right\Vert
_{L^{1}}\left\Vert \chi \right\Vert _{H^{\left\vert \mathbf{s}\right\vert
_{1}+\frac{n+1}{2}}} \\
&=&C\left( s,n\right) \left\Vert \chi \right\Vert _{H^{\left\vert \mathbf{s}%
\right\vert _{1}+\frac{n+1}{2}}}\leq C\left( s,n\right) \left(
\sum_{\left\vert \alpha \right\vert \leq m_{s}}\left\Vert \partial ^{\alpha
}\chi \right\Vert _{L^{2}}\right) .
\end{eqnarray*}%
Here $H^{m}\left( 
\mathbb{R}
^{n}\right) $ is the usual Sobolev space, $m\in 
\mathbb{R}
$.

$\left( \mathtt{c}\right) $ If $\chi \in \mathcal{C}^{k_{s}}\left( 
\mathbb{R}
^{n}\right) $ is $%
\mathbb{Z}
^{n}$-periodic, then for every $u\in \mathcal{H}^{\mathbf{s}}\left( 
\mathbb{R}
^{n}\right) $ we have $\chi u\in \mathcal{H}^{\mathbf{s}}\left( 
\mathbb{R}
^{n}\right) $.

$\left( \mathtt{d}\right) $ If $s_{1}>n_{1}/2,...,s_{j}>n_{j}/2$, then $%
\mathcal{H}^{\mathbf{s}}\left( 
\mathbb{R}
^{n}\right) \subset \mathcal{F}^{-1}L^{1}\left( 
\mathbb{R}
^{n}\right) \subset \mathcal{C}_{\infty }\left( 
\mathbb{R}
^{n}\right) $.

$\left( \mathtt{e}\right) $ Let $\mathbf{k}\in 
\mathbb{N}
^{j}$. If $s_{1}>n_{1}/2+k_{1},...,s_{j}>n_{j}/2+k_{j}$, then for every $%
\alpha =\left( \alpha _{1},...,\alpha _{j}\right) \in 
\mathbb{N}
^{n}=%
\mathbb{N}
^{n_{1}}\times ...\times 
\mathbb{N}
^{n_{j}}$, $\left\vert \alpha _{1}\right\vert \leq k_{1}$, $...$, $%
\left\vert \alpha _{j}\right\vert \leq k_{j}$ we have 
\begin{equation*}
u\in \mathcal{H}^{\mathbf{s}}\left( 
\mathbb{R}
^{n}\right) \text{ }\Rightarrow \text{ }\partial _{%
\mathbb{R}
^{n_{1}}}^{\alpha _{1}}...\partial _{%
\mathbb{R}
^{n_{j}}}^{\alpha _{j}}u\in \mathcal{H}^{\mathbf{s-k}}\left( 
\mathbb{R}
^{n}\right) \subset \mathcal{F}^{-1}L^{1}\left( 
\mathbb{R}
^{n}\right) \subset \mathcal{C}_{\infty }\left( 
\mathbb{R}
^{n}\right) .
\end{equation*}
\end{proposition}

\begin{proof}
$\left( \mathtt{a}\right) $ This part is trivial.

$\left( \mathtt{b}\right) $ Since $\mathcal{S}\left( 
\mathbb{R}
^{n}\right) $ is dense in $\mathcal{H}^{\mathbf{s}}\left( 
\mathbb{R}
^{n}\right) $ and $\mathcal{S}\left( 
\mathbb{R}
^{n}\right) $ is dense in $H^{\left\vert \mathbf{s}\right\vert _{1}+\frac{n+1%
}{2}}\left( 
\mathbb{R}
^{n}\right) $ (see the $\mathcal{B}_{p.k}$ spaces in H\"{o}rmander \cite%
{Hormander} vol. 2 ), we can assume that $\chi ,u\in \mathcal{S}\left( 
\mathbb{R}
^{n}\right) $. In this case we have 
\begin{equation*}
\widehat{\chi u}\left( \xi \right) =\left( 2\pi \right) ^{-n}\widehat{\chi }%
\ast \widehat{u}\left( \xi \right) .
\end{equation*}%
Now we use Peetre's inequality and $\left\langle \left\langle \xi
\right\rangle \right\rangle ^{\left\vert \mathbf{s}\right\vert }\leq
\left\langle \xi \right\rangle ^{\left\vert \mathbf{s}\right\vert _{1}}$ to
obtain%
\begin{equation*}
\left\langle \left\langle \xi \right\rangle \right\rangle ^{\mathbf{s}%
}\left\vert \widehat{\chi u}\left( \xi \right) \right\vert \leq \left( 2\pi
\right) ^{-n}2^{\left\vert \mathbf{s}\right\vert _{1}/2}\left( \int
\left\langle \xi -\eta \right\rangle ^{\left\vert \mathbf{s}\right\vert
_{1}}\left\vert \widehat{\chi }\left( \xi -\eta \right) \right\vert
\left\langle \left\langle \eta \right\rangle \right\rangle ^{\mathbf{s}%
}\left\vert \widehat{u}\left( \eta \right) \right\vert \mathtt{d}\eta
\right) .
\end{equation*}%
Then Schur's lemma implies that%
\begin{eqnarray*}
\left\Vert \chi u\right\Vert _{\mathcal{H}^{\mathbf{s}}} &=&\left\Vert
\left\langle \left\langle \cdot \right\rangle \right\rangle ^{\mathbf{s}%
}\left\vert \widehat{\chi u}\right\vert \right\Vert _{L^{2}} \\
&\leq &\left( 2\pi \right) ^{-n}2^{\left\vert \mathbf{s}\right\vert
_{1}/2}\left( \int \left\langle \eta \right\rangle ^{\left\vert \mathbf{s}%
\right\vert _{1}}\left\vert \widehat{\chi }\left( \eta \right) \right\vert 
\mathtt{d}\eta \right) \left\Vert \left\langle \left\langle \cdot
\right\rangle \right\rangle ^{\mathbf{s}}\left\vert \widehat{u}\right\vert
\right\Vert _{L^{2}} \\
&=&C\left( s,n,\chi \right) \left\Vert u\right\Vert _{\mathcal{H}^{\mathbf{s}%
}}
\end{eqnarray*}%
and Schwarz inequality gives the estimate of $C\left( s,n,\chi \right) $%
\begin{eqnarray*}
C\left( s,n,\chi \right) &=&\left( 2\pi \right) ^{-n}2^{\left\vert \mathbf{s}%
\right\vert _{1}/2}\left( \int \left\langle \eta \right\rangle ^{\left\vert 
\mathbf{s}\right\vert _{1}}\left\vert \widehat{\chi }\left( \eta \right)
\right\vert \mathtt{d}\eta \right) \\
&\leq &\left( 2\pi \right) ^{-n}2^{\left\vert \mathbf{s}\right\vert
_{1}/2}\left\Vert \left\langle \cdot \right\rangle ^{-n-1}\right\Vert
_{L^{1}}\left\Vert \chi \right\Vert _{H^{\left\vert \mathbf{s}\right\vert
_{1}+\frac{n+1}{2}}} \\
&=&C\left( s,n\right) \left\Vert \chi \right\Vert _{H^{\left\vert \mathbf{s}%
\right\vert _{1}+\frac{n+1}{2}}}\leq C\left( s,n\right) \left(
\sum_{\left\vert \alpha \right\vert \leq m_{s}}\left\Vert \partial ^{\alpha
}\chi \right\Vert _{L^{2}}\right) .
\end{eqnarray*}

$\left( \mathtt{c}\right) $ We shall use some results from \cite{Hormander}
vol. 1, pp 177-179, concerning periodic distributions. If $\chi \in \mathcal{%
C}^{k_{s}}\left( 
\mathbb{R}
^{n}\right) $ is $%
\mathbb{Z}
^{n}$-periodic, then 
\begin{equation*}
\chi =\sum_{\gamma \in 
\mathbb{Z}
^{n}}\mathtt{e}^{2\pi \mathtt{i}\left\langle \cdot ,\gamma \right\rangle
}c_{\gamma },
\end{equation*}%
with Fourier coefficients 
\begin{equation*}
c_{\gamma }=\int_{\mathtt{I}}\chi \left( x\right) \mathtt{e}^{-2\pi \mathtt{i%
}\left\langle x,\gamma \right\rangle }\mathtt{d}x,\quad \mathtt{I}=\left[
0,1\right) ^{n},\quad \gamma \in 
\mathbb{Z}
^{n},
\end{equation*}%
satisfying%
\begin{equation*}
\left\vert c_{\gamma }\right\vert \leq Cst\left\Vert \chi \right\Vert _{%
\mathcal{BC}^{k_{s}}\left( 
\mathbb{R}
^{n}\right) }\left\langle 2\pi \gamma \right\rangle ^{-k_{s}},\quad \gamma
\in 
\mathbb{Z}
^{n}.
\end{equation*}%
Since $\widehat{\mathtt{e}^{\mathtt{i}\left\langle \cdot ,\eta \right\rangle
}u}=\widehat{u}\left( \cdot -\eta \right) $, then Peetre's inequality
implies that%
\begin{equation*}
\left\Vert \mathtt{e}^{\mathtt{i}\left\langle \cdot ,\eta \right\rangle
}u\right\Vert _{\mathcal{H}^{\mathbf{s}}}\leq 2^{\left\vert \mathbf{s}%
\right\vert _{1}/2}\left\langle \left\langle \eta \right\rangle
\right\rangle ^{\left\vert \mathbf{s}\right\vert }\left\Vert u\right\Vert _{%
\mathcal{H}^{\mathbf{s}}}\leq 2^{\left\vert \mathbf{s}\right\vert
_{1}/2}\left\langle \eta \right\rangle ^{\left\vert \mathbf{s}\right\vert
_{1}}\left\Vert u\right\Vert _{\mathcal{H}^{\mathbf{s}}}.
\end{equation*}%
It follows that%
\begin{eqnarray*}
\left\Vert \chi u\right\Vert _{\mathcal{H}^{\mathbf{s}}} &\leq &Cst\cdot
2^{\left\vert \mathbf{s}\right\vert _{1}/2}\left\Vert \chi \right\Vert _{%
\mathcal{BC}^{k_{s}}\left( 
\mathbb{R}
^{n}\right) }\left( \sum_{\gamma \in 
\mathbb{Z}
^{n}}\left\langle 2\pi \gamma \right\rangle ^{-k_{s}}\left\langle 2\pi
\gamma \right\rangle ^{\left\vert \mathbf{s}\right\vert _{1}}\right)
\left\Vert u\right\Vert _{\mathcal{H}^{\mathbf{s}}} \\
&\leq &Cst\cdot 2^{\left\vert \mathbf{s}\right\vert _{1}/2}\left(
\sum_{\gamma \in 
\mathbb{Z}
^{n}}\left\langle 2\pi \gamma \right\rangle ^{-n-1}\right) \left\Vert \chi
\right\Vert _{\mathcal{BC}^{k_{s}}\left( 
\mathbb{R}
^{n}\right) }\left\Vert u\right\Vert _{\mathcal{H}^{\mathbf{s}}}.
\end{eqnarray*}

$\left( \mathtt{d}\right) $ Let $u\in \mathcal{H}^{\mathbf{s}}$. If $%
s_{1}>n_{1}/2,...,s_{j}>n_{j}/2$, then $\widehat{u}\in L^{1}\left( 
\mathbb{R}
^{n}\right) $ since $\left\langle \left\langle \cdot \right\rangle
\right\rangle ^{-\mathbf{s}}$, $\left\langle \left\langle \cdot
\right\rangle \right\rangle ^{\mathbf{s}}\widehat{u}\in L^{2}\left( 
\mathbb{R}
^{n}\right) $. Now the Riemann-Lebesgue lemma implies the result.

$\left( \mathtt{e}\right) $ is a consequence of $\left( \mathtt{a}\right) $
and $\left( \mathtt{d}\right) $.
\end{proof}

\begin{lemma}
Let $\varphi \in \mathcal{S}\left( 
\mathbb{R}
^{n}\right) $ and $\theta \in \left[ 0,2\pi \right] ^{n}$. If 
\begin{equation*}
\varphi _{\theta }=\sum_{\gamma \in 
\mathbb{Z}
^{n}}\mathtt{e}^{\mathtt{i}\left\langle \gamma ,\theta \right\rangle
}\varphi \left( \cdot -\gamma \right) =\sum_{\gamma \in 
\mathbb{Z}
^{n}}\mathtt{e}^{\mathtt{i}\left\langle \gamma ,\theta \right\rangle }\tau
_{\gamma }\varphi ,
\end{equation*}%
then 
\begin{equation*}
\widehat{\varphi }_{\theta }=\nu _{\theta }=\left( 2\pi \right)
^{n}\sum_{\gamma \in 
\mathbb{Z}
^{n}}\widehat{\varphi }\left( 2\pi \gamma +\theta \right) \delta _{2\pi
\gamma +\theta }.
\end{equation*}
\end{lemma}

\begin{proof}
We have%
\begin{equation*}
\varphi _{\theta }=\sum_{\gamma \in 
\mathbb{Z}
^{n}}\mathtt{e}^{\mathtt{i}\left\langle \gamma ,\theta \right\rangle
}\varphi \left( \cdot -\gamma \right) =\sum_{\gamma \in 
\mathbb{Z}
^{n}}\mathtt{e}^{\mathtt{i}\left\langle \gamma ,\theta \right\rangle }\delta
_{\gamma }\ast \varphi =\varphi \ast \left( \mathtt{e}^{\mathtt{i}%
\left\langle \cdot ,\theta \right\rangle }S\right) ,
\end{equation*}%
where $S=\sum_{\gamma \in 
\mathbb{Z}
^{n}}\delta _{\gamma }$. We apply Poisson's summation formula, $\mathcal{F}%
\left( \sum_{\gamma \in 
\mathbb{Z}
^{n}}\delta _{\gamma }\right) =\left( 2\pi \right) ^{n}\sum_{\gamma \in 
\mathbb{Z}
^{n}}\delta _{2\pi \gamma }$, to obtain 
\begin{eqnarray*}
\widehat{\varphi }_{\theta } &=&\widehat{\varphi }\cdot \widehat{\left( 
\mathtt{e}^{\mathtt{i}\left\langle \cdot ,\theta \right\rangle }S\right) }=%
\widehat{\varphi }\cdot \tau _{\theta }\widehat{S}=\left( 2\pi \right) ^{n}%
\widehat{\varphi }\sum_{\gamma \in 
\mathbb{Z}
^{n}}\delta _{2\pi \gamma +\theta } \\
&=&\left( 2\pi \right) ^{n}\sum_{\gamma \in 
\mathbb{Z}
^{n}}\widehat{\varphi }\left( 2\pi \gamma +\theta \right) \delta _{2\pi
\gamma +\theta }.
\end{eqnarray*}
\end{proof}

Above and in the rest of the paper for any $x\in 
\mathbb{R}
^{n}$ and for any distribution $u$ on $%
\mathbb{R}
^{n}$, by $\tau _{x}u$ we shall denote the translation by $x$ of $u$, i.e. $%
\tau _{x}u=u\left( \cdot -x\right) =\delta _{x}\ast u$.

As we already said the techniques of Coifman and Meyer, used in the study of
Beurling algebras $A_{\omega }$ and $B_{\omega }$ (see \cite{Meyer} pp
7-10), can be adapted to the case of Sobolev spaces $\mathcal{H}^{\mathbf{s}%
}\left( 
\mathbb{R}
^{n}\right) $. An example is the following result.

\begin{lemma}
\label{ks2}Let $\mathbf{s}\in 
\mathbb{R}
^{j}$. Let $\left\{ u_{\gamma }\right\} _{\gamma \in 
\mathbb{Z}
^{n}}$ be a a family of elements from $\mathcal{H}^{\mathbf{s}}\left( 
\mathbb{R}
^{n}\right) \cap \mathcal{D}_{K}^{\prime }\left( 
\mathbb{R}
^{n}\right) $, where $K\subset 
\mathbb{R}
^{n}$ is a compact subset such that $\left( K-K\right) \cap 
\mathbb{Z}
^{n}=\left\{ 0\right\} $. Put%
\begin{equation*}
u=\sum_{\gamma \in 
\mathbb{Z}
^{n}}\tau _{\gamma }u_{\gamma }=\sum_{\gamma \in 
\mathbb{Z}
^{n}}u_{\gamma }\left( \cdot -\gamma \right) =\sum_{\gamma \in 
\mathbb{Z}
^{n}}\delta _{\gamma }\ast u_{\gamma }\in \mathcal{D}^{\prime }\left( 
\mathbb{R}
^{n}\right) .
\end{equation*}%
Then the following statements are equivalent:

$\left( \mathtt{a}\right) $ $u\in \mathcal{H}^{\mathbf{s}}\left( 
\mathbb{R}
^{n}\right) $.

$\left( \mathtt{b}\right) $ $\sum_{\gamma \in 
\mathbb{Z}
^{n}}\left\Vert u_{\gamma }\right\Vert _{\mathcal{H}^{\mathbf{s}%
}}^{2}<\infty .$

Moreover, there is $C\geq 1$, which does not depend on the family $\left\{
u_{\gamma }\right\} _{\gamma \in 
\mathbb{Z}
^{n}}$, such that%
\begin{equation}
C^{-1}\left\Vert u\right\Vert _{\mathcal{H}^{\mathbf{s}}}\leq \left(
\sum_{\gamma \in 
\mathbb{Z}
^{n}}\left\Vert u_{\gamma }\right\Vert _{\mathcal{H}^{\mathbf{s}%
}}^{2}\right) ^{1/2}\leq C\left\Vert u\right\Vert _{\mathcal{H}^{\mathbf{s}%
}}.  \label{ks1}
\end{equation}
\end{lemma}

\begin{proof}
Let us choose $\varphi \in \mathcal{C}_{0}^{\infty }\left( 
\mathbb{R}
^{n}\right) $ such that $\varphi =1$ on $K$ and \texttt{supp}$\varphi
=K^{\prime }$ satisfies the condition $\left( K^{\prime }-K^{\prime }\right)
\cap 
\mathbb{Z}
^{n}=\left\{ 0\right\} $. For $\theta \in \left[ 0,2\pi \right] ^{n}$ we set 
\begin{eqnarray*}
\varphi _{\theta } &=&\sum_{\gamma \in 
\mathbb{Z}
^{n}}\mathtt{e}^{\mathtt{i}\left\langle \gamma ,\theta \right\rangle }\tau
_{\gamma }\varphi =\sum_{\gamma \in 
\mathbb{Z}
^{n}}\mathtt{e}^{\mathtt{i}\left\langle \gamma ,\theta \right\rangle }\delta
_{\gamma }\ast \varphi , \\
u_{\theta } &=&\sum_{\gamma \in 
\mathbb{Z}
^{n}}\mathtt{e}^{\mathtt{i}\left\langle \gamma ,\theta \right\rangle }\tau
_{\gamma }u_{\gamma }=\sum_{\gamma \in 
\mathbb{Z}
^{n}}\mathtt{e}^{\mathtt{i}\left\langle \gamma ,\theta \right\rangle }\delta
_{\gamma }\ast u_{\gamma }.
\end{eqnarray*}%
Since $\left( K^{\prime }-K^{\prime }\right) \cap 
\mathbb{Z}
^{n}=\left\{ 0\right\} $ we have 
\begin{equation*}
u_{\theta }=\varphi _{\theta }u,\quad \quad u=\varphi _{\theta }u_{-\theta }.
\end{equation*}

\textit{Step 1.} Suppose first that the family $\left\{ u_{\gamma }\right\}
_{\gamma \in 
\mathbb{Z}
^{n}}$ has only a finite number of non-zero terms and we shall prove in this
case the estimate(\ref{ks1}). Since $u_{\theta }$, $u\in \mathcal{E}^{\prime
}\subset \mathcal{S}^{\prime }$ it follows that%
\begin{equation*}
\widehat{u}_{\theta }=\left( 2\pi \right) ^{-n}\nu _{\theta }\ast \widehat{u}%
,\quad \quad \widehat{u}=\left( 2\pi \right) ^{-n}\nu _{\theta }\ast 
\widehat{u}_{-\theta },
\end{equation*}%
where $\nu _{\theta }=\widehat{\varphi }_{\theta }=\left( 2\pi \right)
^{n}\sum_{\gamma \in 
\mathbb{Z}
^{n}}\widehat{\varphi }\left( 2\pi \gamma +\theta \right) \delta _{2\pi
\gamma +\theta }$ is a measure of rapid decay at $\infty $. Since $\widehat{u%
}_{\theta }$, $\widehat{u}\in \mathcal{C}_{pol}^{\infty }\left( 
\mathbb{R}
^{n}\right) $ we get the pointwise equalities%
\begin{eqnarray*}
\widehat{u}_{\theta }\left( \xi \right) &=&\sum_{\gamma \in 
\mathbb{Z}
^{n}}\widehat{\varphi }\left( 2\pi \gamma +\theta \right) \widehat{u}\left(
\xi -2\pi \gamma -\theta \right) , \\
\widehat{u}\left( \xi \right) &=&\sum_{\gamma \in 
\mathbb{Z}
^{n}}\widehat{\varphi }\left( 2\pi \gamma +\theta \right) \widehat{u}%
_{-\theta }\left( \xi -2\pi \gamma -\theta \right) .
\end{eqnarray*}%
By using Peetre's inequality we obtain%
\begin{multline*}
\left\langle \left\langle \xi \right\rangle \right\rangle ^{\mathbf{s}%
}\left\vert \widehat{u}_{\theta }\left( \xi \right) \right\vert \leq
2^{\left\vert \mathbf{s}\right\vert _{1}/2}\sum_{\gamma \in 
\mathbb{Z}
^{n}}\left\langle \left\langle 2\pi \gamma +\theta \right\rangle
\right\rangle ^{\left\vert \mathbf{s}\right\vert }\left\vert \widehat{%
\varphi }\left( 2\pi \gamma +\theta \right) \right\vert \\
\cdot \left\langle \left\langle \xi -2\pi \gamma -\theta \right\rangle
\right\rangle ^{\mathbf{s}}\left\vert \widehat{u}\left( \xi -2\pi \gamma
-\theta \right) \right\vert ,
\end{multline*}%
and%
\begin{multline*}
\left\langle \left\langle \xi \right\rangle \right\rangle ^{\mathbf{s}}%
\widehat{u}\left( \xi \right) \leq 2^{\left\vert \mathbf{s}\right\vert
_{1}/2}\sum_{\gamma \in 
\mathbb{Z}
^{n}}\left\langle \left\langle 2\pi \gamma +\theta \right\rangle
\right\rangle ^{\left\vert \mathbf{s}\right\vert }\left\vert \widehat{%
\varphi }\left( 2\pi \gamma +\theta \right) \right\vert \\
\cdot \left\langle \left\langle \xi -2\pi \gamma -\theta \right\rangle
\right\rangle ^{\mathbf{s}}\left\vert \widehat{u}_{-\theta }\left( \xi -2\pi
\gamma -\theta \right) \right\vert .
\end{multline*}%
From here we obtain further that 
\begin{eqnarray*}
\left\Vert u_{\theta }\right\Vert _{\mathcal{H}^{\mathbf{s}}} &=&\left\Vert
\left\langle \left\langle \cdot \right\rangle \right\rangle ^{\mathbf{s}}%
\widehat{u}_{\theta }\right\Vert _{L^{2}} \\
&\leq &2^{\left\vert \mathbf{s}\right\vert _{1}/2}\left( \sum_{\gamma \in 
\mathbb{Z}
^{n}}\left\langle \left\langle 2\pi \gamma +\theta \right\rangle
\right\rangle ^{\left\vert \mathbf{s}\right\vert }\left\vert \widehat{%
\varphi }\left( 2\pi \gamma +\theta \right) \right\vert \right) \left\Vert
\left\langle \left\langle \cdot \right\rangle \right\rangle ^{\mathbf{s}}%
\widehat{u}\right\Vert _{L^{2}} \\
&=&2^{\left\vert \mathbf{s}\right\vert _{1}/2}\left( \sum_{\gamma \in 
\mathbb{Z}
^{n}}\left\langle \left\langle 2\pi \gamma +\theta \right\rangle
\right\rangle ^{\left\vert \mathbf{s}\right\vert }\left\vert \widehat{%
\varphi }\left( 2\pi \gamma +\theta \right) \right\vert \right) \left\Vert
u\right\Vert _{\mathcal{H}^{\mathbf{s}}} \\
&=&C_{\mathbf{s},n,\varphi }\left\Vert u\right\Vert _{\mathcal{H}^{\mathbf{s}%
}}
\end{eqnarray*}%
and%
\begin{eqnarray*}
\left\Vert u\right\Vert _{\mathcal{H}^{\mathbf{s}}} &\leq &2^{\left\vert 
\mathbf{s}\right\vert _{1}/2}\left( \sum_{\gamma \in 
\mathbb{Z}
^{n}}\left\langle \left\langle 2\pi \gamma +\theta \right\rangle
\right\rangle ^{\left\vert \mathbf{s}\right\vert }\left\vert \widehat{%
\varphi }\left( 2\pi \gamma +\theta \right) \right\vert \right) \left\Vert
u_{-\theta }\right\Vert _{\mathcal{H}^{\mathbf{s}}} \\
&=&C_{\mathbf{s},n,\varphi }\left\Vert u_{-\theta }\right\Vert _{\mathcal{H}%
^{\mathbf{s}}}.
\end{eqnarray*}

The above estimates can be rewritten as 
\begin{eqnarray*}
\int \left\langle \left\langle \xi \right\rangle \right\rangle ^{2\mathbf{s}%
}\left\vert \widehat{u}_{\theta }\left( \xi \right) \right\vert ^{2}\mathtt{d%
}\xi &\leq &C_{\mathbf{s},n,\varphi }^{2}\left\Vert u\right\Vert _{\mathcal{H%
}^{\mathbf{s}}}^{2}, \\
\left\Vert u\right\Vert _{\mathcal{H}^{\mathbf{s}}}^{2} &\leq &C_{\mathbf{s}%
,n,\varphi }^{2}\int \left\langle \left\langle \xi \right\rangle
\right\rangle ^{2\mathbf{s}}\left\vert \widehat{u}_{-\theta }\left( \xi
\right) \right\vert ^{2}\mathtt{d}\xi .
\end{eqnarray*}%
On the other hand, the equality $u_{\theta }=\sum_{\gamma \in 
\mathbb{Z}
^{n}}\mathtt{e}^{\mathtt{i}\left\langle \gamma ,\theta \right\rangle }\tau
_{\gamma }u_{\gamma }$ implies 
\begin{equation*}
\widehat{u}_{\theta }\left( \xi \right) =\sum_{\gamma \in 
\mathbb{Z}
^{n}}\mathtt{e}^{\mathtt{i}\left\langle \gamma ,\theta -\xi \right\rangle }%
\widehat{u}_{\gamma }\left( \xi \right)
\end{equation*}%
with finite sum. The functions $\theta \rightarrow \widehat{u}_{\pm \theta
}\left( \xi \right) $ are in $L^{2}\left( \left[ 0,2\pi \right] ^{n}\right) $
and 
\begin{equation*}
\left( 2\pi \right) ^{-n}\int_{\left[ 0,2\pi \right] ^{n}}\left\vert 
\widehat{u}_{\pm \theta }\left( \xi \right) \right\vert ^{2}\mathtt{d}\theta
=\sum_{\gamma \in 
\mathbb{Z}
^{n}}\left\vert \widehat{u}_{\gamma }\left( \xi \right) \right\vert ^{2}.
\end{equation*}%
Integrating with respect $\theta $ the above inequalities we get that%
\begin{eqnarray*}
\sum_{\gamma \in 
\mathbb{Z}
^{n}}\left\Vert u_{\gamma }\right\Vert _{\mathcal{H}^{\mathbf{s}}}^{2} &\leq
&C_{\mathbf{s},n,\varphi }^{2}\left\Vert u\right\Vert _{\mathcal{H}^{\mathbf{%
s}}}^{2}, \\
\left\Vert u\right\Vert _{\mathcal{H}^{\mathbf{s}}}^{2} &\leq &C_{\mathbf{s}%
,n,\varphi }^{2}\sum_{\gamma \in 
\mathbb{Z}
^{n}}\left\Vert u_{\gamma }\right\Vert _{\mathcal{H}^{\mathbf{s}}}^{2}.
\end{eqnarray*}

\textit{Step 2. }The \textit{g}eneral case is obtained by approximation.

Suppose that $u\in \mathcal{H}^{\mathbf{s}}\left( 
\mathbb{R}
^{n}\right) $. Let $\psi \in \mathcal{C}_{0}^{\infty }\left( 
\mathbb{R}
^{n}\right) $ be such that $\psi =1$ on $B\left( 0,1\right) $. Then $\psi
^{\varepsilon }u\rightarrow u$ in $\mathcal{H}^{\mathbf{s}}\left( 
\mathbb{R}
^{n}\right) $ where $\psi ^{\varepsilon }\left( x\right) =\psi \left(
\varepsilon x\right) $, $0<\varepsilon \leq 1$, $x\in 
\mathbb{R}
^{n}$. Also we have 
\begin{equation*}
\left\Vert \psi ^{\varepsilon }u\right\Vert _{\mathcal{H}^{\mathbf{s}}}\leq
C\left( s,n,\psi \right) \left\Vert u\right\Vert _{\mathcal{H}^{\mathbf{s}%
}},\quad 0<\varepsilon \leq 1,
\end{equation*}%
where 
\begin{eqnarray*}
C\left( s,n,\psi \right) &=&\left( 2\pi \right) ^{-n}2^{\left\vert \mathbf{s}%
\right\vert _{1}/2}\sup_{0<\varepsilon \leq 1}\left( \int \left\langle \eta
\right\rangle ^{\left\vert \mathbf{s}\right\vert _{1}}\varepsilon
^{-n}\left\vert \widehat{\psi }\left( \eta /\varepsilon \right) \right\vert 
\mathtt{d}\eta \right) \\
&=&\left( 2\pi \right) ^{-n}2^{\left\vert \mathbf{s}\right\vert
_{1}/2}\sup_{0<\varepsilon \leq 1}\left( \int \left\langle \varepsilon \eta
\right\rangle ^{\left\vert \mathbf{s}\right\vert _{1}}\left\vert \widehat{%
\psi }\left( \eta \right) \right\vert \mathtt{d}\eta \right) \\
&\leq &\left( 2\pi \right) ^{-n}2^{\left\vert \mathbf{s}\right\vert
_{1}/2}\left( \int \left\langle \eta \right\rangle ^{\left\vert \mathbf{s}%
\right\vert _{1}}\left\vert \widehat{\psi }\left( \eta \right) \right\vert 
\mathtt{d}\eta \right) .
\end{eqnarray*}%
Let $m\in 
\mathbb{N}
,m\geq 1$. Then there is $\varepsilon _{m}$ such that for any $\varepsilon
\in \left( 0,\varepsilon _{m}\right] $ we have 
\begin{equation*}
\psi ^{\varepsilon }u=\sum_{\left\vert \gamma \right\vert \leq m}\tau
_{\gamma }u_{\gamma }+\sum_{finite}\tau _{\gamma }\left( \left( \tau
_{-\gamma }\psi ^{\varepsilon }\right) u_{\gamma }\right) .
\end{equation*}%
By the first part we get that 
\begin{equation*}
\sum_{\left\vert \gamma \right\vert \leq m}\left\Vert u_{\gamma }\right\Vert
_{\mathcal{H}^{\mathbf{s}}}^{2}\leq C_{\mathbf{s},n,\varphi }^{2}\left\Vert
\psi ^{\varepsilon }u\right\Vert _{\mathcal{H}^{\mathbf{s}}}^{2}\leq C_{%
\mathbf{s},n,\varphi }^{2}C\left( s,n,\psi \right) ^{2}\left\Vert
u\right\Vert _{\mathcal{H}^{\mathbf{s}}}^{2}.
\end{equation*}%
Since $m$ is arbitrary, it follows that $\sum_{\gamma \in 
\mathbb{Z}
^{n}}\left\Vert u_{\gamma }\right\Vert _{\mathcal{H}^{\mathbf{s}%
}}^{2}<\infty $. Further from 
\begin{equation*}
\sum_{\left\vert \gamma \right\vert \leq m}\left\Vert u_{\gamma }\right\Vert
_{\mathcal{H}^{\mathbf{s}}}^{2}\leq C_{\mathbf{s},n,\varphi }^{2}\left\Vert
\psi ^{\varepsilon }u\right\Vert _{\mathcal{H}^{\mathbf{s}}}^{2},\quad
0<\varepsilon \leq \varepsilon _{m},
\end{equation*}%
we obtain that%
\begin{equation*}
\sum_{\left\vert \gamma \right\vert \leq m}\left\Vert u_{\gamma }\right\Vert
_{\mathcal{H}^{\mathbf{s}}}^{2}\leq C_{\mathbf{s},n,\varphi }^{2}\left\Vert
u\right\Vert _{\mathcal{H}^{\mathbf{s}}}^{2},\quad \forall m\in 
\mathbb{N}
.
\end{equation*}%
Hence 
\begin{equation*}
\sum_{\gamma \in 
\mathbb{Z}
^{n}}\left\Vert u_{\gamma }\right\Vert _{\mathcal{H}^{\mathbf{s}}}^{2}\leq
C_{\mathbf{s},n,\varphi }^{2}\left\Vert u\right\Vert _{\mathcal{H}^{\mathbf{s%
}}}^{2}.
\end{equation*}

Now suppose that $\sum_{\gamma \in 
\mathbb{Z}
^{n}}\left\Vert u_{\gamma }\right\Vert _{\mathcal{H}^{\mathbf{s}%
}}^{2}<\infty $. For $m\in 
\mathbb{N}
$, $m\geq 1$ we put $u\left( m\right) =\sum_{\left\vert \gamma \right\vert
\leq m}\tau _{\gamma }u_{\gamma }$. Then 
\begin{equation*}
\left\Vert u\left( m+p\right) -u\left( m\right) \right\Vert _{\mathcal{H}^{%
\mathbf{s}}}^{2}\leq C_{\mathbf{s},n,\varphi }^{2}\sum_{m\leq \left\vert
\gamma \right\vert \leq m+p}\left\Vert u_{\gamma }\right\Vert _{\mathcal{H}^{%
\mathbf{s}}}^{2}
\end{equation*}%
It follows that $\left\{ u\left( m\right) \right\} _{m\geq 1}$ is a Cauchy
sequence in $\mathcal{H}^{\mathbf{s}}\left( 
\mathbb{R}
^{n}\right) $. Let $v\in \mathcal{H}^{\mathbf{s}}\left( 
\mathbb{R}
^{n}\right) $ be such that $u\left( m\right) \rightarrow v$ in $\mathcal{H}^{%
\mathbf{s}}\left( 
\mathbb{R}
^{n}\right) $. Since $u\left( m\right) \rightarrow u$ in $\mathcal{D}%
^{\prime }\left( 
\mathbb{R}
^{n}\right) $, it follows that $u=v$. Hence $u\left( m\right) \rightarrow u$
in $\mathcal{H}^{\mathbf{s}}\left( 
\mathbb{R}
^{n}\right) $. Since we have 
\begin{equation*}
\left\Vert u\left( m\right) \right\Vert _{\mathcal{H}^{\mathbf{s}}}^{2}\leq
C_{\mathbf{s},n,\varphi }^{2}\sum_{\left\vert \gamma \right\vert \leq
m}\left\Vert u_{\gamma }\right\Vert _{\mathcal{H}^{\mathbf{s}}}^{2}\leq C_{%
\mathbf{s},n,\varphi }^{2}\sum_{\gamma \in 
\mathbb{Z}
^{n}}\left\Vert u_{\gamma }\right\Vert _{\mathcal{H}^{\mathbf{s}}}^{2},\quad
\forall m\in 
\mathbb{N}
.
\end{equation*}%
we obtain that 
\begin{equation*}
\left\Vert u\right\Vert _{\mathcal{H}^{\mathbf{s}}}^{2}\leq C_{\mathbf{s}%
,n,\varphi }^{2}\sum_{\gamma \in 
\mathbb{Z}
^{n}}\left\Vert u_{\gamma }\right\Vert _{\mathcal{H}^{\mathbf{s}}}^{2}.
\end{equation*}
\end{proof}

To use the previous result we need a convenient partition of unity. Let $%
N\in 
\mathbb{N}
$ and $\left\{ x_{1},...,x_{N}\right\} \subset 
\mathbb{R}
^{n}$ be such that 
\begin{equation*}
\left[ 0,1\right] ^{n}\subset \left( x_{1}+\left[ \frac{1}{3},\frac{2}{3}%
\right] ^{n}\right) \cup ...\cup \left( x_{N}+\left[ \frac{1}{3},\frac{2}{3}%
\right] ^{n}\right)
\end{equation*}%
Let $\widetilde{h}\in \mathcal{C}_{0}^{\infty }\left( 
\mathbb{R}
^{n}\right) $, $\widetilde{h}\geq 0$, be such that $\widetilde{h}=1$ on $%
\left[ \frac{1}{3},\frac{2}{3}\right] ^{n}$ and \texttt{supp}$\widetilde{h}%
\subset \left[ \frac{1}{4},\frac{3}{4}\right] ^{n}$. Then

\begin{enumerate}
\item[$\left( \mathtt{a}\right) $] $\widetilde{H}=\sum_{i=1}^{N}\sum_{\gamma
\in 
\mathbb{Z}
^{n}}\tau _{\gamma +x_{i}}\widetilde{h}\in \mathcal{BC}^{\infty }\left( 
\mathbb{R}
^{n}\right) $ is $%
\mathbb{Z}
^{n}$-periodic and $\widetilde{H}\geq 1$.

\item[$\left( \mathtt{b}\right) $] $h_{i}=\frac{\tau _{x_{i}}\widetilde{h}}{%
\widetilde{H}}\in \mathcal{C}_{0}^{\infty }\left( 
\mathbb{R}
^{n}\right) $, $h_{i}\geq 0$, \texttt{supp}$h_{i}\subset x_{i}+\left[ \frac{1%
}{4},\frac{3}{4}\right] ^{n}=K_{i}$, $\left( K_{i}-K_{i}\right) \cap 
\mathbb{Z}
^{n}=\left\{ 0\right\} $, $i=1,...,N$.

\item[$\left( \mathtt{c}\right) $] $\chi _{i}=\sum_{\gamma \in 
\mathbb{Z}
^{n}}\tau _{\gamma }h_{i}\in \mathcal{BC}^{\infty }\left( 
\mathbb{R}
^{n}\right) $ is $%
\mathbb{Z}
^{n}$-periodic, $i=1,...,N$ and $\sum_{i=1}^{N}\chi _{i}=1.$

\item[$\left( \mathtt{d}\right) $] $h=\sum_{i=1}^{N}h_{i}\in \mathcal{C}%
_{0}^{\infty }\left( 
\mathbb{R}
^{n}\right) $, $h\geq 0$, $\sum_{\gamma \in 
\mathbb{Z}
^{n}}\tau _{\gamma }h=$ $1$.
\end{enumerate}

\noindent A first consequence of previous results is the next proposition.

\begin{proposition}
\label{ks4}Let $\mathbf{s}\in 
\mathbb{R}
^{j}$ and $m_{\mathbf{s}}=\left[ \left\vert \mathbf{s}\right\vert _{1}+\frac{%
n+1}{2}\right] +1$. Then 
\begin{equation*}
\mathcal{BC}^{m_{s}}\left( 
\mathbb{R}
^{n}\right) \cdot \mathcal{H}^{\mathbf{s}}\left( 
\mathbb{R}
^{n}\right) \subset \mathcal{H}^{\mathbf{s}}\left( 
\mathbb{R}
^{n}\right) .
\end{equation*}
\end{proposition}

\begin{proof}
Let $u\in \mathcal{H}^{\mathbf{s}}\left( 
\mathbb{R}
^{n}\right) $. We use the partition of unity constructed above to obtain a
decomposition of $u$ satisfying the conditions of Lemma \ref{ks2}. Using
Proposition \ref{ks3} $\left( \mathtt{c}\right) $, it follows that $\chi
_{i}u\in \mathcal{H}^{\mathbf{s}}\left( 
\mathbb{R}
^{n}\right) $, $i=1,...,N$. We have%
\begin{equation*}
u=\sum_{i=1}^{N}\chi _{i}u
\end{equation*}%
with $\chi _{i}u\in \mathcal{H}^{\mathbf{s}}\left( 
\mathbb{R}
^{n}\right) $,%
\begin{gather*}
\chi _{i}u=\sum_{\gamma \in 
\mathbb{Z}
^{n}}\tau _{\gamma }\left( h_{i}\tau _{-\gamma }u\right) ,\quad h_{i}\tau
_{-\gamma }u\in \mathcal{H}^{\mathbf{s}}\left( 
\mathbb{R}
^{n}\right) \cap \mathcal{D}_{K_{i}}^{\prime }\left( 
\mathbb{R}
^{n}\right) , \\
\left( K_{i}-K_{i}\right) \cap 
\mathbb{Z}
^{n}=\left\{ 0\right\} ,\quad i=1,...,N.
\end{gather*}%
So we can assume that $u\in \mathcal{H}^{\mathbf{s}}\left( 
\mathbb{R}
^{n}\right) $ is of the form described in Lemma \ref{ks2}.

Let $\psi \in \mathcal{BC}^{m_{s}}\left( 
\mathbb{R}
^{n}\right) $. Then 
\begin{equation*}
\psi u=\sum_{\gamma \in 
\mathbb{Z}
^{n}}\psi \tau _{\gamma }u_{\gamma }=\sum_{\gamma \in 
\mathbb{Z}
^{n}}\tau _{\gamma }\left( \psi _{\gamma }u_{\gamma }\right)
\end{equation*}%
with $\psi _{\gamma }=\varphi \left( \tau _{-\gamma }\psi \right) $, where $%
\varphi \in \mathcal{C}_{0}^{\infty }\left( 
\mathbb{R}
^{n}\right) $ is the function considered in the proof of Lemma \ref{ks2}. We
apply Lemma \ref{ks2} and Proposition \ref{ks3} $\left( \mathtt{b}\right) $
to obtain 
\begin{equation*}
\left\Vert \psi u\right\Vert _{\mathcal{H}^{\mathbf{s}}}^{2}\leq C_{\mathbf{s%
},n,\varphi }^{2}\sum_{\gamma \in 
\mathbb{Z}
^{n}}\left\Vert \psi _{\gamma }u_{\gamma }\right\Vert _{\mathcal{H}^{\mathbf{%
s}}}^{2}
\end{equation*}%
and%
\begin{eqnarray*}
\left\Vert \psi _{\gamma }u_{\gamma }\right\Vert _{\mathcal{H}^{\mathbf{s}}}
&\leq &Cst\left( \sum_{\left\vert \alpha \right\vert \leq m_{s}}\left\Vert
\partial ^{\alpha }\left( \varphi \left( \tau _{-\gamma }\psi \right)
\right) \right\Vert _{L^{2}}\right) \left\Vert u_{\gamma }\right\Vert _{%
\mathcal{H}^{\mathbf{s}}} \\
&\leq &Cst\left\Vert \varphi \right\Vert _{H^{m_{s}}}\left\Vert \psi
\right\Vert _{\mathcal{BC}^{m_{s}}}\left\Vert u_{\gamma }\right\Vert _{%
\mathcal{H}^{\mathbf{s}}},\quad \gamma \in 
\mathbb{Z}
^{n}.
\end{eqnarray*}%
Hence another application of Lemma \ref{ks2} gives 
\begin{eqnarray*}
\left\Vert \psi u\right\Vert _{\mathcal{H}^{\mathbf{s}}}^{2} &\leq
&Cst\left\Vert \varphi \right\Vert _{H^{m_{s}}}^{2}\left\Vert \psi
\right\Vert _{\mathcal{BC}^{m_{s}}}^{2}\sum_{\gamma \in 
\mathbb{Z}
^{n}}\left\Vert u_{\gamma }\right\Vert _{\mathcal{H}^{\mathbf{s}}}^{2} \\
&\leq &Cst\left\Vert \varphi \right\Vert _{H^{m_{s}}}^{2}\left\Vert \psi
\right\Vert _{\mathcal{BC}^{m_{s}}}^{2}\left\Vert u\right\Vert _{\mathcal{H}%
^{\mathbf{s}}}^{2}.
\end{eqnarray*}
\end{proof}

\begin{corollary}
Let $\mathbf{s}\in 
\mathbb{R}
^{j}$. Then 
\begin{equation*}
\mathcal{BC}^{\infty }\left( 
\mathbb{R}
^{n}\right) \cdot \mathcal{H}^{\mathbf{s}}\left( 
\mathbb{R}
^{n}\right) \subset \mathcal{H}^{\mathbf{s}}\left( 
\mathbb{R}
^{n}\right) .
\end{equation*}
\end{corollary}

\begin{lemma}
Let $\lambda _{1}$, $\lambda _{2}\geq 0$, $\lambda _{1}+\lambda _{2}>n/2$.
Then 
\begin{equation*}
\left\langle \cdot \right\rangle _{%
\mathbb{R}
^{n}}^{-2\lambda _{1}}\ast \left\langle \cdot \right\rangle _{%
\mathbb{R}
^{n}}^{-2\lambda _{2}}\leq \left\Vert \left\langle \cdot \right\rangle _{%
\mathbb{R}
^{n}}^{-2\left( \lambda _{1}+\lambda _{2}\right) }\right\Vert _{L^{1}}
\end{equation*}
\end{lemma}

\begin{proof}
The case $\lambda _{1}\cdot \lambda _{2}=0$ is trivial. Thus we may assume
that $\lambda _{1}$, $\lambda _{2}>0$, $\lambda _{1}+\lambda _{2}>n/2$. Then 
\begin{equation*}
\left\langle \cdot \right\rangle ^{-2\lambda _{j}}\in L^{p_{j}},\quad p_{j}=%
\frac{\lambda _{1}+\lambda _{2}}{\lambda _{j}}>1,\quad j=1,2.
\end{equation*}%
Since $\frac{1}{p_{1}}+\frac{1}{p_{2}}=1$, by using H\"{o}lder's inequality
we get%
\begin{equation*}
\left\langle \cdot \right\rangle _{%
\mathbb{R}
^{n}}^{-2\lambda _{1}}\ast \left\langle \cdot \right\rangle _{%
\mathbb{R}
^{n}}^{-2\lambda _{2}}\leq \left\Vert \left\langle \cdot \right\rangle _{%
\mathbb{R}
^{n}}^{-2\lambda _{1}}\right\Vert _{L^{p_{1}}}\left\Vert \left\langle \cdot
\right\rangle _{%
\mathbb{R}
^{n}}^{-2\lambda _{2}}\right\Vert _{L^{p_{2}}}
\end{equation*}%
with 
\begin{eqnarray*}
\left\Vert \left\langle \cdot \right\rangle ^{-2\lambda _{j}}\right\Vert
_{L^{p_{j}}}^{p_{j}} &=&\int \left[ \left( 1+\left\vert \xi \right\vert
^{2}\right) ^{-\lambda _{j}}\right] ^{\frac{\lambda _{1}+\lambda _{2}}{%
\lambda _{j}}}\mathtt{d}\xi \\
&=&\int \left( 1+\left\vert \xi \right\vert ^{2}\right) ^{-\lambda
_{1}-\lambda _{2}}\mathtt{d}\xi \\
&=&\left\Vert \left\langle \cdot \right\rangle _{%
\mathbb{R}
^{n}}^{-2\left( \lambda _{1}+\lambda _{2}\right) }\right\Vert _{L^{1}},\quad
j=1,2.
\end{eqnarray*}%
Therefore, 
\begin{eqnarray*}
\left\langle \cdot \right\rangle _{%
\mathbb{R}
^{n}}^{-2\lambda _{1}}\ast \left\langle \cdot \right\rangle _{%
\mathbb{R}
^{n}}^{-2\lambda _{2}} &\leq &\left\Vert \left\langle \cdot \right\rangle _{%
\mathbb{R}
^{n}}^{-2\lambda _{1}}\right\Vert _{L^{p_{1}}}\left\Vert \left\langle \cdot
\right\rangle _{%
\mathbb{R}
^{n}}^{-2\lambda _{2}}\right\Vert _{L^{p_{2}}} \\
&=&\left\Vert \left\langle \cdot \right\rangle _{%
\mathbb{R}
^{n}}^{-2\left( \lambda _{1}+\lambda _{2}\right) }\right\Vert _{L^{1}}^{%
\frac{1}{p_{1}}+\frac{1}{p_{2}}} \\
&=&\left\Vert \left\langle \cdot \right\rangle _{%
\mathbb{R}
^{n}}^{-2\left( \lambda _{1}+\lambda _{2}\right) }\right\Vert _{L^{1}}.
\end{eqnarray*}
\end{proof}

\begin{lemma}
Let $s$, $t\in 
\mathbb{R}
$, $s+t>n/2$. For $\varepsilon \in \left( 0,s+t-n/2\right) $ we put $\sigma
\left( \varepsilon \right) =\min \left\{ s,t,s+t-n/2-\varepsilon \right\} $.
Then 
\begin{equation*}
\left\langle \cdot \right\rangle _{%
\mathbb{R}
^{n}}^{-2s}\ast \left\langle \cdot \right\rangle _{%
\mathbb{R}
^{n}}^{-2t}\leq C\left( s,t,\varepsilon ,n\right) \left\langle \cdot
\right\rangle _{%
\mathbb{R}
^{n}}^{-2\sigma \left( \varepsilon \right) }.
\end{equation*}%
where 
\begin{equation*}
C\left( s,t,\varepsilon ,n\right) =\left\{ 
\begin{array}{ccc}
2^{2\sigma \left( \varepsilon \right) +1}\left\Vert \left\langle \cdot
\right\rangle _{%
\mathbb{R}
^{n}}^{-2\left( s+t-\sigma \left( \varepsilon \right) \right) }\right\Vert
_{L^{1}} & \text{\textit{if}} & s,t\geq 0, \\ 
2^{\left\vert \sigma \left( \varepsilon \right) \right\vert }\left\Vert
\left\langle \cdot \right\rangle _{%
\mathbb{R}
^{n}}^{-2\left( s+t\right) }\right\Vert _{L^{1}} & \text{\textit{if}} & s<0%
\text{ \textit{or }}t<0.%
\end{array}%
\right.
\end{equation*}
\end{lemma}

\begin{proof}
Let us write $\sigma $ for $\sigma \left( \varepsilon \right) $.

\textit{Step 1. }The case $s,t\geq 0$. We have 
\begin{equation*}
\left\langle \cdot \right\rangle _{%
\mathbb{R}
^{n}}^{-2s}\ast \left\langle \cdot \right\rangle _{%
\mathbb{R}
^{n}}^{-2t}\left( \xi \right) =\int_{\left\vert \eta -\xi \right\vert \geq 
\frac{1}{2}\left\vert \xi \right\vert }\left\langle \xi -\eta \right\rangle
_{%
\mathbb{R}
^{n}}^{-2s}\left\langle \eta \right\rangle _{%
\mathbb{R}
^{n}}^{-2t}\mathtt{d}\eta +\int_{\left\vert \eta -\xi \right\vert \leq \frac{%
1}{2}\left\vert \xi \right\vert }\left\langle \xi -\eta \right\rangle _{%
\mathbb{R}
^{n}}^{-2s}\left\langle \eta \right\rangle _{%
\mathbb{R}
^{n}}^{-2t}\mathtt{d}\eta
\end{equation*}

$\left( \mathtt{a}\right) $ If $\left\vert \eta -\xi \right\vert \geq \frac{1%
}{2}\left\vert \xi \right\vert $, then 
\begin{equation*}
\frac{1}{1+\left\vert \xi -\eta \right\vert ^{2}}\leq \frac{4}{1+\left\vert
\xi \right\vert ^{2}}\Leftrightarrow \left\langle \xi -\eta \right\rangle _{%
\mathbb{R}
^{n}}^{-1}\leq 2\left\langle \xi \right\rangle _{%
\mathbb{R}
^{n}}^{-1}
\end{equation*}%
and 
\begin{eqnarray*}
\left\langle \xi -\eta \right\rangle _{%
\mathbb{R}
^{n}}^{-2s} &=&\left\langle \xi -\eta \right\rangle _{%
\mathbb{R}
^{n}}^{-2\sigma }\cdot \left\langle \xi -\eta \right\rangle _{%
\mathbb{R}
^{n}}^{-2\left( s-\sigma \right) } \\
&\leq &2^{2\sigma }\left\langle \xi \right\rangle _{%
\mathbb{R}
^{n}}^{-2\sigma }\left\langle \xi -\eta \right\rangle _{%
\mathbb{R}
^{n}}^{-2\left( s-\sigma \right) }
\end{eqnarray*}%
Since $s-\sigma +t=s+t-n/2-\varepsilon -\sigma +n/2+\varepsilon \geq
n/2+\varepsilon >n/2$, the previous lemma allows to evaluate the integral on
the domain $\left\vert \eta -\xi \right\vert \geq \frac{1}{2}\left\vert \xi
\right\vert $%
\begin{eqnarray*}
\int_{\left\vert \eta -\xi \right\vert \geq \frac{1}{2}\left\vert \xi
\right\vert }\left\langle \xi -\eta \right\rangle _{%
\mathbb{R}
^{n}}^{-2s}\left\langle \eta \right\rangle _{%
\mathbb{R}
^{n}}^{-2t}\mathtt{d}\eta &\leq &2^{2\sigma }\left\langle \xi \right\rangle
_{%
\mathbb{R}
^{n}}^{-2\sigma }\int_{\left\vert \eta -\xi \right\vert \geq \frac{1}{2}%
\left\vert \xi \right\vert }\left\langle \xi -\eta \right\rangle _{%
\mathbb{R}
^{n}}^{-2\left( s-\sigma \right) }\left\langle \eta \right\rangle _{%
\mathbb{R}
^{n}}^{-2t}\mathtt{d}\eta \\
&\leq &2^{2\sigma }\left\langle \xi \right\rangle _{%
\mathbb{R}
^{n}}^{-2\sigma }\left( \left\langle \cdot \right\rangle _{%
\mathbb{R}
^{n}}^{-2\left( s-\sigma \right) }\ast \left\langle \cdot \right\rangle _{%
\mathbb{R}
^{n}}^{-2t}\right) \left( \xi \right) \\
&\leq &2^{2\sigma }\left\Vert \left\langle \cdot \right\rangle _{%
\mathbb{R}
^{n}}^{-2\left( s+t-\sigma \right) }\right\Vert _{L^{1}}\left\langle \xi
\right\rangle _{%
\mathbb{R}
^{n}}^{-2\sigma }
\end{eqnarray*}

$\left( \mathtt{b}\right) $ If $\left\vert \eta -\xi \right\vert \leq \frac{1%
}{2}\left\vert \xi \right\vert $, then $\left\vert \eta \right\vert \geq
\left\vert \xi \right\vert -\left\vert \eta -\xi \right\vert \geq \frac{1}{2}%
\left\vert \xi \right\vert $. We can therefore use $\left( \mathtt{a}\right) 
$ to evaluate the integral on the domain $\left\vert \eta -\xi \right\vert
\leq \frac{1}{2}\left\vert \xi \right\vert $. It follows that 
\begin{eqnarray*}
\int_{\left\vert \eta -\xi \right\vert \leq \frac{1}{2}\left\vert \xi
\right\vert }\left\langle \xi -\eta \right\rangle _{%
\mathbb{R}
^{n}}^{-2s}\left\langle \eta \right\rangle _{%
\mathbb{R}
^{n}}^{-2t}\mathtt{d}\eta &\leq &\int_{\left\vert \eta \right\vert \geq 
\frac{1}{2}\left\vert \xi \right\vert }\left\langle \xi -\eta \right\rangle
_{%
\mathbb{R}
^{n}}^{-2s}\left\langle \eta \right\rangle _{%
\mathbb{R}
^{n}}^{-2t}\mathtt{d}\eta \\
&=&\int_{\left\vert \zeta -\xi \right\vert \geq \frac{1}{2}\left\vert \xi
\right\vert }\left\langle \zeta \right\rangle _{%
\mathbb{R}
^{n}}^{-2s}\left\langle \xi -\zeta \right\rangle _{%
\mathbb{R}
^{n}}^{-2t}\mathtt{d}\zeta \\
&\leq &2^{2\sigma }\left\Vert \left\langle \cdot \right\rangle _{%
\mathbb{R}
^{n}}^{-2\left( s+t-\sigma \right) }\right\Vert _{L^{1}}\left\langle \xi
\right\rangle _{%
\mathbb{R}
^{n}}^{-2\sigma }
\end{eqnarray*}

$\left( \mathtt{c}\right) $ From $\left( \mathtt{a}\right) $ and $\left( 
\mathtt{b}\right) $ we obtain%
\begin{equation*}
\left\langle \cdot \right\rangle _{%
\mathbb{R}
^{n}}^{-2s}\ast \left\langle \cdot \right\rangle _{%
\mathbb{R}
^{n}}^{-2t}\leq 2^{2\sigma +1}\left\Vert \left\langle \cdot \right\rangle _{%
\mathbb{R}
^{n}}^{-2\left( s+t-\sigma \right) }\right\Vert _{L^{1}}\left\langle \cdot
\right\rangle _{%
\mathbb{R}
^{n}}^{-2\sigma }.
\end{equation*}

\textit{Step 2.} Next we consider the case $s<0$ or $t<0$. If $s<0$ and $%
s+t>n/2$, then $\sigma =s$. In this case we use Peetre's inequality to
obtain:%
\begin{eqnarray*}
\left\langle \cdot \right\rangle _{%
\mathbb{R}
^{n}}^{-2s}\ast \left\langle \cdot \right\rangle _{%
\mathbb{R}
^{n}}^{-2t}\left( \xi \right) &=&\int \left\langle \xi -\eta \right\rangle _{%
\mathbb{R}
^{n}}^{-2s}\left\langle \eta \right\rangle _{%
\mathbb{R}
^{n}}^{-2t}\mathtt{d}\eta \\
&\leq &2^{\left\vert s\right\vert }\int \left\langle \xi \right\rangle _{%
\mathbb{R}
^{n}}^{-2s}\left\langle \eta \right\rangle _{%
\mathbb{R}
^{n}}^{-2\left( s+t\right) }\mathtt{d}\eta \\
&=&2^{\left\vert \sigma \right\vert }\left\Vert \left\langle \cdot
\right\rangle _{%
\mathbb{R}
^{n}}^{-2\left( s+t\right) }\right\Vert _{L^{1}}\left\langle \xi
\right\rangle _{%
\mathbb{R}
^{n}}^{-2\sigma }
\end{eqnarray*}%
The case $t<0$ can be treated similarly.
\end{proof}

Since $\left\langle \left\langle \cdot \right\rangle \right\rangle ^{\mathbf{%
s}}=\left\langle \cdot \right\rangle _{%
\mathbb{R}
^{n_{1}}}^{s_{1}}\otimes ...\otimes \left\langle \cdot \right\rangle _{%
\mathbb{R}
^{n_{j}}}^{s_{j}}$, $\mathbf{s}=\left( s_{1},...,s_{j}\right) \in 
\mathbb{R}
^{j}$ we obtain

\begin{corollary}
Let $\mathbf{s}$, $\mathbf{t}$, $\mathbf{\varepsilon }$, \textbf{$\sigma $}$%
\left( \mathbf{\varepsilon }\right) \in 
\mathbb{R}
^{j}$ such that, $s_{l}+t_{l}>n_{l}/2$, $0<\varepsilon
_{l}<s_{l}+t_{l}-n_{l}/2$, \textbf{$\sigma $}$_{l}\left( \mathbf{\varepsilon 
}\right) =$\textbf{$\sigma $}$_{l}\left( \varepsilon _{l}\right) =\min
\left\{ s_{l},t_{l},s_{l}+t_{l}-n_{l}/2-\varepsilon _{l}\right\} $ for any $%
l\in \left\{ 1,...,j\right\} $. Then there is $C\left( \mathbf{s},\mathbf{t},%
\mathbf{\varepsilon },n\right) >0$ such that%
\begin{equation*}
\left\langle \left\langle \cdot \right\rangle \right\rangle ^{-2\mathbf{s}%
}\ast \left\langle \left\langle \cdot \right\rangle \right\rangle ^{-2%
\mathbf{t}}\leq C\left( \mathbf{s},\mathbf{t},\mathbf{\varepsilon },n\right)
\left\langle \left\langle \cdot \right\rangle \right\rangle ^{-2\mathbf{%
\sigma }\left( \mathbf{\varepsilon }\right) }.
\end{equation*}
\end{corollary}

\begin{proposition}
\label{ks7}Let $\mathbf{s}$, $\mathbf{t}$, $\mathbf{\varepsilon }$, \textbf{$%
\sigma $}$\left( \mathbf{\varepsilon }\right) \in 
\mathbb{R}
^{j}$ such that, $s_{l}+t_{l}>n_{l}/2$, $0<\varepsilon
_{l}<s_{l}+t_{l}-n_{l}/2$, \textbf{$\sigma $}$_{l}\left( \mathbf{\varepsilon 
}\right) =$\textbf{$\sigma $}$_{l}\left( \varepsilon _{l}\right) =\min
\left\{ s_{l},t_{l},s_{l}+t_{l}-n_{l}/2-\varepsilon _{l}\right\} $ for any $%
l\in \left\{ 1,...,j\right\} $. Then 
\begin{equation*}
\mathcal{H}^{\mathbf{s}}\left( 
\mathbb{R}
^{n}\right) \cdot \mathcal{H}^{\mathbf{t}}\left( 
\mathbb{R}
^{n}\right) \subset \mathcal{H}^{\mathbf{\sigma }\left( \mathbf{\varepsilon }%
\right) }\left( 
\mathbb{R}
^{n}\right)
\end{equation*}
\end{proposition}

\begin{proof}
Let us write $\sigma $ for $\sigma \left( \varepsilon \right) $. Let $u,v\in 
\mathcal{S}\left( 
\mathbb{R}
^{n}\right) $. Then 
\begin{multline*}
\left\Vert u\cdot v\right\Vert _{\mathcal{H}^{\mathbf{\sigma }%
}}^{2}=\left\Vert \left\langle \left\langle \cdot \right\rangle
\right\rangle ^{\mathbf{\sigma }}\widehat{u\cdot v}\right\Vert
_{L^{2}}^{2}=\int \left\vert \left\langle \left\langle \xi \right\rangle
\right\rangle ^{\mathbf{\sigma }}\widehat{u\cdot v}\left( \xi \right)
\right\vert ^{2}\mathtt{d}\xi \\
=\left( 2\pi \right) ^{-2n}\int \left\vert \left\langle \left\langle \xi
\right\rangle \right\rangle ^{\mathbf{\sigma }}\widehat{u}\ast \widehat{v}%
\left( \xi \right) \right\vert ^{2}\mathtt{d}\xi
\end{multline*}%
By using Schwarz's inequality and the above corollary we can estimate the
integrand as follows%
\begin{multline*}
\left\vert \left\langle \left\langle \xi \right\rangle \right\rangle ^{%
\mathbf{\sigma }}\widehat{u}\ast \widehat{v}\left( \xi \right) \right\vert
^{2}\leq \left( \int \left\vert \left\langle \left\langle \eta \right\rangle
\right\rangle ^{\mathbf{s}}\widehat{u}\left( \eta \right) \right\vert
\left\vert \left\langle \left\langle \xi -\eta \right\rangle \right\rangle ^{%
\mathbf{t}}\widehat{v}\left( \xi -\eta \right) \right\vert \frac{%
\left\langle \left\langle \xi \right\rangle \right\rangle ^{\mathbf{\sigma }}%
}{\left\langle \left\langle \eta \right\rangle \right\rangle ^{\mathbf{s}%
}\left\langle \left\langle \xi -\eta \right\rangle \right\rangle ^{\mathbf{t}%
}}\mathtt{d}\eta \right) ^{2} \\
\leq \left( \int \left\vert \left\langle \left\langle \eta \right\rangle
\right\rangle ^{\mathbf{s}}\widehat{u}\left( \eta \right) \right\vert
^{2}\left\vert \left\langle \left\langle \xi -\eta \right\rangle
\right\rangle ^{\mathbf{t}}\widehat{v}\left( \xi -\eta \right) \right\vert
^{2}\mathtt{d}\eta \right) \left( \int \frac{\left\langle \left\langle \xi
\right\rangle \right\rangle ^{2\mathbf{\sigma }}}{\left\langle \left\langle
\eta \right\rangle \right\rangle ^{2\mathbf{s}}\left\langle \left\langle \xi
-\eta \right\rangle \right\rangle ^{2\mathbf{t}}}\mathtt{d}\eta \right) \\
\leq C\left( \mathbf{s},\mathbf{t},\mathbf{\varepsilon },n\right) \int
\left\vert \left\langle \left\langle \eta \right\rangle \right\rangle ^{%
\mathbf{s}}\widehat{u}\left( \eta \right) \right\vert ^{2}\left\vert
\left\langle \left\langle \xi -\eta \right\rangle \right\rangle ^{\mathbf{t}}%
\widehat{v}\left( \xi -\eta \right) \right\vert ^{2}\mathtt{d}\eta
\end{multline*}%
Hence 
\begin{eqnarray*}
\left\Vert u\cdot v\right\Vert _{\mathcal{H}^{\mathbf{\sigma }}}^{2} &\leq
&C^{\prime }\left( \mathbf{s},\mathbf{t},\mathbf{\varepsilon },n\right) \int
\left( \int \left\vert \left\langle \left\langle \eta \right\rangle
\right\rangle ^{\mathbf{s}}\widehat{u}\left( \eta \right) \right\vert
^{2}\left\vert \left\langle \left\langle \xi -\eta \right\rangle
\right\rangle ^{\mathbf{t}}\widehat{v}\left( \xi -\eta \right) \right\vert
^{2}\mathtt{d}\eta \right) \mathtt{d}\xi \\
&=&C^{\prime }\left( \mathbf{s},\mathbf{t},\mathbf{\varepsilon },n\right)
\left\Vert u\right\Vert _{\mathcal{H}^{\mathbf{s}}}^{2}\left\Vert
v\right\Vert _{\mathcal{H}^{\mathbf{t}}}^{2}
\end{eqnarray*}%
To conclude we use the fact that $\mathcal{S}\left( 
\mathbb{R}
^{n}\right) $ is dense in any $\mathcal{H}^{\mathbf{m}}\left( 
\mathbb{R}
^{n}\right) $.
\end{proof}

\begin{corollary}
Let $\mathbf{s}\in 
\mathbb{R}
^{j}$. If $s_{1}>n_{1}/2,...,s_{j}>n_{j}/2$, then $\mathcal{H}^{\mathbf{s}%
}\left( 
\mathbb{R}
^{n}\right) $ is a Banach algebra.
\end{corollary}

\section{Kato-Sobolev spaces $\mathcal{K}_{p}^{\mathbf{s}}\left( 
\mathbb{R}
^{n}\right) $}

We begin by proving some results that will be useful later. Let $\varphi
,\psi \in \mathcal{C}_{0}^{\infty }\left( \mathbb{R}^{n}\right) $ (or $%
\varphi ,\psi \in \mathcal{S}\left( \mathbb{R}^{n}\right) $). Then the maps 
\begin{eqnarray*}
\mathbb{R}^{n}\times \mathbb{R}^{n} &\ni &\left( x,y\right) \overset{f}{%
\longrightarrow }\varphi \left( x\right) \psi \left( x-y\right) =\left(
\varphi \tau _{y}\psi \right) \left( x\right) \in 
\mathbb{C}
, \\
\mathbb{R}^{n}\times \mathbb{R}^{n} &\ni &\left( x,y\right) \overset{g}{%
\longrightarrow }\varphi \left( y\right) \psi \left( x-y\right) =\varphi
\left( y\right) \left( \tau _{y}\psi \right) \left( x\right) \in 
\mathbb{C}
,
\end{eqnarray*}%
are in $\mathcal{C}_{0}^{\infty }\left( \mathbb{R}^{n}\times \mathbb{R}%
^{n}\right) $ (respectively in $\mathcal{S}\left( \mathbb{R}^{n}\times 
\mathbb{R}^{n}\right) $). To see this we note that%
\begin{equation*}
f=\left( \varphi \otimes \psi \right) \circ T,\quad g=\left( \varphi \otimes
\psi \right) \circ S
\end{equation*}%
where 
\begin{eqnarray*}
T &:&\mathbb{R}^{n}\times \mathbb{R}^{n}\rightarrow \mathbb{R}^{n}\times 
\mathbb{R}^{n},\quad T\left( x,y\right) =\left( x,x-y\right) ,\quad T\equiv
\left( 
\begin{array}{cc}
\mathtt{I} & 0 \\ 
\mathtt{I} & -\mathtt{I}%
\end{array}%
\right) , \\
S &:&\mathbb{R}^{n}\times \mathbb{R}^{n}\rightarrow \mathbb{R}^{n}\times 
\mathbb{R}^{n},\quad S\left( x,y\right) =\left( y,x-y\right) ,\quad S\equiv
\left( 
\begin{array}{cc}
0 & \mathtt{I} \\ 
\mathtt{I} & -\mathtt{I}%
\end{array}%
\right) .
\end{eqnarray*}

Let $u\in \mathcal{D}^{\prime }\left( \mathbb{R}^{n}\right) $ (or $u\in 
\mathcal{S}^{\prime }\left( \mathbb{R}^{n}\right) $). Then using Fubini
theorem for distributions we get 
\begin{eqnarray*}
\left\langle u\otimes 1,f\right\rangle &=&\left\langle \left( u\otimes
1\right) \left( x,y\right) ,\varphi \left( x\right) \psi \left( x-y\right)
\right\rangle \\
&=&\left\langle u\left( x\right) ,\left\langle 1\left( y\right) ,\varphi
\left( x\right) \psi \left( x-y\right) \right\rangle \right\rangle \\
&=&\left\langle u\left( x\right) ,\varphi \left( x\right) \left\langle
1\left( y\right) ,\psi \left( x-y\right) \right\rangle \right\rangle \\
&=&\left\langle u\left( x\right) ,\varphi \left( x\right) \int \psi \left(
x-y\right) \mathtt{d}y\right\rangle \\
&=&\left( \int \psi \right) \left\langle u,\varphi \right\rangle
\end{eqnarray*}%
and 
\begin{eqnarray*}
\left\langle u\otimes 1,f\right\rangle &=&\left\langle 1\left( y\right)
,\left\langle u\left( x\right) ,\varphi \left( x\right) \psi \left(
x-y\right) \right\rangle \right\rangle \\
&=&\int \left\langle u,\varphi \tau _{y}\psi \right\rangle \mathtt{d}y.
\end{eqnarray*}%
It follows that%
\begin{equation*}
\left( \int \psi \right) \left\langle u,\varphi \right\rangle =\int
\left\langle u,\varphi \tau _{y}\psi \right\rangle \mathtt{d}y
\end{equation*}%
valid for

\begin{itemize}
\item[\texttt{(i)}] $u\in \mathcal{D}^{\prime }\left( \mathbb{R}^{n}\right) $%
, $\varphi ,\psi \in \mathcal{C}_{0}^{\infty }\left( \mathbb{R}^{n}\right) $;

\item[\texttt{(ii)}] $u\in \mathcal{S}^{\prime }\left( \mathbb{R}^{n}\right)
,$ $\varphi ,\psi \in \mathcal{S}\left( \mathbb{R}^{n}\right) $.
\end{itemize}

We also have%
\begin{eqnarray*}
\left\langle u\otimes 1,g\right\rangle &=&\left\langle \left( u\otimes
1\right) \left( x,y\right) ,\varphi \left( y\right) \psi \left( x-y\right)
\right\rangle \\
&=&\left\langle u\left( x\right) ,\left\langle 1\left( y\right) ,\varphi
\left( y\right) \psi \left( x-y\right) \right\rangle \right\rangle \\
&=&\left\langle u\left( x\right) ,\left( \varphi \ast \psi \right) \left(
x\right) \right\rangle \\
&=&\left\langle u,\varphi \ast \psi \right\rangle
\end{eqnarray*}%
and 
\begin{eqnarray*}
\left\langle u\otimes 1,g\right\rangle &=&\left\langle 1\left( y\right)
,\left\langle u\left( x\right) ,\varphi \left( y\right) \psi \left(
x-y\right) \right\rangle \right\rangle \\
&=&\int \varphi \left( y\right) \left\langle u,\tau _{y}\psi \right\rangle 
\mathtt{d}y.
\end{eqnarray*}%
Hence%
\begin{equation*}
\left\langle u,\varphi \ast \psi \right\rangle =\int \varphi \left(
y\right) \left\langle u,\tau _{y}\psi \right\rangle \mathtt{d}y
\end{equation*}%
true for

\begin{itemize}
\item[\texttt{(i)}] $u\in \mathcal{D}^{\prime }\left( \mathbb{R}^{n}\right) $%
, $\varphi ,\psi \in \mathcal{C}_{0}^{\infty }\left( \mathbb{R}^{n}\right) $;

\item[\texttt{(ii)}] $u\in \mathcal{S}^{\prime }\left( \mathbb{R}^{n}\right)
,$ $\varphi ,\psi \in \mathcal{S}\left( \mathbb{R}^{n}\right) $.
\end{itemize}

\begin{lemma}
Let $\varphi ,\psi \in \mathcal{C}_{0}^{\infty }\left( \mathbb{R}^{n}\right) 
$ $($or $\varphi ,\psi \in \mathcal{S}\left( \mathbb{R}^{n}\right) )$ and $%
u\in \mathcal{D}^{\prime }\left( \mathbb{R}^{n}\right) $ $($or $u\in 
\mathcal{S}^{\prime }\left( \mathbb{R}^{n}\right) )$. Then 
\begin{equation}
\left( \int \psi \right) \left\langle u,\varphi \right\rangle =\int
\left\langle u,\varphi \tau _{y}\psi \right\rangle \mathtt{d}y  \label{ks5}
\end{equation}%
\begin{equation}
\left\langle u,\varphi \ast \psi \right\rangle =\int \varphi \left(
y\right) \left\langle u,\tau _{y}\psi \right\rangle \mathtt{d}y  \label{ks9}
\end{equation}
\end{lemma}

If $\varepsilon _{1},...,\varepsilon _{n}$ is a basis in $\mathbb{R}^{n}$,
we say that $\Gamma =\oplus _{j=1}^{n}%
\mathbb{Z}
\varepsilon _{j}$ is a lattice.

Let $\Gamma \subset \mathbb{R}^{n}$ be a lattice. Let $\psi \in \mathcal{S}%
\left( \mathbb{R}^{n}\right) $. Then $\sum_{\gamma \in \Gamma }\tau _{\gamma
}\psi =\sum_{\gamma \in \Gamma }\psi \left( \cdot -\gamma \right) $ is
uniformly convergent on compact subsets of $\mathbb{R}^{n}$. Since $\partial
^{\alpha }\psi \in \mathcal{S}\left( \mathbb{R}^{n}\right) $, it follows
that there is $\Psi \in \mathcal{C}^{\infty }\left( \mathbb{R}^{n}\right) $
such that%
\begin{equation*}
\Psi =\sum_{\gamma \in \Gamma }\tau _{\gamma }\psi =\sum_{\gamma \in \Gamma
}\psi \left( \cdot -\gamma \right) \text{\quad \textit{in }}\mathcal{C}%
^{\infty }\left( \mathbb{R}^{n}\right) .
\end{equation*}%
Moreover we have $\tau _{\gamma }\Psi =\Psi \left( \cdot -\gamma \right)
=\Psi $ for any $\gamma \in \Gamma $. From here we obtain that $\Psi \in 
\mathcal{BC}^{\infty }\left( \mathbb{R}^{n}\right) $. If $\Psi \left(
y\right) \neq 0$ for any $y\in 
\mathbb{R}
^{n}$, then $\frac{1}{\Psi }\in \mathcal{BC}^{\infty }\left( \mathbb{R}%
^{n}\right) $.

Let $\varphi \in \mathcal{S}\left( \mathbb{R}^{n}\right) $. Then 
\begin{equation*}
\varphi \Psi =\sum_{\gamma \in \Gamma }\varphi \left( \tau _{\gamma }\psi
\right)
\end{equation*}%
with the series convergent in $\mathcal{S}\left( \mathbb{R}^{n}\right) $.
Indeed we have 
\begin{multline*}
\sum_{\gamma \in \Gamma }\left\langle x\right\rangle ^{k}\left\vert \partial
^{\alpha }\varphi \left( x\right) \partial ^{\beta }\psi \left( x-\gamma
\right) \right\vert \\
\leq \sup_{y}\left\langle y\right\rangle ^{n+1}\left\vert \partial ^{\beta
}\psi \left( y\right) \right\vert \sum_{\gamma \in \Gamma }\left\langle
x\right\rangle ^{k}\left\vert \partial ^{\alpha }\varphi \left( x\right)
\left\langle x-\gamma \right\rangle ^{-n-1}\right\vert \\
\leq 2^{\frac{n+1}{2}}\sup_{y}\left\langle y\right\rangle ^{n+1}\left\vert
\partial ^{\beta }\psi \left( y\right) \right\vert \sup_{z}\left\langle
z\right\rangle ^{k+n+1}\left\vert \partial ^{\alpha }\varphi \left( z\right)
\right\vert \sum_{\gamma \in \Gamma }\left\langle \gamma \right\rangle
^{-n-1}.
\end{multline*}%
This estimate proves the convergence of the series in $\mathcal{S}\left( 
\mathbb{R}^{n}\right) $. Let $\chi $ be the sum of the series $\sum_{\gamma
\in \Gamma }\varphi \left( \tau _{\gamma }\psi \right) $ in $\mathcal{S}%
\left( \mathbb{R}^{n}\right) $. Then for any $y\in 
\mathbb{R}
^{n}$ we have%
\begin{eqnarray*}
\chi \left( y\right) &=&\left\langle \delta _{y},\chi \right\rangle
=\left\langle \delta _{y},\sum_{\gamma \in \Gamma }\varphi \left( \tau
_{\gamma }\psi \right) \right\rangle \\
&=&\sum_{\gamma \in \Gamma }\left\langle \delta _{y},\varphi \left( \tau
_{\gamma }\psi \right) \right\rangle =\sum_{\gamma \in \Gamma }\varphi
\left( y\right) \psi \left( y-\gamma \right) \\
&=&\varphi \left( y\right) \Psi \left( y\right) .
\end{eqnarray*}%
So $\varphi \Psi =\sum_{\gamma \in \Gamma }\varphi \left( \tau _{\gamma
}\psi \right) $ in $\mathcal{S}\left( \mathbb{R}^{n}\right) $.

If $\psi ,\varphi \in \mathcal{C}_{0}^{\infty }\left( \mathbb{R}^{n}\right) $
and $\mathcal{S}\left( \mathbb{R}^{n}\right) $ is replaced by $\mathcal{C}%
_{0}^{\infty }\left( \mathbb{R}^{n}\right) $, then the previous observations
are trivial.

\begin{lemma}
Let $u\in \mathcal{D}^{\prime }\left( \mathbb{R}^{n}\right) $ $($or $u\in 
\mathcal{S}^{\prime }\left( \mathbb{R}^{n}\right) )$ and $\psi ,\varphi \in 
\mathcal{C}_{0}^{\infty }\left( \mathbb{R}^{n}\right) $ $($or $\psi ,\varphi
\in \mathcal{S}\left( \mathbb{R}^{n}\right) )$. Then $\Psi =\sum_{\gamma \in
\Gamma }\tau _{\gamma }\psi \in \mathcal{BC}^{\infty }\left( \mathbb{R}%
^{n}\right) $ is $\Gamma $-periodic and%
\begin{equation}
\left\langle u,\Psi \varphi \right\rangle =\sum_{\gamma \in \Gamma
}\left\langle u,\left( \tau _{\gamma }\psi \right) \varphi \right\rangle .
\label{ks6}
\end{equation}
\end{lemma}

\begin{lemma}
$(\mathtt{a})$ Let $\chi \in \mathcal{S}\left( \mathbb{R}^{n}\right) $ and $%
u\in \mathcal{S}^{\prime }\left( \mathbb{R}^{n}\right) $. Then $\widehat{%
\chi u}\in \mathcal{S}^{\prime }\left( \mathbb{R}^{n}\right) \cap \mathcal{C}%
_{pol}^{\infty }\left( \mathbb{R}^{n}\right) $. In fact we have 
\begin{equation*}
\widehat{\chi u}\left( \xi \right) =\left\langle \mathtt{e}^{-\mathtt{i}%
\left\langle \cdot ,\xi \right\rangle }u,\chi \right\rangle =\left\langle u,%
\mathtt{e}^{-\mathtt{i}\left\langle \cdot ,\xi \right\rangle }\chi
\right\rangle ,\quad \xi \in \mathbb{R}^{n}.
\end{equation*}

$(\mathtt{b})$ Let $u\in \mathcal{D}^{\prime }\left( \mathbb{R}^{n}\right) $ 
$($or $u\in \mathcal{S}^{\prime }\left( \mathbb{R}^{n}\right) )$ and $\chi
\in \mathcal{C}_{0}^{\infty }\left( \mathbb{R}^{n}\right) $ $($or $\chi \in 
\mathcal{S}\left( \mathbb{R}^{n}\right) )$. Then 
\begin{equation*}
\mathbb{R}^{n}\times \mathbb{R}^{n}\ni \left( y,\xi \right) \rightarrow 
\widehat{u\tau _{y}\chi }\left( \xi \right) =\left\langle u,\mathtt{e}^{-%
\mathtt{i}\left\langle \cdot ,\xi \right\rangle }\chi \left( \cdot -y\right)
\right\rangle \in 
\mathbb{C}%
\end{equation*}%
is a $\mathcal{C}^{\infty }$-function.
\end{lemma}

\begin{proof}
Let $q:\mathbb{R}_{x}^{n}\times \mathbb{R}_{\xi }^{n}\rightarrow \mathbb{R}$%
, $q\left( x,\xi \right) =\left\langle x,\xi \right\rangle $. Then $\mathtt{e%
}^{-\mathtt{i}q}\left( u\otimes 1\right) \in \mathcal{S}^{\prime }\left( 
\mathbb{R}_{x}^{n}\times \mathbb{R}_{\xi }^{n}\right) $. If $\varphi \in 
\mathcal{S}\left( \mathbb{R}_{\xi }^{n}\right) $, then we have%
\begin{eqnarray*}
\left\langle \mathtt{e}^{-\mathtt{i}q}\left( u\otimes 1\right) ,\chi \otimes
\varphi \right\rangle &=&\left\langle u\otimes 1,\mathtt{e}^{-\mathtt{i}%
q}\left( \chi \otimes \varphi \right) \right\rangle \\
&=&\left\langle u\left( x\right) ,\left\langle 1\left( \xi \right) ,\mathtt{e%
}^{-\mathtt{i}q\left( x,\xi \right) }\chi \left( x\right) \varphi \left( \xi
\right) \right\rangle \right\rangle \\
&=&\left\langle u\left( x\right) ,\chi \left( x\right) \left\langle 1\left(
\xi \right) ,\mathtt{e}^{-\mathtt{i}\left\langle x,\xi \right\rangle
}\varphi \left( \xi \right) \right\rangle \right\rangle \\
&=&\left\langle u,\chi \widehat{\varphi }\right\rangle =\left\langle 
\widehat{\chi u},\varphi \right\rangle
\end{eqnarray*}%
and 
\begin{eqnarray*}
\left\langle \widehat{\chi u},\varphi \right\rangle &=&\left\langle \mathtt{e%
}^{-\mathtt{i}q}\left( u\otimes 1\right) ,\chi \otimes \varphi \right\rangle
\\
&=&\left\langle 1\left( \xi \right) ,\left\langle u\left( x\right) ,\mathtt{e%
}^{-\mathtt{i}\left\langle x,\xi \right\rangle }\chi \left( x\right) \varphi
\left( \xi \right) \right\rangle \right\rangle \\
&=&\left\langle 1\left( \xi \right) ,\varphi \left( \xi \right) \left\langle
u,\mathtt{e}^{-\mathtt{i}\left\langle \cdot ,\xi \right\rangle }\chi
\right\rangle \right\rangle \\
&=&\left\langle 1\left( \xi \right) ,\varphi \left( \xi \right) \left\langle 
\mathtt{e}^{-\mathtt{i}\left\langle \cdot ,\xi \right\rangle }u,\chi
\right\rangle \right\rangle \\
&=&\int \varphi \left( \xi \right) \left\langle \mathtt{e}^{-\mathtt{i}%
\left\langle \cdot ,\xi \right\rangle }u,\chi \right\rangle \mathtt{d}\xi
\end{eqnarray*}%
This proves that%
\begin{equation*}
\widehat{\chi u}\left( \xi \right) =\left\langle \mathtt{e}^{-\mathtt{i}%
\left\langle \cdot ,\xi \right\rangle }u,\chi \right\rangle ,\quad \xi \in 
\mathbb{R}^{n}.
\end{equation*}
\end{proof}

Let $u\in \mathcal{D}^{\prime }\left( 
\mathbb{R}
^{n}\right) $ $($or $u\in \mathcal{S}^{\prime }\left( 
\mathbb{R}
^{n}\right) )$ and $\chi \in \mathcal{C}_{0}^{\infty }\left( 
\mathbb{R}
^{n}\right) \smallsetminus 0$ $($or $\chi \in \mathcal{S}\left( 
\mathbb{R}
^{n}\right) \smallsetminus 0)$. Let $\widetilde{\chi }\in \mathcal{C}%
_{0}^{\infty }\left( \mathbb{R}^{n}\right) $ $($or $\widetilde{\chi }\in 
\mathcal{S}\left( \mathbb{R}^{n}\right) )$and $\varphi \in \mathcal{C}%
_{0}^{\infty }\left( 
\mathbb{R}
^{n}\right) $. By using (\ref{ks5}) we get 
\begin{eqnarray*}
\left\langle u\tau _{z}\widetilde{\chi },\varphi \right\rangle &=&\frac{1}{%
\left\Vert \chi \right\Vert _{L^{2}}^{2}}\int \left\langle u\tau _{z}%
\widetilde{\chi },\left( \tau _{y}\chi \right) \left( \tau _{y}\overline{%
\chi }\right) \varphi \right\rangle \mathtt{d}y \\
&=&\frac{1}{\left\Vert \chi \right\Vert _{L^{2}}^{2}}\int \left\langle
u\tau _{y}\chi ,\left( \tau _{z}\widetilde{\chi }\right) \left( \tau _{y}%
\overline{\chi }\right) \varphi \right\rangle \mathtt{d}y,
\end{eqnarray*}%
\begin{equation*}
\left\vert \left\langle u\tau _{z}\widetilde{\chi },\varphi \right\rangle
\right\vert \leq \frac{1}{\left\Vert \chi \right\Vert _{L^{2}}^{2}}\int
\left\Vert u\tau _{y}\chi \right\Vert _{\mathcal{H}^{\mathbf{s}}}\left\Vert
\left( \tau _{z}\widetilde{\chi }\right) \left( \tau _{y}\overline{\chi }%
\right) \varphi \right\Vert _{\mathcal{H}^{-\mathbf{s}}}\mathtt{d}y.
\end{equation*}

Let $\Gamma \subset \mathbb{R}^{n}$ be a lattice. Let $u\in \mathcal{D}%
^{\prime }\left( 
\mathbb{R}
^{n}\right) $ $($or $u\in \mathcal{S}^{\prime }\left( 
\mathbb{R}
^{n}\right) )$ and let $\chi \in \mathcal{C}_{0}^{\infty }\left( 
\mathbb{R}
^{n}\right) $ $($or $\chi \in \mathcal{S}\left( 
\mathbb{R}
^{n}\right) )$ be such that%
\begin{equation*}
\Psi =\Psi _{\Gamma ,\chi }=\sum_{\gamma \in \Gamma }\left\vert \tau
_{\gamma }\chi \right\vert ^{2}>0.
\end{equation*}%
Then $\Psi ,\frac{1}{\Psi }\in \mathcal{BC}^{\infty }\left( \mathbb{R}%
^{n}\right) $ and both are $\Gamma $-periodic. Let $\widetilde{\chi }\in 
\mathcal{C}_{0}^{\infty }\left( 
\mathbb{R}
^{n}\right) $ $($or $\widetilde{\chi }\in \mathcal{S}\left( 
\mathbb{R}
^{n}\right) )$. Using (\ref{ks6}) we obtain that 
\begin{eqnarray*}
\left\langle u\tau _{z}\widetilde{\chi },\varphi \right\rangle
&=&\sum_{\gamma \in \Gamma }\left\langle u\tau _{\gamma }\chi ,\frac{1}{\Psi 
}\left( \tau _{\gamma }\overline{\chi }\right) \left( \tau _{z}\widetilde{%
\chi }\right) \varphi \right\rangle , \\
\left\vert \left\langle u\tau _{z}\widetilde{\chi },\varphi \right\rangle
\right\vert &\leq &\sum_{\gamma \in \Gamma }\left\Vert u\tau _{\gamma }\chi
\right\Vert _{\mathcal{H}^{\mathbf{s}}}\left\Vert \frac{1}{\Psi }\left( \tau
_{\gamma }\overline{\chi }\right) \left( \tau _{z}\widetilde{\chi }\right)
\varphi \right\Vert _{\mathcal{H}^{-\mathbf{s}}} \\
&\leq &C_{\Psi }\sum_{\gamma \in \Gamma }\left\Vert u\tau _{\gamma }\chi
\right\Vert _{\mathcal{H}^{\mathbf{s}}}\left\Vert \left( \tau _{\gamma }%
\overline{\chi }\right) \left( \tau _{z}\widetilde{\chi }\right) \varphi
\right\Vert _{\mathcal{H}^{-\mathbf{s}}}.
\end{eqnarray*}%
In the last inequality we used the Proposition \ref{ks4} and the fact that $%
\frac{1}{\Psi }\in \mathcal{BC}^{\infty }\left( 
\mathbb{R}
^{n}\right) $.

If $\left( Y,\mathtt{\mu }\right) $ is either $%
\mathbb{R}
^{n}$ with Lebesgue measure or $\Gamma $ with the counting measure, then the
previous estimates can be written as:%
\begin{equation*}
\left\vert \left\langle u\tau _{z}\widetilde{\chi },\varphi \right\rangle
\right\vert \leq Cst\int_{Y}\left\Vert u\tau _{y}\chi \right\Vert _{\mathcal{%
H}^{\mathbf{s}}}\left\Vert \left( \tau _{z}\widetilde{\chi }\right) \left(
\tau _{y}\overline{\chi }\right) \varphi \right\Vert _{\mathcal{H}^{-\mathbf{%
s}}}\mathtt{d\mu }\left( y\right)
\end{equation*}

We shall use Proposition \ref{ks4} to estimate $\left\Vert \left( \tau _{z}%
\widetilde{\chi }\right) \left( \tau _{y}\overline{\chi }\right) \varphi
\right\Vert _{\mathcal{H}^{-\mathbf{s}}}$. Let us write $m_{\mathbf{s}}$ for$%
\left[ \left\vert \mathbf{s}\right\vert _{1}+\frac{n+1}{2}\right] +1$. Then
we have%
\begin{equation*}
\left\Vert \left( \tau _{z}\widetilde{\chi }\right) \left( \tau _{y}%
\overline{\chi }\right) \varphi \right\Vert _{\mathcal{H}^{-\mathbf{s}}}\leq
Cst\sup_{\left\vert \alpha +\beta \right\vert \leq m_{\mathbf{s}}}\left\vert
\left( \left( \tau _{z}\partial ^{\alpha }\widetilde{\chi }\right) \left(
\tau _{y}\partial ^{\beta }\overline{\chi }\right) \right) \right\vert
\left\Vert \varphi \right\Vert _{\mathcal{H}^{-\mathbf{s}}}.
\end{equation*}%
For any $N\in 
\mathbb{N}
$ there is a continuous seminorm $p=p_{N,\mathbf{s}}$ on $\mathcal{S}\left( 
\mathbb{R}^{n}\right) $ so that 
\begin{eqnarray*}
\left\vert \left( \tau _{z}\partial ^{\alpha }\widetilde{\chi }\right)
\left( \tau _{y}\partial ^{\beta }\overline{\chi }\right) \left( x\right)
\right\vert &\leq &p\left( \widetilde{\chi }\right) p\left( \chi \right)
\left\langle x-z\right\rangle ^{-2N}\left\langle x-y\right\rangle ^{-2N} \\
&\leq &2^{N}p\left( \widetilde{\chi }\right) p\left( \chi \right)
\left\langle 2x-z-y\right\rangle ^{-N}\left\langle z-y\right\rangle ^{-N} \\
&\leq &2^{N}p\left( \widetilde{\chi }\right) p\left( \chi \right)
\left\langle z-y\right\rangle ^{-N},\quad \left\vert \alpha +\beta
\right\vert \leq m_{\mathbf{s}}.
\end{eqnarray*}%
Here we used the inequality%
\begin{equation*}
\left\langle X\right\rangle ^{-2N}\left\langle Y\right\rangle ^{-2N}\leq
2^{N}\left\langle X+Y\right\rangle ^{-N}\left\langle X-Y\right\rangle
^{-N},\quad X,Y\in \mathbb{R}^{m}
\end{equation*}%
which is a consequence of Peetre's inequality:%
\begin{equation*}
\left. 
\begin{array}{c}
\left\langle X+Y\right\rangle ^{N}\leq 2^{\frac{N}{2}}\left\langle
X\right\rangle ^{N}\left\langle Y\right\rangle ^{N} \\ 
\\ 
\left\langle X-Y\right\rangle ^{N}\leq 2^{\frac{N}{2}}\left\langle
X\right\rangle ^{N}\left\langle Y\right\rangle ^{N}%
\end{array}%
\right\} \Rightarrow \left\langle X+Y\right\rangle ^{N}\left\langle
X-Y\right\rangle ^{N}\leq 2^{N}\left\langle X\right\rangle ^{2N}\left\langle
Y\right\rangle ^{2N}
\end{equation*}%
Hence 
\begin{gather*}
\sup_{\left\vert \alpha +\beta \right\vert \leq m_{\mathbf{s}}}\left\vert
\left( \left( \tau _{z}\partial ^{\alpha }\widetilde{\chi }\right) \left(
\tau _{y}\partial ^{\beta }\overline{\chi }\right) \right) \right\vert \leq
2^{N}p_{N,\mathbf{s}}\left( \widetilde{\chi }\right) p_{N,\mathbf{s}}\left(
\chi \right) \left\langle z-y\right\rangle ^{-N}, \\
\left\Vert \left( \tau _{z}\widetilde{\chi }\right) \left( \tau _{y}%
\overline{\chi }\right) \varphi \right\Vert _{\mathcal{H}^{-\mathbf{s}}}\leq
C\left( N,\mathbf{s},\chi \mathbf{,}\widetilde{\chi }\right) \left\langle
z-y\right\rangle ^{-N}\left\Vert \varphi \right\Vert _{\mathcal{H}^{-\mathbf{%
s}}}, \\
\left\vert \left\langle u\tau _{z}\widetilde{\chi },\varphi \right\rangle
\right\vert \leq C\left( N,\mathbf{s},\chi \mathbf{,}\widetilde{\chi }%
\right) \left( \int_{Y}\left\Vert u\tau _{y}\chi \right\Vert _{\mathcal{H}^{%
\mathbf{s}}}\left\langle z-y\right\rangle ^{-N}\mathtt{d\mu }\left( y\right)
\right) \left\Vert \varphi \right\Vert _{\mathcal{H}^{-\mathbf{s}}}.
\end{gather*}%
The last estimate implies that 
\begin{equation*}
\left\Vert u\tau _{z}\widetilde{\chi }\right\Vert _{\mathcal{H}^{\mathbf{s}%
}}\leq C\left( N,\mathbf{s},\chi \mathbf{,}\widetilde{\chi }\right) \left(
\int_{Y}\left\Vert u\tau _{y}\chi \right\Vert _{\mathcal{H}^{\mathbf{s}%
}}\left\langle z-y\right\rangle ^{-N}\mathtt{d\mu }\left( y\right) \right)
\end{equation*}%
Let $N=n+1$ and $1\leq p<\infty $. If $\left( Z,\mathtt{\upsilon }\right) $
is either $%
\mathbb{R}
^{n}$ with Lebesgue measure or a lattice with the counting measure, then
Schur's lemma implies%
\begin{equation*}
\left( \int_{Z}\left\Vert u\tau _{z}\widetilde{\chi }\right\Vert _{\mathcal{H%
}^{\mathbf{s}}}^{p}\mathtt{d\upsilon }\left( z\right) \right) ^{\frac{1}{p}%
}\leq C^{\prime }\left( n,\mathbf{s},\chi \mathbf{,}\widetilde{\chi }\right)
\left\Vert \left\langle \cdot \right\rangle ^{-n-1}\right\Vert
_{L^{1}}\left( \int_{Y}\left\Vert u\tau _{y}\chi \right\Vert _{\mathcal{H}^{%
\mathbf{s}}}^{p}\mathtt{d\mu }\left( y\right) \right) ^{\frac{1}{p}}
\end{equation*}%
For $p=\infty $ we have 
\begin{equation*}
\sup_{z}\left\Vert u\tau _{z}\widetilde{\chi }\right\Vert _{\mathcal{H}^{%
\mathbf{s}}}\leq C^{\prime }\left( n,\mathbf{s},\chi \mathbf{,}\widetilde{%
\chi }\right) \left\Vert \left\langle \cdot \right\rangle ^{-n-1}\right\Vert
_{L^{1}}\sup_{y}\left\Vert u\tau _{y}\chi \right\Vert _{\mathcal{H}^{\mathbf{%
s}}}^{p}.
\end{equation*}%
By taking different combinations of $\left( Y,\mathtt{\mu }\right) $ and $%
\left( Z,\mathtt{\upsilon }\right) $ we obtain the following result.

\begin{proposition}
\label{ks8}Let $u\in \mathcal{D}^{\prime }\left( 
\mathbb{R}
^{n}\right) $ $($or $u\in \mathcal{S}^{\prime }\left( 
\mathbb{R}
^{n}\right) )$ and $\chi \in \mathcal{C}_{0}^{\infty }\left( 
\mathbb{R}
^{n}\right) \smallsetminus 0$ $($or $\chi \in \mathcal{S}\left( 
\mathbb{R}
^{n}\right) \smallsetminus 0)$. Let $1\leq p<\infty $.

$\left( \mathtt{a}\right) $ If $\widetilde{\chi }\in \mathcal{C}_{0}^{\infty
}\left( 
\mathbb{R}
^{n}\right) $ $($or $\widetilde{\chi }\in \mathcal{S}\left( 
\mathbb{R}
^{n}\right) )$, then there is $C\left( n,\mathbf{s},\chi \mathbf{,}%
\widetilde{\chi }\right) >0$ such that%
\begin{eqnarray*}
\left( \int \left\Vert u\tau _{\widetilde{y}}\widetilde{\chi }\right\Vert _{%
\mathcal{H}^{\mathbf{s}}}^{p}\mathtt{d}\widetilde{y}\right) ^{\frac{1}{p}}
&\leq &C\left( n,\mathbf{s},\chi \mathbf{,}\widetilde{\chi }\right) \left(
\int \left\Vert u\tau _{y}\chi \right\Vert _{\mathcal{H}^{\mathbf{s}}}^{p}%
\mathtt{d}y\right) ^{\frac{1}{p}}, \\
\sup_{\widetilde{y}}\left\Vert u\tau _{\widetilde{y}}\widetilde{\chi }%
\right\Vert _{\mathcal{H}^{\mathbf{s}}} &\leq &C\left( n,\mathbf{s},\chi 
\mathbf{,}\widetilde{\chi }\right) \sup_{y}\left\Vert u\tau _{y}\chi
\right\Vert _{\mathcal{H}^{\mathbf{s}}}.
\end{eqnarray*}

$\left( \mathtt{b}\right) $ If $\Gamma \subset 
\mathbb{R}
^{n}$ is a lattice such that 
\begin{equation*}
\Psi =\Psi _{\Gamma ,\chi }=\sum_{\gamma \in \Gamma }\left\vert \tau
_{\gamma }\chi \right\vert ^{2}>0
\end{equation*}%
and $\widetilde{\chi }\in \mathcal{C}_{0}^{\infty }\left( 
\mathbb{R}
^{n}\right) $ $($or $\widetilde{\chi }\in \mathcal{S}\left( 
\mathbb{R}
^{n}\right) )$, then there is $C\left( n,\mathbf{s},\Gamma ,\chi \mathbf{,}%
\widetilde{\chi }\right) >0$ such that%
\begin{eqnarray*}
\left( \int \left\Vert u\tau _{\widetilde{y}}\widetilde{\chi }\right\Vert _{%
\mathcal{H}^{\mathbf{s}}}^{p}\mathtt{d}\widetilde{y}\right) ^{\frac{1}{p}}
&\leq &C\left( n,\mathbf{s},\Gamma ,\chi \mathbf{,}\widetilde{\chi }\right)
\left( \sum_{\gamma \in \Gamma }\left\Vert u\tau _{\gamma }\chi \right\Vert
_{\mathcal{H}^{\mathbf{s}}}^{p}\right) ^{\frac{1}{p}}, \\
\sup_{\widetilde{y}}\left\Vert u\tau _{\widetilde{y}}\widetilde{\chi }%
\right\Vert _{\mathcal{H}^{\mathbf{s}}} &\leq &C\left( n,\mathbf{s},\Gamma
,\chi \mathbf{,}\widetilde{\chi }\right) \sup_{\gamma }\left\Vert u\tau
_{\gamma }\chi \right\Vert _{\mathcal{H}^{\mathbf{s}}}.
\end{eqnarray*}

$\left( \mathtt{c}\right) $ If $\widetilde{\Gamma }\subset 
\mathbb{R}
^{n}$ is a lattice and $\widetilde{\chi }\in \mathcal{C}_{0}^{\infty }\left( 
\mathbb{R}
^{n}\right) $ $($or $\widetilde{\chi }\in \mathcal{S}\left( 
\mathbb{R}
^{n}\right) )$, then there is $C\left( n,\mathbf{s},\widetilde{\Gamma },\chi 
\mathbf{,}\widetilde{\chi }\right) >0$ such that%
\begin{eqnarray*}
\left( \sum_{\widetilde{\gamma }\in \widetilde{\Gamma }}\left\Vert u\tau _{%
\widetilde{\gamma }}\chi \right\Vert _{\mathcal{H}^{\mathbf{s}}}^{p}\right)
^{\frac{1}{p}} &\leq &C\left( n,\mathbf{s},\widetilde{\Gamma },\chi \mathbf{,%
}\widetilde{\chi }\right) \left( \int \left\Vert u\tau _{y}\chi \right\Vert
_{\mathcal{H}^{\mathbf{s}}}^{p}\mathtt{d}y\right) ^{\frac{1}{p}}, \\
\sup_{\widetilde{\gamma }}\left\Vert u\tau _{\widetilde{\gamma }}\widetilde{%
\chi }\right\Vert _{\mathcal{H}^{\mathbf{s}}} &\leq &C\left( n,\mathbf{s},%
\widetilde{\Gamma },\chi \mathbf{,}\widetilde{\chi }\right)
\sup_{y}\left\Vert u\tau _{y}\chi \right\Vert _{\mathcal{H}^{\mathbf{s}}}.
\end{eqnarray*}

$\left( \mathtt{d}\right) $ If $\Gamma ,\widetilde{\Gamma }\subset 
\mathbb{R}
^{n}$ are lattices such that 
\begin{equation*}
\Psi =\Psi _{\Gamma ,\chi }=\sum_{\gamma \in \Gamma }\left\vert \tau
_{\gamma }\chi \right\vert ^{2}>0
\end{equation*}%
and $\widetilde{\chi }\in \mathcal{C}_{0}^{\infty }\left( 
\mathbb{R}
^{n}\right) $ $($or $\widetilde{\chi }\in \mathcal{S}\left( 
\mathbb{R}
^{n}\right) )$, then there is $C\left( n,\mathbf{s,}\Gamma ,\widetilde{%
\Gamma },\chi \mathbf{,}\widetilde{\chi }\right) >0$ such that 
\begin{eqnarray*}
\left( \sum_{\widetilde{\gamma }\in \widetilde{\Gamma }}\left\Vert u\tau _{%
\widetilde{\gamma }}\chi \right\Vert _{\mathcal{H}^{\mathbf{s}}}^{p}\right)
^{\frac{1}{p}} &\leq &C\left( n,\mathbf{s},\Gamma ,\chi \mathbf{,}\widetilde{%
\chi }\right) \left( \sum_{\gamma \in \Gamma }\left\Vert u\tau _{\gamma
}\chi \right\Vert _{\mathcal{H}^{\mathbf{s}}}^{p}\right) ^{\frac{1}{p}}, \\
\sup_{\widetilde{\gamma }}\left\Vert u\tau _{\widetilde{\gamma }}\widetilde{%
\chi }\right\Vert _{\mathcal{H}^{\mathbf{s}}} &\leq &C\left( n,\mathbf{s}%
,\Gamma ,\chi \mathbf{,}\widetilde{\chi }\right) \sup_{\gamma }\left\Vert
u\tau _{\gamma }\chi \right\Vert _{\mathcal{H}^{\mathbf{s}}}.
\end{eqnarray*}
\end{proposition}

\begin{definition}
Let $1\leq p\leq \infty $, $\mathbf{s}\in 
\mathbb{R}
^{j}$ and $u\in \mathcal{D}^{\prime }\left( 
\mathbb{R}
^{n}\right) $. We say that $u$ belongs to $\mathcal{K}_{p}^{\mathbf{s}%
}\left( 
\mathbb{R}
^{n}\right) $ if there is $\chi \in \mathcal{C}_{0}^{\infty }\left( \mathbb{R%
}^{n}\right) \smallsetminus 0$ such that the measurable function $%
\mathbb{R}
^{n}\ni y\rightarrow \left\Vert u\tau _{y}\chi \right\Vert _{\mathcal{H}^{%
\mathbf{s}}}\in 
\mathbb{R}
$ belongs to $L^{p}\left( 
\mathbb{R}
^{n}\right) $. We put%
\begin{eqnarray*}
\left\Vert u\right\Vert _{\mathbf{s},p,\chi } &=&\left( \int \left\Vert
u\tau _{y}\chi \right\Vert _{\mathcal{H}^{\mathbf{s}}}^{p}\mathtt{d}y\right)
^{\frac{1}{p}},\quad 1\leq p<\infty , \\
\left\Vert u\right\Vert _{\mathbf{s},\infty ,\chi } &\equiv &\left\Vert
u\right\Vert _{\mathbf{s},\mathtt{ul},\chi }=\sup_{y}\left\Vert u\tau
_{y}\chi \right\Vert _{\mathcal{H}^{\mathbf{s}}}.
\end{eqnarray*}
\end{definition}

\begin{proposition}
\label{ks15}$\left( \mathtt{a}\right) $ The above definition does not depend
on the choice of the function $\chi \in \mathcal{C}_{0}^{\infty }\left( 
\mathbb{R}^{n}\right) \smallsetminus 0$.

$\left( \mathtt{b}\right) $ If $\chi \in \mathcal{C}_{0}^{\infty }\left( 
\mathbb{R}
^{n}\right) \smallsetminus 0$, then $\left\Vert \cdot \right\Vert _{\mathbf{s%
},p,\chi }$ is a norm on $\mathcal{K}_{p}^{\mathbf{s}}\left( 
\mathbb{R}
^{n}\right) $ and the topology that defines does not depend on the function $%
\chi $.

$\left( \mathtt{c}\right) $ Let $\Gamma \subset 
\mathbb{R}
^{n}$ be a lattice and $\chi \in \mathcal{C}_{0}^{\infty }\left( 
\mathbb{R}
^{n}\right) $ be a function with the property that%
\begin{equation*}
\Psi =\Psi _{\Gamma ,\chi }=\sum_{\gamma \in \Gamma }\left\vert \tau
_{\gamma }\chi \right\vert ^{2}>0.
\end{equation*}%
Then 
\begin{equation*}
\mathcal{K}_{p}^{\mathbf{s}}\left( 
\mathbb{R}
^{n}\right) \ni u\rightarrow \left\{ 
\begin{array}{cc}
\left( \sum_{\gamma \in \Gamma }\left\Vert u\tau _{\gamma }\chi \right\Vert
_{\mathcal{H}^{\mathbf{s}}}^{p}\right) ^{\frac{1}{p}} & 1\leq p<\infty \\ 
\sup_{\gamma }\left\Vert u\tau _{\gamma }\chi \right\Vert _{\mathcal{H}^{%
\mathbf{s}}} & p=\infty%
\end{array}%
\right.
\end{equation*}%
is a norm on $\mathcal{K}_{p}^{\mathbf{s}}\left( 
\mathbb{R}
^{n}\right) $ and the topology that defines is the topology of $\mathcal{K}%
_{p}^{\mathbf{s}}\left( 
\mathbb{R}
^{n}\right) $. We shall use the notation 
\begin{equation*}
\left\Vert u\right\Vert _{\mathbf{s},p,\Gamma ,\chi }=\left\{ 
\begin{array}{cc}
\left( \sum_{\gamma \in \Gamma }\left\Vert u\tau _{\gamma }\chi \right\Vert
_{\mathcal{H}^{\mathbf{s}}}^{p}\right) ^{\frac{1}{p}} & 1\leq p<\infty \\ 
\sup_{\gamma }\left\Vert u\tau _{\gamma }\chi \right\Vert _{\mathcal{H}^{%
\mathbf{s}}} & p=\infty%
\end{array}%
\right. .
\end{equation*}

$\left( \mathtt{d}\right) $ If $1\leq p\leq q\leq \infty $, Then%
\begin{equation*}
\mathcal{K}_{1}^{\mathbf{s}}\left( 
\mathbb{R}
^{n}\right) \subset \mathcal{K}_{p}^{\mathbf{s}}\left( 
\mathbb{R}
^{n}\right) \subset \mathcal{K}_{q}^{\mathbf{s}}\left( 
\mathbb{R}
^{n}\right) \subset \mathcal{K}_{\infty }^{\mathbf{s}}\left( 
\mathbb{R}
^{n}\right) \equiv \mathcal{H}_{\mathtt{ul}}^{\mathbf{s}}\left( 
\mathbb{R}
^{n}\right) \subset \mathcal{S}^{\prime }\left( 
\mathbb{R}
^{n}\right) .
\end{equation*}

$\left( \mathtt{e}\right) $ If $s_{1}^{\prime }\leq s_{1}$,...,$%
s_{j}^{\prime }\leq s_{j}$, then $\mathcal{K}_{p}^{\mathbf{s}}\left( 
\mathbb{R}
^{n}\right) \subset \mathcal{K}_{p}^{\mathbf{s}^{\prime }}\left( 
\mathbb{R}
^{n}\right) $.

$\left( \mathtt{f}\right) $ $\left( \mathcal{K}_{p}^{\mathbf{s}}\left( 
\mathbb{R}
^{n}\right) ,\left\Vert \cdot \right\Vert _{\mathbf{s},p,\chi }\right) $ is
a Banach space.

$\left( \mathtt{g}\right) $ $u\in \mathcal{K}_{p}^{\mathbf{s}}\left( 
\mathbb{R}
^{n}\right) $ if and only if there is $l\in \left\{ 1,...,j\right\} $ such
that $u,\partial _{k}u\in \mathcal{K}_{p}^{\mathbf{s-\delta }_{l}}\left( 
\mathbb{R}
^{n}\right) $ for any $k\in N_{l}$, where $\mathbf{\delta }_{l}=\left(
\delta _{l1},...,\delta _{lj}\right) $.

$\left( \mathtt{h}\right) $ If $s_{1}>n_{1}/2,...,s_{j}>n_{j}/2$, then $%
\mathcal{K}_{\infty }^{\mathbf{s}}\left( 
\mathbb{R}
^{n}\right) \equiv \mathcal{H}_{\mathtt{ul}}^{\mathbf{s}}\left( 
\mathbb{R}
^{n}\right) \subset \mathcal{B}\mathcal{C}\left( 
\mathbb{R}
^{n}\right) $.
\end{proposition}

\begin{proof}
$\left( \mathtt{a}\right) $ $\left( \mathtt{b}\right) $ $\left( \mathtt{c}%
\right) $ are immediate consequences of the previous proposition.

$\left( \mathtt{d}\right) $ The inclusions $\mathcal{K}_{1}^{\mathbf{s}%
}\left( 
\mathbb{R}
^{n}\right) \subset \mathcal{K}_{p}^{\mathbf{s}}\left( 
\mathbb{R}
^{n}\right) \subset \mathcal{K}_{q}^{\mathbf{s}}\left( 
\mathbb{R}
^{n}\right) \subset \mathcal{K}_{\infty }^{\mathbf{s}}\left( 
\mathbb{R}
^{n}\right) $ are consequences of the elementary inclusions $l^{1}\subset
l^{p}\subset l^{q}\subset l^{\infty }$. What remains to be shown is the
inclusion $\mathcal{K}_{\infty }^{\mathbf{s}}\left( 
\mathbb{R}
^{n}\right) \equiv \mathcal{H}_{\mathtt{ul}}^{\mathbf{s}}\left( 
\mathbb{R}
^{n}\right) \subset \mathcal{S}^{\prime }\left( 
\mathbb{R}
^{n}\right) $. Let $u\in \mathcal{H}_{\mathtt{ul}}^{\mathbf{s}}\left( 
\mathbb{R}
^{n}\right) $, $\chi \in \mathcal{C}_{0}^{\infty }\left( 
\mathbb{R}
^{n}\right) \smallsetminus 0$ and $\varphi \in \mathcal{C}_{0}^{\infty
}\left( 
\mathbb{R}
^{n}\right) $. We have%
\begin{eqnarray*}
\left\langle u,\varphi \right\rangle &=&\frac{1}{\left\Vert \chi \right\Vert
_{L^{2}}^{2}}\int \left\langle u,\left( \tau _{y}\chi \right) \left( \tau
_{y}\overline{\chi }\right) \varphi \right\rangle \mathtt{d}y \\
&=&\frac{1}{\left\Vert \chi \right\Vert _{L^{2}}^{2}}\int \left\langle
u\tau _{y}\chi ,\left( \tau _{y}\overline{\chi }\right) \varphi
\right\rangle \mathtt{d}y,
\end{eqnarray*}%
\begin{eqnarray*}
\left\vert \left\langle u,\varphi \right\rangle \right\vert &\leq &\frac{1}{%
\left\Vert \chi \right\Vert _{L^{2}}^{2}}\int \left\vert \left\langle u\tau
_{y}\chi ,\left( \tau _{y}\overline{\chi }\right) \varphi \right\rangle
\right\vert \mathtt{d}y \\
&\leq &\frac{1}{\left\Vert \chi \right\Vert _{L^{2}}^{2}}\int \left\Vert
u\tau _{y}\chi \right\Vert _{\mathcal{H}^{\mathbf{s}}}\left\Vert \left( \tau
_{y}\overline{\chi }\right) \varphi \right\Vert _{\mathcal{H}^{-\mathbf{s}}}%
\mathtt{d}y \\
&\leq &\frac{1}{\left\Vert \chi \right\Vert _{L^{2}}^{2}}\left\Vert
u\right\Vert _{\mathbf{s},\infty ,\chi }\int \left\Vert \left( \tau _{y}%
\overline{\chi }\right) \varphi \right\Vert _{\mathcal{H}^{-\mathbf{s}}}%
\mathtt{d}y
\end{eqnarray*}%
We shall use Proposition \ref{ks4} to estimate $\left\Vert \left( \tau _{y}%
\overline{\chi }\right) \varphi \right\Vert _{\mathcal{H}^{-\mathbf{s}}}$.
Let $\widetilde{\chi }\in \mathcal{C}_{0}^{\infty }\left( 
\mathbb{R}
^{n}\right) $, $\widetilde{\chi }=1$ on $\mathtt{supp}\chi $. If $m_{\mathbf{%
s}}=\left[ \left\vert \mathbf{s}\right\vert _{1}+\frac{n+1}{2}\right] +1$,
then we obtain that 
\begin{eqnarray*}
\left\Vert \left( \tau _{y}\overline{\chi }\right) \varphi \right\Vert _{%
\mathcal{H}^{-\mathbf{s}}} &\leq &C\sup_{\left\vert \alpha +\beta
\right\vert \leq m_{\mathbf{s}}}\left\vert \left( \partial ^{\alpha }\varphi
\right) \left( \tau _{y}\partial ^{\beta }\overline{\chi }\right)
\right\vert \left\Vert \tau _{y}\widetilde{\chi }\right\Vert _{\mathcal{H}^{-%
\mathbf{s}}} \\
&=&C\sup_{\left\vert \alpha +\beta \right\vert \leq m_{\mathbf{s}%
}}\left\vert \left( \partial ^{\alpha }\varphi \right) \left( \tau
_{y}\partial ^{\beta }\overline{\chi }\right) \right\vert \left\Vert 
\widetilde{\chi }\right\Vert _{\mathcal{H}^{-\mathbf{s}}}.
\end{eqnarray*}%
Since $\chi $, $\varphi \in \mathcal{S}\left( 
\mathbb{R}
^{n}\right) $ it follows that there is a continuous seminorm $p=p_{n,\mathbf{%
s}}$ on $\mathcal{S}\left( 
\mathbb{R}
^{n}\right) $ so that 
\begin{eqnarray*}
\left\vert \left( \partial ^{\alpha }\varphi \right) \left( \tau
_{y}\partial ^{\beta }\overline{\chi }\right) \left( x\right) \right\vert
&\leq &p\left( \varphi \right) p\left( \chi \right) \left\langle
x-y\right\rangle ^{-2\left( n+1\right) }\left\langle x\right\rangle
^{-2\left( n+1\right) } \\
&\leq &2^{n+1}p\left( \varphi \right) p\left( \chi \right) \left\langle
2x-y\right\rangle ^{-\left( n+1\right) }\left\langle y\right\rangle
^{-\left( n+1\right) } \\
&\leq &2^{n+1}p\left( \varphi \right) p\left( \chi \right) \left\langle
y\right\rangle ^{-\left( n+1\right) },\quad \left\vert \alpha +\beta
\right\vert \leq m_{\mathbf{s}}.
\end{eqnarray*}%
Hence 
\begin{equation*}
\left\vert \left\langle u,\varphi \right\rangle \right\vert \leq 2^{n+1}C%
\frac{1}{\left\Vert \chi \right\Vert _{L^{2}}^{2}}\left\Vert u\right\Vert _{%
\mathbf{s},\infty ,\chi }\left\Vert \left\langle \cdot \right\rangle
^{-\left( n+1\right) }\right\Vert _{L^{1}}\left\Vert \widetilde{\chi }%
\right\Vert _{\mathcal{H}^{-\mathbf{s}}}p\left( \chi \right) p\left( \varphi
\right) .
\end{equation*}

$\left( \mathtt{e}\right) $ is trivial.

$\left( \mathtt{f}\right) $ Let $\left\{ u_{n}\right\} $ be a Cauchy
sequence in $\mathcal{K}_{p}^{\mathbf{s}}\left( 
\mathbb{R}
^{n}\right) $. Since $\mathcal{K}_{p}^{\mathbf{s}}\left( 
\mathbb{R}
^{n}\right) \subset \mathcal{S}^{\prime }\left( 
\mathbb{R}
^{n}\right) $ and $\mathcal{S}^{\prime }\left( 
\mathbb{R}
^{n}\right) $ is sequentially complete, there is $u\in \mathcal{S}^{\prime
}\left( 
\mathbb{R}
^{n}\right) $ such that $u_{n}\rightarrow u$ in $\mathcal{S}^{\prime }\left( 
\mathbb{R}
^{n}\right) $.

Let $\Gamma \subset 
\mathbb{R}
^{n}$ be a lattice and $\chi \in \mathcal{C}_{0}^{\infty }\left( 
\mathbb{R}
^{n}\right) $ be a function with the property that%
\begin{equation*}
\Psi =\Psi _{\Gamma ,\chi }=\sum_{\gamma \in \Gamma }\left\vert \tau
_{\gamma }\chi \right\vert ^{2}>0.
\end{equation*}%
Then for any $\gamma \in \Gamma $ there is $u_{\gamma }\in \mathcal{H}^{%
\mathbf{s}}$ such that $u_{n}\tau _{\gamma }\chi \rightarrow u_{\gamma }$ in 
$\mathcal{H}^{\mathbf{s}}\left( 
\mathbb{R}
^{n}\right) $. As $u_{n}\rightarrow u$ in $\mathcal{S}^{\prime }\left( 
\mathbb{R}
^{n}\right) $ it follows that $u_{\gamma }=u\tau _{\gamma }\chi $ for any $%
\gamma \in \Gamma $.

Since $\left\{ u_{n}\right\} $ is a Cauchy sequence in $\mathcal{K}_{p}^{%
\mathbf{s}}\left( 
\mathbb{R}
^{n}\right) $ there is $M\in \left( 0,\infty \right) $ such that $\left\Vert
u_{n}\right\Vert _{\mathbf{s},p,\Gamma ,\chi }\leq M$ for any $n\in 
\mathbb{N}
$. Let $\varepsilon >0$. Then there is $n_{\varepsilon }$ such that if $m$, $%
n\geq n_{\varepsilon }$, then $\left\Vert u_{m}-u_{n}\right\Vert _{\mathbf{s}%
,p,\Gamma ,\chi }<\varepsilon $.

Let $F\subset \Gamma $ a finite subset. Then 
\begin{eqnarray*}
\left( \sum_{\gamma \in F}\left\Vert u\tau _{\gamma }\chi \right\Vert _{%
\mathcal{H}^{\mathbf{s}}}^{p}\right) ^{\frac{1}{p}} &\leq &\left(
\sum_{\gamma \in F}\left\Vert u\tau _{\gamma }\chi -u_{n}\tau _{\gamma }\chi
\right\Vert _{\mathcal{H}^{\mathbf{s}}}^{p}\right) ^{\frac{1}{p}}+\left(
\sum_{\gamma \in F}\left\Vert u_{n}\tau _{\gamma }\chi \right\Vert _{%
\mathcal{H}^{\mathbf{s}}}^{p}\right) ^{\frac{1}{p}} \\
&\leq &\left( \sum_{\gamma \in F}\left\Vert u\tau _{\gamma }\chi -u_{n}\tau
_{\gamma }\chi \right\Vert _{\mathcal{H}^{\mathbf{s}}}^{p}\right) ^{\frac{1}{%
p}}+M
\end{eqnarray*}%
By passing to the limit we obtain $\left( \sum_{\gamma \in F}\left\Vert
u\tau _{\gamma }\chi \right\Vert _{\mathcal{H}^{\mathbf{s}}}^{p}\right) ^{%
\frac{1}{p}}\leq M$ for any $F\subset \Gamma $ a finite subset. Hence $u\in 
\mathcal{K}_{p}^{\mathbf{s}}\left( 
\mathbb{R}
^{n}\right) $.

For $F\subset \Gamma $ a finite subset and $m,n\geq n_{\varepsilon }$ we
have 
\begin{multline*}
\left( \sum_{\gamma \in F}\left\Vert u\tau _{\gamma }\chi -u_{n}\tau
_{\gamma }\chi \right\Vert _{\mathcal{H}^{\mathbf{s}}}^{p}\right) ^{\frac{1}{%
p}} \\
\leq \left( \sum_{\gamma \in F}\left\Vert u\tau _{\gamma }\chi -u_{m}\tau
_{\gamma }\chi \right\Vert _{\mathcal{H}^{\mathbf{s}}}^{p}\right) ^{\frac{1}{%
p}}+\left( \sum_{\gamma \in F}\left\Vert u_{n}\tau _{\gamma }\chi -u_{m}\tau
_{\gamma }\chi \right\Vert _{\mathcal{H}^{\mathbf{s}}}^{p}\right) ^{\frac{1}{%
p}} \\
\leq \left( \sum_{\gamma \in F}\left\Vert u\tau _{\gamma }\chi -u_{m}\tau
_{\gamma }\chi \right\Vert _{\mathcal{H}^{\mathbf{s}}}^{p}\right) ^{\frac{1}{%
p}}+\varepsilon
\end{multline*}%
By letting $m\rightarrow \infty $ we obtain $\left( \sum_{\gamma \in
F}\left\Vert u\tau _{\gamma }\chi -u_{n}\tau _{\gamma }\chi \right\Vert _{%
\mathcal{H}^{\mathbf{s}}}^{p}\right) ^{\frac{1}{p}}\leq \varepsilon $ for
any $F\subset \Gamma $ a finite subset and $n\geq n_{\varepsilon }$. This
implies that $u_{n}\rightarrow u$ in $\mathcal{K}_{p}^{\mathbf{s}}\left( 
\mathbb{R}
^{n}\right) $. The case $p=\infty $ is even simpler.
\end{proof}

\begin{remark}
Kato-Sobolev spaces are particular cases of Wiener amalgam spaces. More
precisely we have%
\begin{equation*}
\mathcal{K}_{p}^{\mathbf{s}}\left( 
\mathbb{R}
^{n}\right) =W\left( \mathcal{H}^{\mathbf{s}},L^{p}\right)
\end{equation*}%
with local component $\mathcal{H}^{\mathbf{s}}\left( 
\mathbb{R}
^{n}\right) $ and global component $L^{p}\left( 
\mathbb{R}
^{n}\right) $. Wiener amalgam spaces were introduced by Hans Georg
Feichtinger in 1980.
\end{remark}

\begin{proposition}
\label{ks12}Let $\mathbf{s}$, $\mathbf{t}$, $\mathbf{\varepsilon }$, \textbf{%
$\sigma $}$\left( \mathbf{\varepsilon }\right) \in 
\mathbb{R}
^{j}$ such that, $s_{l}+t_{l}>n_{l}/2$, $0<\varepsilon
_{l}<s_{l}+t_{l}-n_{l}/2$, \textbf{$\sigma $}$_{l}\left( \mathbf{\varepsilon 
}\right) =$\textbf{$\sigma $}$_{l}\left( \varepsilon _{l}\right) =\min
\left\{ s_{l},t_{l},s_{l}+t_{l}-n_{l}/2-\varepsilon _{l}\right\} $ for any $%
l\in \left\{ 1,...,j\right\} $. If $\frac{1}{p}+\frac{1}{q}=\frac{1}{r}$,
then%
\begin{equation*}
\mathcal{K}_{p}^{\mathbf{s}}\left( 
\mathbb{R}
^{n}\right) \cdot \mathcal{K}_{q}^{\mathbf{t}}\left( 
\mathbb{R}
^{n}\right) \subset \mathcal{K}_{r}^{\mathbf{\sigma }}\left( 
\mathbb{R}
^{n}\right)
\end{equation*}
\end{proposition}

\begin{proof}
Let $\chi \in \mathcal{C}_{0}^{\infty }\left( 
\mathbb{R}
^{n}\right) \smallsetminus 0$, $u\in \mathcal{K}_{p}^{\mathbf{s}}\left( 
\mathbb{R}
^{n}\right) $ \c{s}i $v\in \mathcal{H}_{q}^{\mathbf{t}}\left( 
\mathbb{R}
^{n}\right) $. By using Proposition \ref{ks7} we obtain that $uv\tau
_{y}\chi ^{2}\in \mathcal{H}^{\mathbf{\sigma }}$ and%
\begin{equation*}
\left\Vert uv\tau _{y}\chi ^{2}\right\Vert _{\mathcal{H}^{\mathbf{\sigma }%
}}\leq C\left\Vert u\tau _{y}\chi \right\Vert _{\mathcal{H}^{\mathbf{s}%
}}\left\Vert u\tau _{y}\chi \right\Vert _{\mathcal{H}^{\mathbf{t}}}
\end{equation*}%
Finally, H\"{o}lder's inequality implies that 
\begin{equation*}
\left\Vert uv\right\Vert _{\mathbf{\sigma },r,\chi ^{2}}\leq C\left\Vert
u\right\Vert _{\mathbf{s},p,\chi }\left\Vert v\right\Vert _{\mathbf{s}%
,q,\chi }
\end{equation*}
\end{proof}

\begin{corollary}
Let $\mathbf{s}\in 
\mathbb{R}
^{j}$ and $1\leq p\leq \infty $. If $s_{1}>n_{1}/2,...,s_{j}>n_{j}/2$, then $%
\mathcal{K}_{p}^{\mathbf{s}}\left( 
\mathbb{R}
^{n}\right) $ is an ideal in $\mathcal{K}_{\infty }^{\mathbf{s}}\left( 
\mathbb{R}
^{n}\right) \equiv \mathcal{H}_{\mathtt{ul}}^{\mathbf{s}}\left( 
\mathbb{R}
^{n}\right) $ with respect to the usual product.
\end{corollary}

Now using the techniques of Coifman and Meyer, developed for the study of
Beurling algebras $A_{\omega }$ and $B_{\omega }$ (see \cite{Meyer} pp
7-10), we shall prove an interesting result.

\begin{theorem}
$\mathcal{H}^{\mathbf{s}}\left( 
\mathbb{R}
^{n}\right) =\mathcal{K}_{2}^{\mathbf{s}}\left( 
\mathbb{R}
^{n}\right) $.
\end{theorem}

To prove the result, we shall use partition of unity built in the previous
section. Let $N\in 
\mathbb{N}
$ and $\left\{ x_{1},...,x_{N}\right\} \subset 
\mathbb{R}
^{n}$ be such that 
\begin{equation*}
\left[ 0,1\right] ^{n}\subset \left( x_{1}+\left[ \frac{1}{3},\frac{2}{3}%
\right] ^{n}\right) \cup ...\cup \left( x_{N}+\left[ \frac{1}{3},\frac{2}{3}%
\right] ^{n}\right)
\end{equation*}%
Let $\widetilde{h}\in \mathcal{C}_{0}^{\infty }\left( 
\mathbb{R}
^{n}\right) $, $\widetilde{h}\geq 0$, be such that $\widetilde{h}=1$ on $%
\left[ \frac{1}{3},\frac{2}{3}\right] ^{n}$ and \texttt{supp}$\widetilde{h}%
\subset \left[ \frac{1}{4},\frac{3}{4}\right] ^{n}$. Then

\begin{enumerate}
\item[$\left( \mathtt{a}\right) $] $\widetilde{H}=\sum_{i=1}^{N}\sum_{\gamma
\in 
\mathbb{Z}
^{n}}\tau _{\gamma +x_{i}}\widetilde{h}\in \mathcal{BC}^{\infty }\left( 
\mathbb{R}
^{n}\right) $ is $%
\mathbb{Z}
^{n}$-periodic and $\widetilde{H}\geq 1$.

\item[$\left( \mathtt{b}\right) $] $h_{i}=\frac{\tau _{x_{i}}\widetilde{h}}{%
\widetilde{H}}\in \mathcal{C}_{0}^{\infty }\left( 
\mathbb{R}
^{n}\right) $, $h_{i}\geq 0$, \texttt{supp}$h_{i}\subset x_{i}+\left[ \frac{1%
}{4},\frac{3}{4}\right] ^{n}=K_{i}$, $\left( K_{i}-K_{i}\right) \cap 
\mathbb{Z}
^{n}=\left\{ 0\right\} $, $i=1,...,N$.

\item[$\left( \mathtt{c}\right) $] $\chi _{i}=\sum_{\gamma \in 
\mathbb{Z}
^{n}}\tau _{\gamma }h_{i}\in \mathcal{BC}^{\infty }\left( 
\mathbb{R}
^{n}\right) $ is $%
\mathbb{Z}
^{n}$-periodic, $i=1,...,N$ and $\sum_{i=1}^{N}\chi _{i}=1.$

\item[$\left( \mathtt{d}\right) $] $h=\sum_{i=1}^{N}h_{i}\in \mathcal{C}%
_{0}^{\infty }\left( 
\mathbb{R}
^{n}\right) $, $h\geq 0$, $\sum_{\gamma \in 
\mathbb{Z}
^{n}}\tau _{\gamma }h=$ $1$.
\end{enumerate}

\begin{lemma}
$\mathcal{K}_{2}^{\mathbf{s}}\left( 
\mathbb{R}
^{n}\right) \subset \mathcal{H}^{\mathbf{s}}\left( 
\mathbb{R}
^{n}\right) $.
\end{lemma}

\begin{proof}
Let $u\in \mathcal{K}_{2}^{\mathbf{s}}\left( 
\mathbb{R}
^{n}\right) $. We have 
\begin{equation*}
u=\sum_{j=1}^{N}\chi _{j}u\quad \text{\textit{with}}\quad \chi
_{j}u=\sum_{\gamma \in 
\mathbb{Z}
^{n}}\left( \tau _{\gamma }h_{j}\right) u.
\end{equation*}%
Since $u\in \mathcal{K}_{2}^{\mathbf{s}}\left( 
\mathbb{R}
^{n}\right) $ applying Proposition \ref{ks8} we get that 
\begin{equation*}
\sum_{\gamma \in 
\mathbb{Z}
^{n}}\left\Vert \left( \tau _{\gamma }h_{j}\right) u\right\Vert _{\mathcal{H}%
^{\mathbf{s}}}^{2}<\infty .
\end{equation*}%
Using Lemma \ref{ks2} it follows that $\chi _{j}u\in \mathcal{H}^{\mathbf{s}%
}\left( 
\mathbb{R}
^{n}\right) $ and 
\begin{equation*}
\left\Vert \chi _{j}u\right\Vert _{\mathcal{H}^{\mathbf{s}}}\approx \left(
\sum_{\gamma \in 
\mathbb{Z}
^{n}}\left\Vert \left( \tau _{\gamma }h_{j}\right) u\right\Vert _{\mathcal{H}%
^{\mathbf{s}}}^{2}\right) ^{\frac{1}{2}}\leq C_{j}\left\Vert u\right\Vert _{%
\mathbf{s},2}
\end{equation*}%
where $\left\Vert \cdot \right\Vert _{\mathbf{s},2}$ is a fixed norm on $%
\mathcal{K}_{2}^{\mathbf{s}}\left( 
\mathbb{R}
^{n}\right) $. So $u=\sum_{j=1}^{N}\chi _{j}u\in \mathcal{H}^{\mathbf{s}%
}\left( 
\mathbb{R}
^{n}\right) $ and%
\begin{equation*}
\left\Vert u\right\Vert _{\mathcal{H}^{\mathbf{s}}}\leq
\sum_{j=1}^{N}\left\Vert \chi _{j}u\right\Vert _{\mathcal{H}^{\mathbf{s}%
}}\leq \left( \sum_{j=1}^{N}C_{j}\right) \left\Vert u\right\Vert _{\mathbf{s}%
,2}.
\end{equation*}
\end{proof}

\begin{lemma}
$\mathcal{H}^{\mathbf{s}}\left( 
\mathbb{R}
^{n}\right) \subset \mathcal{K}_{2}^{\mathbf{s}}\left( 
\mathbb{R}
^{n}\right) $.
\end{lemma}

\begin{proof}
Then the following statements are equivalent:

\texttt{(i)} $u\in \mathcal{H}^{\mathbf{s}}\left( 
\mathbb{R}
^{n}\right) $

\texttt{(ii)} $\chi _{j}u\in \mathcal{H}^{\mathbf{s}}\left( 
\mathbb{R}
^{n}\right) $, $j=1,...,N$. (Here we use Proposition \ref{ks3} $\left( 
\mathtt{c}\right) $)

\texttt{(iii)} $\left\{ \left\Vert \left( \tau _{\gamma }h_{j}\right)
u\right\Vert _{\mathcal{H}^{\mathbf{s}}}\right\} _{\gamma \in 
\mathbb{Z}
^{n}}\in l^{2}\left( 
\mathbb{Z}
^{n}\right) $, $j=1,...,N$. (Here we use Lemma \ref{ks2})

Since $h=\sum_{j=1}^{N}h_{j}$ and 
\begin{equation*}
\left\Vert \left( \tau _{\gamma }h\right) u\right\Vert _{\mathcal{H}^{%
\mathbf{s}}}\leq \sum_{j=1}^{N}\left\Vert \left( \tau _{\gamma }h_{j}\right)
u\right\Vert _{\mathcal{H}^{\mathbf{s}}},\quad \gamma \in 
\mathbb{Z}
^{n}
\end{equation*}%
we get that $\left\{ \left\Vert \left( \tau _{\gamma }h\right) u\right\Vert
_{\mathcal{H}^{\mathbf{s}}}\right\} _{\gamma \in 
\mathbb{Z}
^{n}}\in l^{2}\left( 
\mathbb{Z}
^{n}\right) $. Since $h=\sum_{j=1}^{N}h_{j}\in \mathcal{C}_{0}^{\infty
}\left( 
\mathbb{R}
^{n}\right) $, $h\geq 0$, $\sum_{\gamma \in 
\mathbb{Z}
^{n}}\tau _{\gamma }h=1$ it follows that $u\in \mathcal{K}_{2}^{\mathbf{s}%
}\left( 
\mathbb{R}
^{n}\right) $ and 
\begin{eqnarray*}
\left\Vert u\right\Vert _{\mathbf{s},2,h} &\approx &\left\Vert \left\{
\left\Vert \left( \tau _{\gamma }h\right) u\right\Vert _{\mathcal{H}^{%
\mathbf{s}}}\right\} _{\gamma \in 
\mathbb{Z}
^{n}}\right\Vert _{l^{2}\left( 
\mathbb{Z}
^{n}\right) } \\
&\leq &\sum_{j=1}^{N}\left\Vert \left\{ \left\Vert \left( \tau _{\gamma
}h_{j}\right) u\right\Vert _{\mathcal{H}^{\mathbf{s}}}\right\} _{\gamma \in 
\mathbb{Z}
^{n}}\right\Vert _{l^{2}\left( 
\mathbb{Z}
^{n}\right) } \\
&\approx &\sum_{j=1}^{N}\left\Vert \chi _{j}u\right\Vert _{\mathcal{H}^{%
\mathbf{s}}}\leq Cst\left\Vert u\right\Vert _{\mathcal{H}^{\mathbf{s}}}.
\end{eqnarray*}
\end{proof}

\begin{corollary}[Kato]
\label{ks14}Let $\mathbf{s}$, $\mathbf{t}$, $\mathbf{\varepsilon }$, \textbf{%
$\sigma $}$\left( \mathbf{\varepsilon }\right) \in 
\mathbb{R}
^{j}$ such that, $s_{l}+t_{l}>n_{l}/2$, $0<\varepsilon
_{l}<s_{l}+t_{l}-n_{l}/2$, \textbf{$\sigma $}$_{l}\left( \mathbf{\varepsilon 
}\right) =$\textbf{$\sigma $}$_{l}\left( \varepsilon _{l}\right) =\min
\left\{ s_{l},t_{l},s_{l}+t_{l}-n_{l}/2-\varepsilon _{l}\right\} $ for any $%
l\in \left\{ 1,...,j\right\} $. Then%
\begin{equation*}
\mathcal{H}_{\mathtt{ul}}^{\mathbf{s}}\left( 
\mathbb{R}
^{n}\right) \cdot \mathcal{H}^{\mathbf{t}}\left( 
\mathbb{R}
^{n}\right) \subset \mathcal{H}^{\mathbf{\sigma }}\left( 
\mathbb{R}
^{n}\right) ,\quad \mathcal{H}^{\mathbf{s}}\left( 
\mathbb{R}
^{n}\right) \cdot \mathcal{H}_{\mathtt{ul}}^{\mathbf{t}}\left( 
\mathbb{R}
^{n}\right) \subset \mathcal{H}^{\mathbf{\sigma }}\left( 
\mathbb{R}
^{n}\right) .
\end{equation*}
\end{corollary}

\begin{lemma}
If $1\leq p<\infty $, then $\mathcal{S}\left( 
\mathbb{R}
^{n}\right) $ is dense in $\mathcal{K}_{p}^{\mathbf{s}}\left( 
\mathbb{R}
^{n}\right) $.
\end{lemma}

\begin{proof}
\texttt{(i)} Let $\psi \in \mathcal{C}_{0}^{\infty }\left( 
\mathbb{R}
^{n}\right) $ be such that $\psi =1$ on $B\left( 0,1\right) $, $\psi
^{\varepsilon }\left( x\right) =\psi \left( \varepsilon x\right) $, $%
0<\varepsilon \leq 1$, $x\in 
\mathbb{R}
^{n}$. If $u\in \mathcal{H}^{\mathbf{s}}\left( 
\mathbb{R}
^{n}\right) $, then $\psi ^{\varepsilon }u\rightarrow u$ in $\mathcal{H}^{%
\mathbf{s}}\left( 
\mathbb{R}
^{n}\right) $. Moreover we have 
\begin{equation*}
\left\Vert \psi ^{\varepsilon }u\right\Vert _{\mathcal{H}^{\mathbf{s}}}\leq
C\left( s,n,\psi \right) \left\Vert u\right\Vert _{\mathcal{H}^{\mathbf{s}%
}},\quad 0<\varepsilon \leq 1,
\end{equation*}%
where 
\begin{eqnarray*}
C\left( s,n,\psi \right) &=&\left( 2\pi \right) ^{-n}2^{\left\vert \mathbf{s}%
\right\vert _{1}/2}\sup_{0<\varepsilon \leq 1}\left( \int \left\langle \eta
\right\rangle ^{\left\vert \mathbf{s}\right\vert _{1}}\varepsilon
^{-n}\left\vert \widehat{\psi }\left( \eta /\varepsilon \right) \right\vert 
\mathtt{d}\eta \right) \\
&=&\left( 2\pi \right) ^{-n}2^{\left\vert \mathbf{s}\right\vert
_{1}/2}\sup_{0<\varepsilon \leq 1}\left( \int \left\langle \varepsilon \eta
\right\rangle ^{\left\vert \mathbf{s}\right\vert _{1}}\left\vert \widehat{%
\psi }\left( \eta \right) \right\vert \mathtt{d}\eta \right) \\
&=&\left( 2\pi \right) ^{-n}2^{\left\vert \mathbf{s}\right\vert
_{1}/2}\left( \int \left\langle \eta \right\rangle ^{\left\vert \mathbf{s}%
\right\vert _{1}}\left\vert \widehat{\psi }\left( \eta \right) \right\vert 
\mathtt{d}\eta \right) .
\end{eqnarray*}

\texttt{(ii)} Suppose that $u\in \mathcal{K}_{p}^{\mathbf{s}}\left( 
\mathbb{R}
^{n}\right) $. Let $F\subset 
\mathbb{Z}
^{n}$ be an arbitrary finite subset. Then the subadditivity property of the
norm $\left\Vert \cdot \right\Vert _{l^{p}}$ implies that:%
\begin{multline*}
\left\Vert \psi ^{\varepsilon }u-u\right\Vert _{\mathbf{s},p,%
\mathbb{Z}
^{n},\chi }\leq \left( \sum_{\gamma \in F}\left\Vert \psi ^{\varepsilon
}u\tau _{\gamma }\chi -u\tau _{\gamma }\chi \right\Vert _{\mathcal{H}^{%
\mathbf{s}}}^{p}\right) ^{\frac{1}{p}}+\left( \sum_{\gamma \in 
\mathbb{Z}
^{n}\smallsetminus F}\left\Vert \psi ^{\varepsilon }u\tau _{\gamma }\chi
\right\Vert _{\mathcal{H}^{\mathbf{s}}}^{p}\right) ^{\frac{1}{p}} \\
+\left( \sum_{\gamma \in 
\mathbb{Z}
^{n}\smallsetminus F}\left\Vert u\tau _{\gamma }\chi \right\Vert _{\mathcal{H%
}^{\mathbf{s}}}^{p}\right) ^{\frac{1}{p}} \\
\leq \left( \sum_{\gamma \in F}\left\Vert \psi ^{\varepsilon }u\tau _{\gamma
}\chi -u\tau _{\gamma }\chi \right\Vert _{\mathcal{H}^{\mathbf{s}%
}}^{p}\right) ^{\frac{1}{p}}+\left( C\left( s,n,\psi \right) +1\right)
\left( \sum_{\gamma \in 
\mathbb{Z}
^{n}\smallsetminus F}\left\Vert u\tau _{\gamma }\chi \right\Vert _{\mathcal{H%
}^{\mathbf{s}}}^{p}\right) ^{\frac{1}{p}}
\end{multline*}%
By making $\varepsilon \rightarrow 0$ we deduce that 
\begin{equation*}
\limsup_{\varepsilon \rightarrow 0}\left\Vert \psi ^{\varepsilon
}u-u\right\Vert _{\mathbf{s},p,%
\mathbb{Z}
^{n},\chi }\leq \left( C\left( s,n,\psi \right) +1\right) \left(
\sum_{\gamma \in 
\mathbb{Z}
^{n}\smallsetminus F}\left\Vert u\tau _{\gamma }\chi \right\Vert _{\mathcal{H%
}^{\mathbf{s}}}^{p}\right) ^{\frac{1}{p}}
\end{equation*}%
for any $F\subset 
\mathbb{Z}
^{n}$ finite subset. Hence $\lim_{\varepsilon \rightarrow 0}\psi
^{\varepsilon }u=u$ in $\mathcal{K}_{p}^{\mathbf{s}}\left( 
\mathbb{R}
^{n}\right) $. The immediate consequence is that

\texttt{(iii)} $\mathcal{E}^{\prime }\left( 
\mathbb{R}
^{n}\right) \cap \mathcal{K}_{p}^{\mathbf{s}}\left( 
\mathbb{R}
^{n}\right) $ is dense in $\mathcal{K}_{p}^{\mathbf{s}}\left( 
\mathbb{R}
^{n}\right) $.

\texttt{(iv)} Suppose that $u\in \mathcal{E}^{\prime }\left( 
\mathbb{R}
^{n}\right) \cap \mathcal{K}_{p}^{\mathbf{s}}\left( 
\mathbb{R}
^{n}\right) $. Let $\varphi \in \mathcal{C}_{0}^{\infty }\left( 
\mathbb{R}
^{n}\right) $ be such that \texttt{supp}$\varphi \subset B\left( 0;1\right) $%
, $\int \varphi \left( x\right) \mathtt{d}x=1$. For $\varepsilon \in \left(
0,1\right] $, we set $\varphi _{\varepsilon }=\varepsilon ^{-n}\varphi
\left( \cdot /\varepsilon \right) $. Let $K=$\texttt{supp}$u+\overline{%
B\left( 0;1\right) }$. Let $\chi \in \mathcal{C}_{0}^{\infty }\left( 
\mathbb{R}
^{n}\right) \smallsetminus 0$. Then there is a finite set $F=F_{K,\chi
}\subset 
\mathbb{Z}
^{n}$ such that $\left( \tau _{\gamma }\chi \right) \left( \varphi
_{\varepsilon }\ast u-u\right) =0$ for any $\gamma \in 
\mathbb{Z}
^{n}\smallsetminus F$. It follows that 
\begin{eqnarray*}
\left\Vert \varphi _{\varepsilon }\ast u-u\right\Vert _{\mathbf{s},p,%
\mathbb{Z}
^{n},\chi } &=&\left( \sum_{\gamma \in F}\left\Vert \left( \tau _{\gamma
}\chi \right) \left( \varphi _{\varepsilon }\ast u-u\right) \right\Vert _{%
\mathcal{H}^{\mathbf{s}}}^{p}\right) ^{\frac{1}{p}} \\
&\approx &\left( \sum_{\gamma \in F}\left\Vert \left( \tau _{\gamma }\chi
\right) \left( \varphi _{\varepsilon }\ast u-u\right) \right\Vert _{\mathcal{%
H}^{\mathbf{s}}}^{2}\right) ^{\frac{1}{2}} \\
&\approx &\left\Vert \varphi _{\varepsilon }\ast u-u\right\Vert _{\mathcal{H}%
^{\mathbf{s}}}\rightarrow 0,\quad \text{\textit{as }}\varepsilon \rightarrow
0.
\end{eqnarray*}
\end{proof}

\section{Wiener-L\'{e}vy theorem for Kato-Sobolev algebras}

We shall work only in the case $j=1$, i.e. only in the case of the usual
Kato-Sobolev spaces. The case $j>1$ can be treated in the same way but with
more complicated notations and statements which can hide the ideas and the
beauty of some arguments. So 
\begin{gather*}
\mathcal{H}^{s}\left( 
\mathbb{R}
^{n}\right) =\left\{ u\in \mathcal{S}^{\prime }\left( 
\mathbb{R}
^{n}\right) :\left( 1-\triangle _{%
\mathbb{R}
^{n}}\right) ^{s/2}u\in L^{2}\left( 
\mathbb{R}
^{n}\right) \right\} , \\
\left\Vert u\right\Vert _{\mathcal{H}^{\mathbf{s}}}=\left\Vert \left(
1-\triangle _{%
\mathbb{R}
^{n}}\right) ^{s/2}u\right\Vert _{L^{2}},\quad u\in \mathcal{H}^{\mathbf{s}},
\end{gather*}

Let $1\leq p\leq \infty $, $s\in 
\mathbb{R}
$ and $u\in \mathcal{D}^{\prime }\left( 
\mathbb{R}
^{n}\right) $. We say that $u$ belongs to $u\in \mathcal{K}_{p}^{s}\left( 
\mathbb{R}
^{n}\right) $ if there is $\chi \in \mathcal{C}_{0}^{\infty }\left( \mathbb{R%
}^{n}\right) \smallsetminus 0$ such that the measurable function $%
\mathbb{R}
^{n}\ni y\rightarrow \left\Vert u\tau _{y}\chi \right\Vert _{\mathcal{H}%
^{s}}\in 
\mathbb{R}
$ belongs to $L^{p}\left( 
\mathbb{R}
^{n}\right) $. We put%
\begin{eqnarray*}
\left\Vert u\right\Vert _{s,p,\chi } &=&\left( \int \left\Vert u\tau
_{y}\chi \right\Vert _{\mathcal{H}^{s}}^{p}\mathtt{d}y\right) ^{\frac{1}{p}%
},\quad 1\leq p<\infty , \\
\left\Vert u\right\Vert _{s,\infty ,\chi } &\equiv &\left\Vert u\right\Vert
_{s,\mathtt{ul},\chi }=\sup_{y}\left\Vert u\tau _{y}\chi \right\Vert _{%
\mathcal{H}^{s}}.
\end{eqnarray*}%
In Kato's notation $\mathcal{K}_{\infty }^{s}\left( 
\mathbb{R}
^{n}\right) \equiv \mathcal{H}_{\mathtt{ul}}^{s}\left( 
\mathbb{R}
^{n}\right) $ the uniformly local Sobolev space of order $s$.

\begin{lemma}
$\left( \mathtt{a}\right) $ $\mathcal{BC}^{m}\left( 
\mathbb{R}
^{n}\right) \subset \mathcal{H}_{\mathtt{ul}}^{m}\left( 
\mathbb{R}
^{n}\right) $ for any $m\in 
\mathbb{N}
$.

$\left( \mathtt{b}\right) $ $\mathcal{BC}^{\left[ \left\vert s\right\vert %
\right] +1}\left( 
\mathbb{R}
^{n}\right) \subset \mathcal{H}_{\mathtt{ul}}^{s}\left( 
\mathbb{R}
^{n}\right) $ for any $s\in 
\mathbb{R}
$.
\end{lemma}

\begin{proof}
$\left( \mathtt{a}\right) $ Let $u\in \mathcal{BC}^{m}\left( 
\mathbb{R}
^{n}\right) $ and $\chi \in \mathcal{S}\left( 
\mathbb{R}
^{n}\right) $. Then using Leibniz's formula 
\begin{equation*}
\partial ^{\alpha }\left( u\tau _{y}\chi \right) =\sum_{\beta \leq \alpha }%
\binom{\alpha }{\beta }\partial ^{\beta }u\cdot \tau _{y}\partial ^{\alpha
-\beta }\chi
\end{equation*}%
we get that $\partial ^{\alpha }\left( u\tau _{y}\chi \right) \in
L^{2}\left( 
\mathbb{R}
^{n}\right) $ for any $\alpha \in 
\mathbb{N}
^{n}$ with $\left\vert \alpha \right\vert \leq m$. Also there is $C=C\left(
m,n\right) >0$ such that 
\begin{equation*}
\left\Vert u\tau _{y}\chi \right\Vert _{\mathcal{H}^{m}}\approx \left(
\sum_{\left\vert \alpha \right\vert \leq m}\left\Vert \partial ^{\alpha
}\left( u\tau _{y}\chi \right) \right\Vert _{L^{2}}^{2}\right) ^{1/2}\leq
C\left\Vert u\right\Vert _{\mathcal{BC}^{m}}\left\Vert \chi \right\Vert _{%
\mathcal{H}^{m}},\quad y\in 
\mathbb{R}
^{n}
\end{equation*}%
which implies 
\begin{equation*}
\left\Vert u\right\Vert _{m,\mathtt{ul},\chi }\leq C\left\Vert u\right\Vert
_{\mathcal{BC}^{m}}\left\Vert \chi \right\Vert _{\mathcal{H}^{m}}.
\end{equation*}%
$\left( \mathtt{b}\right) $ We have $\mathcal{BC}^{\left[ \left\vert
s\right\vert \right] +1}\left( 
\mathbb{R}
^{n}\right) \subset \mathcal{H}_{\mathtt{ul}}^{\left[ \left\vert
s\right\vert \right] +1}\left( 
\mathbb{R}
^{n}\right) \subset \mathcal{H}_{\mathtt{ul}}^{\left\vert s\right\vert
}\left( 
\mathbb{R}
^{n}\right) \subset \mathcal{H}_{\mathtt{ul}}^{s}\left( 
\mathbb{R}
^{n}\right) $.
\end{proof}

Let $\varphi \in \mathcal{C}_{0}^{\infty }\left( 
\mathbb{R}
^{n}\right) $, $\varphi \geq 0$ be such that \texttt{supp}$\varphi \subset
B\left( 0;1\right) $, $\int \varphi \left( x\right) \mathtt{d}x=1$. For $%
\varepsilon \in \left( 0,1\right] $, we set $\varphi _{\varepsilon
}=\varepsilon ^{-n}\varphi \left( \cdot /\varepsilon \right) $.

\begin{lemma}
If $s^{\prime }\leq s$, then 
\begin{equation*}
\left\Vert \varphi _{\varepsilon }\ast u-u\right\Vert _{\mathcal{H}%
^{s^{\prime }}}\leq 2^{1-\min \left\{ s-s^{\prime },1\right\} }\varepsilon
^{\min \left\{ s-s^{\prime },1\right\} }\left\Vert u\right\Vert _{\mathcal{H}%
^{s}},\quad u\in \mathcal{H}^{s}\left( 
\mathbb{R}
^{n}\right) .
\end{equation*}
\end{lemma}

\begin{proof}
We have 
\begin{equation*}
\mathcal{F}\left( \varphi _{\varepsilon }\ast u-u\right) \left( \xi \right)
=\left( \widehat{\varphi }\left( \varepsilon \xi \right) -1\right) \widehat{u%
}\left( \xi \right)
\end{equation*}%
with 
\begin{equation*}
\widehat{\varphi }\left( \varepsilon \xi \right) -1=\int \left( \mathtt{e}^{%
\mathbb{-}\mathtt{i}\left\langle x,\varepsilon \xi \right\rangle }-1\right)
\varphi \left( x\right) \mathtt{d}x
\end{equation*}%
Since $\left\vert \mathtt{e}^{\mathbb{-}\mathtt{i}\lambda }-1\right\vert
\leq \left\vert \lambda \right\vert $ we get%
\begin{equation*}
\left\vert \widehat{\varphi }\left( \varepsilon \xi \right) -1\right\vert
\leq \left\{ 
\begin{array}{c}
2\int \varphi \left( x\right) \mathtt{d}x \\ 
\varepsilon \left\vert \xi \right\vert \int \left\vert x\right\vert \varphi
\left( x\right) \mathtt{d}x%
\end{array}%
\right. \leq \left\{ 
\begin{array}{c}
2 \\ 
\varepsilon \left\vert \xi \right\vert%
\end{array}%
\right.
\end{equation*}%
If $0\leq s-s^{\prime }\leq 1$, then 
\begin{eqnarray*}
\left\vert \widehat{\varphi }\left( \varepsilon \xi \right) -1\right\vert
&=&\left\vert \widehat{\varphi }\left( \varepsilon \xi \right) -1\right\vert
^{1-\left( s-s^{\prime }\right) }\left\vert \widehat{\varphi }\left(
\varepsilon \xi \right) -1\right\vert ^{s-s^{\prime }} \\
&\leq &2^{1-\left( s-s^{\prime }\right) }\varepsilon ^{s-s^{\prime
}}\left\vert \xi \right\vert ^{s-s^{\prime }}\leq 2^{1-\left( s-s^{\prime
}\right) }\varepsilon ^{s-s^{\prime }}\left\langle \xi \right\rangle
^{s-s^{\prime }}
\end{eqnarray*}%
which implies that 
\begin{equation*}
\left\Vert \varphi _{\varepsilon }\ast u-u\right\Vert _{\mathcal{H}%
^{s^{\prime }}}\leq 2^{1-\left( s-s^{\prime }\right) }\varepsilon
^{s-s^{\prime }}\left\Vert u\right\Vert _{\mathcal{H}^{s}},\quad u\in 
\mathcal{H}^{s}\left( 
\mathbb{R}
^{n}\right) .
\end{equation*}

If $s^{\prime }\leq s-1$, then 
\begin{equation*}
\left\Vert \varphi _{\varepsilon }\ast u-u\right\Vert _{\mathcal{H}%
^{s^{\prime }}}\leq \varepsilon \left\Vert u\right\Vert _{\mathcal{H}%
^{s^{\prime }+1}}\leq \varepsilon \left\Vert u\right\Vert _{\mathcal{H}%
^{s}},\quad u\in \mathcal{H}^{s}\left( 
\mathbb{R}
^{n}\right) .
\end{equation*}
\end{proof}

Let $\chi ,\chi _{0}\in \mathcal{C}_{0}^{\infty }\left( 
\mathbb{R}
^{n}\right) \smallsetminus 0$ be such that $\chi _{0}=1$ on \texttt{supp}$%
\chi +B\left( 0;1\right) $. Let $u\in \mathcal{H}_{\mathtt{ul}}^{s}\left( 
\mathbb{R}
^{n}\right) $. Then for $0<\varepsilon \leq 1$ we have%
\begin{equation*}
\tau _{y}\chi \left( \varphi _{\varepsilon }\ast u-u\right) =\tau _{y}\chi
\left( \varphi _{\varepsilon }\ast \left( u\tau _{y}\chi _{0}\right) -u\tau
_{y}\chi _{0}\right) .
\end{equation*}%
Proposition \ref{ks3} and the previous lemma imply 
\begin{eqnarray*}
\left\Vert \tau _{y}\chi \left( \varphi _{\varepsilon }\ast u-u\right)
\right\Vert _{\mathcal{H}^{s^{\prime }}} &\leq &C_{s^{\prime },\chi
}\left\Vert \varphi _{\varepsilon }\ast \left( u\tau _{y}\chi _{0}\right)
-u\tau _{y}\chi _{0}\right\Vert _{\mathcal{H}^{s^{\prime }}} \\
&\leq &C_{s^{\prime },\chi }2^{1-\min \left\{ s-s^{\prime },1\right\}
}\varepsilon ^{\min \left\{ s-s^{\prime },1\right\} }\left\Vert u\tau
_{y}\chi _{0}\right\Vert _{\mathcal{H}^{s}}
\end{eqnarray*}%
It follows that%
\begin{equation*}
\left\Vert \varphi _{\varepsilon }\ast u-u\right\Vert _{s^{\prime },\mathtt{%
ul},\chi }\leq C_{s^{\prime },\chi }2^{1-\min \left\{ s-s^{\prime
},1\right\} }\varepsilon ^{\min \left\{ s-s^{\prime },1\right\} }\left\Vert
u\right\Vert _{s,\mathtt{ul},\chi _{0}}
\end{equation*}

\begin{definition}
$\mathcal{H}_{\mathtt{ul}}^{s\left( s^{\prime }\right) }\left( 
\mathbb{R}
^{n}\right) \equiv \left( \mathcal{H}_{\mathtt{ul}}^{s}\left( 
\mathbb{R}
^{n}\right) ,\left\Vert \cdot \right\Vert _{s^{\prime },\mathtt{ul}}\right)
. $
\end{definition}

\begin{corollary}
\label{ks10}$\left( \mathtt{a}\right) $ If $s^{\prime }<s$, then $\mathcal{H}%
_{\mathtt{ul}}^{s}\left( 
\mathbb{R}
^{n}\right) \cap \mathcal{C}^{\infty }\left( 
\mathbb{R}
^{n}\right) $ is dense in $\mathcal{H}_{\mathtt{ul}}^{s\left( s^{\prime
}\right) }\left( 
\mathbb{R}
^{n}\right) $.

$\left( \mathtt{b}\right) $ If $\frac{n}{2}<s^{\prime }<s$, then $\mathcal{BC%
}^{\infty }\left( 
\mathbb{R}
^{n}\right) $ is dense in $\mathcal{H}_{\mathtt{ul}}^{s\left( s^{\prime
}\right) }\left( 
\mathbb{R}
^{n}\right) $.
\end{corollary}

\begin{proof}
$\left( \mathtt{b}\right) $ If $s>\frac{n}{2}$, then $\mathcal{H}_{\mathtt{ul%
}}^{s}\left( 
\mathbb{R}
^{n}\right) \subset \mathcal{BC}\left( 
\mathbb{R}
^{n}\right) $. Therefore $\varphi _{\varepsilon }\ast \mathcal{H}_{\mathtt{ul%
}}^{s}\left( 
\mathbb{R}
^{n}\right) \subset \mathcal{BC}^{\infty }\left( 
\mathbb{R}
^{n}\right) $.
\end{proof}

We need another auxiliary result.

\begin{lemma}
The map 
\begin{equation*}
\mathcal{C}_{0}^{\infty }\left( 
\mathbb{R}
^{n}\right) \times \mathcal{H}_{\mathtt{ul}}^{s}\left( 
\mathbb{R}
^{n}\right) \ni \left( \varphi ,u\right) \rightarrow \varphi \ast u\in 
\mathcal{H}_{\mathtt{ul}}^{s}\left( 
\mathbb{R}
^{n}\right)
\end{equation*}%
is well defined and for any $\chi \in \mathcal{S}\left( 
\mathbb{R}
^{n}\right) \smallsetminus 0$ we have the estimate%
\begin{equation*}
\left\Vert \varphi \ast u\right\Vert _{s,\mathtt{ul},\chi }\leq \left\Vert
\varphi \right\Vert _{L^{1}}\left\Vert u\right\Vert _{s,\mathtt{ul},\chi
},\quad \left( \varphi ,u\right) \in \mathcal{C}_{0}^{\infty }\left( 
\mathbb{R}
^{n}\right) \times \mathcal{H}_{\mathtt{ul}}^{s}\left( 
\mathbb{R}
^{n}\right) .
\end{equation*}
\end{lemma}

\begin{proof}
Let $\left( \varphi ,u\right) \in \mathcal{C}_{0}^{\infty }\left( 
\mathbb{R}
^{n}\right) \times \mathcal{H}_{\mathtt{ul}}^{s}\left( 
\mathbb{R}
^{n}\right) $, $\chi \in \mathcal{S}\left( 
\mathbb{R}
^{n}\right) \smallsetminus 0$ and $\psi \in \mathcal{S}\left( 
\mathbb{R}
^{n}\right) $. Then using (\ref{ks9}) we obtain%
\begin{eqnarray*}
\left\langle \tau _{z}\chi \left( \varphi \ast u\right) ,\psi \right\rangle
&=&\left\langle u,\check{\varphi}\ast \left( \left( \tau _{z}\chi \right)
\psi \right) \right\rangle =\int \check{\varphi}\left( y\right) \left\langle
u,\tau _{y}\left( \left( \tau _{z}\chi \right) \psi \right) \right\rangle 
\mathtt{d}y \\
&=&\int \varphi \left( y\right) \left\langle u,\tau _{-y}\left( \left( \tau
_{z}\chi \right) \psi \right) \right\rangle \mathtt{d}y=\int \varphi \left(
y\right) \left\langle \left( \tau _{z-y}\chi \right) u,\tau _{-y}\psi
\right\rangle \mathtt{d}y,
\end{eqnarray*}%
where $\check{\varphi}\left( y\right) =\varphi \left( -y\right) $. Since 
\begin{equation*}
\left\vert \left\langle \left( \tau _{z-y}\chi \right) u,\tau _{-y}\psi
\right\rangle \right\vert \leq \left\Vert \left( \tau _{z-y}\chi \right)
u\right\Vert _{\mathcal{H}^{s}}\left\Vert \tau _{-y}\psi \right\Vert _{%
\mathcal{H}^{-s}}\leq \left\Vert u\right\Vert _{s,\mathtt{ul},\chi
}\left\Vert \psi \right\Vert _{\mathcal{H}^{-s}}
\end{equation*}%
it follows that 
\begin{equation*}
\left\vert \left\langle \tau _{z}\chi \left( \varphi \ast u\right) ,\psi
\right\rangle \right\vert \leq \left\Vert \varphi \right\Vert
_{L^{1}}\left\Vert u\right\Vert _{s,\mathtt{ul},\chi }\left\Vert \psi
\right\Vert _{\mathcal{H}^{-s}}
\end{equation*}%
Hence $\tau _{z}\chi \left( \varphi \ast u\right) \in \mathcal{H}^{s}\left( 
\mathbb{R}
^{n}\right) $ and $\left\Vert \tau _{z}\chi \left( \varphi \ast u\right)
\right\Vert _{\mathcal{H}^{s}}\leq \left\Vert \varphi \right\Vert
_{L^{1}}\left\Vert u\right\Vert _{s,\mathtt{ul},\chi }$ for every $z\in 
\mathbb{R}
^{n}$, i.e. $\varphi \ast u\in \mathcal{H}_{\mathtt{ul}}^{s}\left( 
\mathbb{R}
^{n}\right) $ and%
\begin{equation*}
\left\Vert \varphi \ast u\right\Vert _{s,\mathtt{ul},\chi }\leq \left\Vert
\varphi \right\Vert _{L^{1}}\left\Vert u\right\Vert _{s,\mathtt{ul},\chi }
\end{equation*}
\end{proof}

\begin{theorem}[Wiener-L\'{e}vy for $\mathcal{H}_{\mathtt{ul}}^{s}\left( 
\mathbb{R}
^{n}\right) $, weak form]
\label{ks11}Let $\mathit{\Omega =\mathring{\Omega}}\subset 
\mathbb{C}
^{d}$ and $\mathit{\Phi }:\mathit{\Omega }\rightarrow 
\mathbb{C}
$ a holomorphic function. Let $s>n/2$.

$\left( \mathtt{a}\right) $ If $u=\left( u_{1},...,u_{d}\right) \in \mathcal{%
H}_{\mathtt{ul}}^{s}\left( 
\mathbb{R}
^{n}\right) ^{d}$ satisfies the condition $\overline{u\left( 
\mathbb{R}
^{n}\right) }\subset \mathit{\Omega }$, then 
\begin{equation*}
\mathit{\Phi }\circ u\equiv \mathit{\Phi }\left( u\right) \in \mathcal{H}_{%
\mathtt{ul}}^{s^{\prime }}\left( 
\mathbb{R}
^{n}\right) ,\quad \forall s^{\prime }<s.
\end{equation*}

$\left( \mathtt{b}\right) $ Suppose that $s^{\prime }\in \left( n/2,s\right) 
$. If $u,u_{\varepsilon }\in \mathcal{H}_{\mathtt{ul}}^{s}\left( 
\mathbb{R}
^{n}\right) ^{d}$, $0<\varepsilon \leq 1$, $\overline{u\left( 
\mathbb{R}
^{n}\right) }\subset \mathit{\Omega }$ and $u_{\varepsilon }\rightarrow u$
in $\mathcal{H}_{\mathtt{ul}}^{s^{\prime }}\left( 
\mathbb{R}
^{n}\right) ^{d}$ as $\varepsilon \rightarrow 0$, then there is $\varepsilon
_{0}\in \left( 0,1\right] $ such that $\overline{u_{\varepsilon }\left( 
\mathbb{R}
^{n}\right) }\subset \mathit{\Omega }$ for every $0<\varepsilon \leq
\varepsilon _{0}$ and $\mathit{\Phi }\left( u_{\varepsilon }\right)
\rightarrow \mathit{\Phi }\left( u\right) $ in $\mathcal{H}_{\mathtt{ul}%
}^{s^{\prime }}\left( 
\mathbb{R}
^{n}\right) $ as $\varepsilon \rightarrow 0$.
\end{theorem}

\begin{proof}
On $%
\mathbb{C}
^{d}$ we shall consider the distance given by the norm 
\begin{equation*}
\left\vert z\right\vert _{\infty }=\max \left\{ \left\vert z_{1}\right\vert
,...,\left\vert z_{d}\right\vert \right\} ,\quad z\in 
\mathbb{C}
^{d}.
\end{equation*}
Let $r=\mathtt{dist}\left( \overline{u\left( 
\mathbb{R}
^{n}\right) },%
\mathbb{C}
^{d}\smallsetminus \mathit{\Omega }\right) /8$. Since $\overline{u\left( 
\mathbb{R}
^{n}\right) }\subset \mathit{\Omega }$ it follows that $r>0$ and 
\begin{equation*}
\bigcup_{y\in \overline{u\left( 
\mathbb{R}
^{n}\right) }}\overline{B\left( y;4r\right) }\subset \mathit{\Omega }.
\end{equation*}

Let $s^{\prime }\in \left( n/2,s\right) $. On $\mathcal{H}_{\mathtt{ul}%
}^{s^{\prime }}\left( 
\mathbb{R}
^{n}\right) ^{d}$ we shall consider the norm 
\begin{equation*}
\left\vert \left\vert \left\vert u\right\vert \right\vert \right\vert
_{s^{\prime },\mathtt{ul}}=\max \left\{ \left\Vert u_{1}\right\Vert
_{s^{\prime },\mathtt{ul}},...,\left\Vert u_{d}\right\Vert _{s^{\prime },%
\mathtt{ul}}\right\} ,\quad u\in \mathcal{H}_{\mathtt{ul}}^{s^{\prime
}}\left( 
\mathbb{R}
^{n}\right) ^{d},
\end{equation*}%
where $\left\Vert \cdot \right\Vert _{s^{\prime },\mathtt{ul}}$ is a fixed
Banach algebra norm on $\mathcal{H}_{\mathtt{ul}}^{s^{\prime }}\left( 
\mathbb{R}
^{n}\right) $, and on $\mathcal{BC}\left( 
\mathbb{R}
^{n}\right) ^{d}$ we shall consider the norm 
\begin{equation*}
\left\vert \left\vert \left\vert u\right\vert \right\vert \right\vert
_{\infty }=\max \left\{ \left\Vert u_{1}\right\Vert _{\infty
},...,\left\Vert u_{d}\right\Vert _{\infty }\right\} ,\quad u\in \mathcal{BC}%
\left( 
\mathbb{R}
^{n}\right) ^{d}.
\end{equation*}%
Since $\mathcal{H}_{\mathtt{ul}}^{s^{\prime }}\left( 
\mathbb{R}
^{n}\right) \subset \mathcal{BC}\left( 
\mathbb{R}
^{n}\right) $ there is $C\geq 1$ so that 
\begin{equation*}
\left\Vert \cdot \right\Vert _{\infty }\leq C\left\Vert \cdot \right\Vert
_{s^{\prime },\mathtt{ul}}
\end{equation*}%
According to Corollary \ref{ks10} $\mathcal{BC}^{\infty }\left( 
\mathbb{R}
^{n}\right) $ is dense in $\mathcal{H}_{\mathtt{ul}}^{s\left( s^{\prime
}\right) }\left( 
\mathbb{R}
^{n}\right) $. Therefore we find $v=\left( v_{1},...,v_{d}\right) \in 
\mathcal{BC}^{\infty }\left( 
\mathbb{R}
^{n}\right) ^{d}$ so that 
\begin{equation*}
\left\vert \left\vert \left\vert u-v\right\vert \right\vert \right\vert
_{s^{\prime },\mathtt{ul}}<r/C.
\end{equation*}

Then 
\begin{equation*}
\left\vert \left\vert \left\vert u-v\right\vert \right\vert \right\vert
_{\infty }\leq C\left\vert \left\vert \left\vert u-v\right\vert \right\vert
\right\vert _{s^{\prime },\mathtt{ul}}<r.
\end{equation*}%
Using the last estimate we show that $\overline{v\left( 
\mathbb{R}
^{n}\right) }\subset \bigcup_{x\in 
\mathbb{R}
^{n}}B\left( u\left( x\right) ;r\right) $. Indeed, if $z\in \overline{%
v\left( 
\mathbb{R}
^{n}\right) }$, then there is $x\in 
\mathbb{R}
^{n}$ such that 
\begin{equation*}
\left\vert z-v\left( x\right) \right\vert _{\infty }<r-\left\vert \left\vert
\left\vert v-u\right\vert \right\vert \right\vert _{\infty }
\end{equation*}%
It follows that 
\begin{eqnarray*}
\left\vert z-u\left( x\right) \right\vert _{\infty } &\leq &\left\vert
z-v\left( x\right) \right\vert _{\infty }+\left\vert v\left( x\right)
-u\left( x\right) \right\vert _{\infty } \\
&\leq &\left\vert z-v\left( x\right) \right\vert _{\infty }+\left\vert
\left\vert \left\vert v-u\right\vert \right\vert \right\vert _{\infty } \\
&<&r-\left\vert \left\vert \left\vert v-u\right\vert \right\vert \right\vert
_{\infty }+\left\vert \left\vert \left\vert v-u\right\vert \right\vert
\right\vert _{\infty }=r
\end{eqnarray*}%
so $z\in B\left( u\left( x\right) ;r\right) $.

From $\overline{v\left( 
\mathbb{R}
^{n}\right) }\subset \bigcup_{x\in 
\mathbb{R}
^{n}}B\left( u\left( x\right) ;r\right) $ we get 
\begin{equation*}
\overline{v\left( 
\mathbb{R}
^{n}\right) }+\overline{B\left( 0;3r\right) }\subset \bigcup_{x\in 
\mathbb{R}
^{n}}B\left( u\left( x\right) ;4r\right) \subset \mathit{\Omega },
\end{equation*}
hence the map%
\begin{equation*}
\mathbb{R}
^{n}\times \overline{B\left( 0;3r\right) }\ni \left( x,\zeta \right)
\rightarrow \mathit{\Phi }\left( v\left( x\right) +\zeta \right) \in 
\mathbb{C}
.
\end{equation*}%
is well defined. Let $\Gamma \left( r\right) $ denote the polydisc $\left(
\partial \mathbb{D}\left( 0,3r\right) \right) ^{d}$. Since $\overline{%
v\left( 
\mathbb{R}
^{n}\right) }+\Gamma \left( r\right) \subset \mathit{\Omega }$ is a compact
subset, the map 
\begin{equation*}
\Gamma \left( r\right) \ni \zeta \rightarrow \mathit{\Phi }\left( \zeta
+v\right) \in \mathcal{BC}^{\left[ s^{\prime }\right] +1}\left( 
\mathbb{R}
^{n}\right) \subset \mathcal{H}_{\mathtt{ul}}^{s^{\prime }}\left( 
\mathbb{R}
^{n}\right)
\end{equation*}%
is continuous.

On the other hand we have 
\begin{equation*}
\left( \zeta _{1}+v_{1}-u_{1}\right) ^{-1},...,\left( \zeta
_{d}+v_{d}-u_{d}\right) ^{-1}\in \mathcal{H}_{\mathtt{ul}}^{s^{\prime
}}\left( 
\mathbb{R}
^{n}\right)
\end{equation*}%
because $\left\Vert u_{1}-v_{1}\right\Vert _{s^{\prime },\mathtt{ul}%
},...,\left\Vert u_{d}-v_{d}\right\Vert _{s^{\prime },\mathtt{ul}}<r/C\leq r$
and $\left\vert \zeta _{1}\right\vert =...=\left\vert \zeta _{d}\right\vert
=3r$.

It follows that the integral 
\begin{equation}
h=\frac{1}{\left( 2\pi \mathtt{i}\right) ^{d}}\int_{\Gamma \left( r\right) }%
\frac{\mathit{\Phi }\left( \zeta +v\right) }{\left( \zeta
_{1}+v_{1}-u_{1}\right) ...\left( \zeta _{d}+v_{d}-u_{d}\right) }\mathtt{d}%
\zeta  \label{ksc}
\end{equation}%
defines an element $h\in \mathcal{H}_{\mathtt{ul}}^{s^{\prime }}\left( 
\mathbb{R}
^{n}\right) $.

Let 
\begin{equation*}
\delta _{x}:\mathcal{H}_{\mathtt{ul}}^{s^{\prime }}\left( 
\mathbb{R}
^{n}\right) \subset \mathcal{BC}\left( 
\mathbb{R}
^{n}\right) \rightarrow 
\mathbb{C}
,\quad w\rightarrow w\left( x\right) ,
\end{equation*}%
be the evaluation functional at $x\in 
\mathbb{R}
^{n}$. Then 
\begin{eqnarray*}
h\left( x\right) &=&\frac{1}{\left( 2\pi \mathtt{i}\right) ^{d}}\int_{\Gamma
\left( r\right) }\frac{\mathit{\Phi }\left( \zeta +v\left( x\right) \right) 
}{\left( \zeta _{1}-\left( u_{1}\left( x\right) -v_{1}\left( x\right)
\right) \right) ...\left( \zeta _{d}-\left( u_{d}\left( x\right)
-v_{d}\left( x\right) \right) \right) }\mathtt{d}\zeta \\
&=&\mathit{\Phi }\left( \zeta +v\left( x\right) \right) |_{\zeta =u\left(
x\right) -v\left( x\right) }=\mathit{\Phi }\left( u\left( x\right) \right)
\end{eqnarray*}%
because $\left\vert u\left( x\right) -v\left( x\right) \right\vert _{\infty
}\leq \left\vert \left\vert \left\vert u-v\right\vert \right\vert
\right\vert _{\infty }<r$, so $u\left( x\right) -v\left( x\right) $ is
within polydisc $\Gamma \left( r\right) $. Hence $h=\mathit{\Phi }\circ
u\equiv \mathit{\Phi }\left( u\right) \in \mathcal{H}_{\mathtt{ul}%
}^{s^{\prime }}\left( 
\mathbb{R}
^{n}\right) $, for any $s^{\prime }\in \left( n/2,s\right) $ so 
\begin{equation*}
\mathit{\Phi }\circ u\equiv \mathit{\Phi }\left( u\right) \in \mathcal{H}_{%
\mathtt{ul}}^{s^{\prime }}\left( 
\mathbb{R}
^{n}\right) ,\quad \forall s^{\prime }<s.
\end{equation*}

$\left( \mathtt{b}\right) $ Let $\varepsilon _{0}\in \left( 0,1\right] $ be
such that for any $0<\varepsilon \leq \varepsilon _{0}$ we have%
\begin{equation*}
\left\vert \left\vert \left\vert u-u_{\varepsilon }\right\vert \right\vert
\right\vert _{s^{\prime },\mathtt{ul}}<r/C.
\end{equation*}%
Then $\left\vert \left\vert \left\vert u-u_{\varepsilon }\right\vert
\right\vert \right\vert _{\infty }\leq C\left\vert \left\vert \left\vert
u-u_{\varepsilon }\right\vert \right\vert \right\vert _{s^{\prime },\mathtt{%
ul}}<r$ and $\overline{u_{\varepsilon }\left( 
\mathbb{R}
^{n}\right) }\subset \bigcup_{x\in 
\mathbb{R}
^{n}}B\left( u\left( x\right) ;r\right) \subset \mathit{\Omega }$ for every $%
0<\varepsilon \leq \varepsilon _{0}$.

On the other hand we have $\left\vert \left\vert \left\vert v-u_{\varepsilon
}\right\vert \right\vert \right\vert _{s^{\prime },\mathtt{ul}}\leq
\left\vert \left\vert \left\vert v-u\right\vert \right\vert \right\vert
_{s^{\prime },\mathtt{ul}}+\left\vert \left\vert \left\vert u-u_{\varepsilon
}\right\vert \right\vert \right\vert _{s^{\prime },\mathtt{ul}}<r/C+r/C\leq
2r$. It follows that 
\begin{equation*}
\left( \zeta _{1}+v_{1}-u_{\varepsilon 1}\right) ^{-1},...,\left( \zeta
_{d}+v_{d}-u_{\varepsilon d}\right) ^{-1}\in \mathcal{H}_{\mathtt{ul}%
}^{s^{\prime }}\left( 
\mathbb{R}
^{n}\right)
\end{equation*}%
because $\left\Vert u_{\varepsilon 1}-v_{1}\right\Vert _{s^{\prime },\mathtt{%
ul}},...,\left\Vert u_{\varepsilon d}-v_{d}\right\Vert _{s^{\prime },\mathtt{%
ul}}<2r$ and $\left\vert \zeta _{1}\right\vert =...=\left\vert \zeta
_{d}\right\vert =3r$.

We obtain that 
\begin{eqnarray*}
\mathit{\Phi }\left( u_{\varepsilon }\right) &=&\frac{1}{\left( 2\pi \mathtt{%
i}\right) ^{d}}\int_{\Gamma \left( r\right) }\frac{\mathit{\Phi }\left(
\zeta +v\right) }{\left( \zeta _{1}+v_{1}-u_{\varepsilon 1}\right) ...\left(
\zeta _{d}+v_{d}-u_{\varepsilon d}\right) }\mathtt{d}\zeta \\
&\rightarrow &\frac{1}{\left( 2\pi \mathtt{i}\right) ^{d}}\int_{\Gamma
\left( r\right) }\frac{\mathit{\Phi }\left( \zeta +v\right) }{\left( \zeta
_{1}+v_{1}-u_{1}\right) ...\left( \zeta _{d}+v_{d}-u_{d}\right) }\mathtt{d}%
\zeta =\mathit{\Phi }\left( u\right)
\end{eqnarray*}%
as $\varepsilon \rightarrow 0$.
\end{proof}

\begin{remark}
According to Coquand and Stolzenberg \cite{Coquand}, this type of
representation formula, $($\ref{ksc}$)$, was introduced more than 60 years
ago by A. P. Calder\'{o}n.
\end{remark}

\begin{lemma}
Suppose that $s>\max \left\{ n/2,3/4\right\} $. Let $\mathit{\Omega =%
\mathring{\Omega}}\subset 
\mathbb{C}
^{d}$ and $\mathit{\Phi }:\mathit{\Omega }\rightarrow 
\mathbb{C}
$ a holomorphic function. If $u=\left( u_{1},...,u_{d}\right) \in \mathcal{H}%
_{\mathtt{ul}}^{s}\left( 
\mathbb{R}
^{n}\right) ^{d}$ satisfies the condition $\overline{u\left( 
\mathbb{R}
^{n}\right) }\subset \mathit{\Omega }$, then 
\begin{equation*}
\partial _{j}\mathit{\Phi }\left( u\right) =\sum_{k=1}^{d}\frac{\partial 
\mathit{\Phi }}{\partial z_{k}}\left( u\right) \cdot \partial
_{j}u_{k},\quad \text{in }\mathcal{D}^{\prime }\left( 
\mathbb{R}
^{n}\right) ,\quad j=1,...,n.
\end{equation*}
\end{lemma}

\begin{proof}
Let $s^{\prime }$ be such that $\max \left\{ n/2,3/4,s-1\right\} <s^{\prime
}<s$. Then $s^{\prime }+s^{\prime }-1>n/2$. Let $u=\left(
u_{1},...,u_{d}\right) \in \left( \mathcal{H}_{\mathtt{ul}}^{s}\left( 
\mathbb{R}
^{n}\right) \right) ^{d}$. We consider the family $\left\{ u_{\varepsilon
}\right\} _{0<\varepsilon \leq 1}$ 
\begin{equation*}
u_{\varepsilon }=\varphi _{\varepsilon }\ast u=\left( \varphi _{\varepsilon
}\ast u_{1},...,\varphi _{\varepsilon }\ast u_{d}\right) \in \mathcal{H}_{%
\mathtt{ul}}^{s}\left( 
\mathbb{R}
^{n}\right) ^{d}
\end{equation*}%
Then $u_{\varepsilon }\rightarrow u$ in $\mathcal{H}_{\mathtt{ul}%
}^{s^{\prime }}\left( 
\mathbb{R}
^{n}\right) ^{d}$ as $\varepsilon \rightarrow 0$, $\partial
_{j}u_{\varepsilon }=\varphi _{\varepsilon }\ast \partial _{j}u\in \mathcal{H%
}_{\mathtt{ul}}^{s-1}\left( 
\mathbb{R}
^{n}\right) ^{d}$ and $\partial _{j}u_{\varepsilon }\rightarrow \partial
_{j}u$ in $\mathcal{H}_{\mathtt{ul}}^{s^{\prime }-1}\left( 
\mathbb{R}
^{n}\right) ^{d}$ as $\varepsilon \rightarrow 0$. Since $\mathit{\Phi }%
\left( u_{\varepsilon }\right) \rightarrow \mathit{\Phi }\left( u\right) $
in $\mathcal{H}_{\mathtt{ul}}^{s^{\prime }}\left( 
\mathbb{R}
^{n}\right) \subset \mathcal{BC}\left( 
\mathbb{R}
^{n}\right) \subset \mathcal{D}^{\prime }\left( 
\mathbb{R}
^{n}\right) $ (Theorem \ref{ks11} $\left( \mathtt{b}\right) $), it follows
that $\partial _{j}\mathit{\Phi }\left( u_{\varepsilon }\right) \rightarrow
\partial _{j}\mathit{\Phi }\left( u\right) $ in $\mathcal{D}^{\prime }\left( 
\mathbb{R}
^{n}\right) $, $j=1,...,n$.

On the other hand we have%
\begin{equation*}
\partial _{j}\mathit{\Phi }\left( u_{\varepsilon }\right) =\sum_{k=1}^{d}%
\frac{\partial \mathit{\Phi }}{\partial z_{k}}\left( u_{\varepsilon }\right)
\cdot \partial _{j}u_{\varepsilon k},\quad \text{\textit{in} }\mathcal{C}%
^{\infty }\left( 
\mathbb{R}
^{n}\right) ,\quad j=1,...,n.
\end{equation*}%
Let $\delta >0$ be such that $s^{\prime }-n/2-\delta >0$. Then 
\begin{equation*}
s^{\prime }-1=\min \left\{ s^{\prime },s^{\prime }-1,s^{\prime }+s^{\prime
}-1-n/2-\delta \right\} .
\end{equation*}%
Since 
\begin{gather*}
\frac{\partial \mathit{\Phi }}{\partial z_{k}}\left( u_{\varepsilon }\right)
\rightarrow \frac{\partial \mathit{\Phi }}{\partial z_{k}}\left( u\right)
,\quad \text{\textit{in} }\mathcal{H}_{\mathtt{ul}}^{s^{\prime }}\left( 
\mathbb{R}
^{n}\right) ,\text{ }(\text{Theorem\textit{\ }}\ref{ks11}\left( \mathtt{b}%
\right) ),\quad k=1,...,d, \\
\partial _{j}u_{\varepsilon }\rightarrow \partial _{j}u,\quad \text{\textit{%
in} }\mathcal{H}_{\mathtt{ul}}^{s^{\prime }-1}\left( 
\mathbb{R}
^{n}\right) ^{d},\quad j=1,...,n,
\end{gather*}%
using Proposition \ref{ks12} we get that%
\begin{equation*}
\partial _{j}\mathit{\Phi }\left( u_{\varepsilon }\right) =\sum_{k=1}^{d}%
\frac{\partial \mathit{\Phi }}{\partial z_{k}}\left( u_{\varepsilon }\right)
\cdot \partial _{j}u_{\varepsilon k}\rightarrow \partial _{j}\mathit{\Phi }%
\left( u\right) =\sum_{k=1}^{d}\frac{\partial \mathit{\Phi }}{\partial z_{k}}%
\left( u\right) \cdot \partial _{j}u_{k},\quad \text{\textit{in} }\mathcal{H}%
_{\mathtt{ul}}^{s^{\prime }-1}\left( 
\mathbb{R}
^{n}\right)
\end{equation*}%
for $j=1,...,n$. Hence 
\begin{equation*}
\partial _{j}\mathit{\Phi }\left( u\right) =\sum_{k=1}^{d}\frac{\partial 
\mathit{\Phi }}{\partial z_{k}}\left( u\right) \cdot \partial
_{j}u_{k},\quad \text{in }\mathcal{D}^{\prime }\left( 
\mathbb{R}
^{n}\right) ,\quad j=1,...,n.
\end{equation*}
\end{proof}

\begin{remark}
Let us note that $\partial _{j}u_{\varepsilon k}=\varphi _{\varepsilon }\ast
\partial _{j}u_{k}\rightarrow \partial _{j}u_{k}$ in $\mathcal{H}_{\mathtt{ul%
}}^{s^{\prime }-1}\left( 
\mathbb{R}
^{n}\right) $, but $\partial _{j}u_{k}$ $\in \mathcal{H}_{\mathtt{ul}%
}^{s-1}\left( 
\mathbb{R}
^{n}\right) $, $j=1,...,n$, $k=1,...,d$. This remark leads to the complete
version of the Wiener-L\'{e}vy theorem.
\end{remark}

\begin{theorem}[Wiener-L\'{e}vy for $\mathcal{H}_{\mathtt{ul}}^{s}\left( 
\mathbb{R}
^{n}\right) $]
Suppose that $s>\max \left\{ n/2,3/4\right\} $. Let $\mathit{\Omega =%
\mathring{\Omega}}\subset 
\mathbb{C}
^{d}$ and $\mathit{\Phi }:\mathit{\Omega }\rightarrow 
\mathbb{C}
$ a holomorphic function. If $u=\left( u_{1},...,u_{d}\right) \in \mathcal{H}%
_{\mathtt{ul}}^{s}\left( 
\mathbb{R}
^{n}\right) ^{d}$ satisfies the condition $\overline{u\left( 
\mathbb{R}
^{n}\right) }\subset \mathit{\Omega }$, then%
\begin{equation*}
\mathit{\Phi }\circ u\equiv \mathit{\Phi }\left( u\right) \in \mathcal{H}_{%
\mathtt{ul}}^{s}\left( 
\mathbb{R}
^{n}\right) .
\end{equation*}
\end{theorem}

\begin{proof}
Let $s^{\prime }$ be such that $\max \left\{ n/2,3/4,s-1\right\} <s^{\prime
}<s$. Then $s^{\prime }+s^{\prime }-1>n/2$. Let $\delta >0$ be such that $%
s^{\prime }-n/2-\delta >0$. Then 
\begin{equation*}
s-1=\min \left\{ s^{\prime },s-1,s^{\prime }+s-1-n/2-\delta \right\} .
\end{equation*}%
Since 
\begin{gather*}
\frac{\partial \mathit{\Phi }}{\partial z_{k}}\left( u\right) \in \mathcal{H}%
_{\mathtt{ul}}^{s^{\prime }}\left( 
\mathbb{R}
^{n}\right) ,\text{ }(\text{Theorem\textit{\ }}\ref{ks11}\left( \mathtt{a}%
\right) ),\quad k=1,...,d, \\
\partial _{j}u\in \mathcal{H}_{\mathtt{ul}}^{s-1}\left( 
\mathbb{R}
^{n}\right) ^{d},\quad j=1,...,n,
\end{gather*}%
using Proposition \ref{ks12} we get that%
\begin{equation*}
\partial _{j}\mathit{\Phi }\left( u\right) =\sum_{k=1}^{d}\frac{\partial 
\mathit{\Phi }}{\partial z_{k}}\left( u\right) \cdot \partial _{j}u_{k}\in 
\mathcal{H}_{\mathtt{ul}}^{s-1}\left( 
\mathbb{R}
^{n}\right) ,\quad j=1,...,n.
\end{equation*}%
Now $\mathit{\Phi }\left( u\right) \in \mathcal{H}_{\mathtt{ul}}^{s^{\prime
}}\left( 
\mathbb{R}
^{n}\right) \subset \mathcal{H}_{\mathtt{ul}}^{s-1}\left( 
\mathbb{R}
^{n}\right) $ and $\partial _{j}\mathit{\Phi }\left( u\right) \in \mathcal{H}%
_{\mathtt{ul}}^{s-1}\left( 
\mathbb{R}
^{n}\right) $, $j=1,...,n$ imply $\mathit{\Phi }\left( u\right) \in \mathcal{%
H}_{\mathtt{ul}}^{s}\left( 
\mathbb{R}
^{n}\right) $.
\end{proof}

\begin{corollary}[Kato]
\label{ks13}Suppose that $s>\max \left\{ n/2,3/4\right\} $.

$\left( \mathtt{a}\right) $ If $u\in \mathcal{H}_{\mathtt{ul}}^{s}\left( 
\mathbb{R}
^{n}\right) $ satisfies the condition 
\begin{equation*}
\left\vert u\left( x\right) \right\vert \geq c>0,\quad x\in 
\mathbb{R}
^{n},
\end{equation*}%
then 
\begin{equation*}
\frac{1}{u}\in \mathcal{H}_{\mathtt{ul}}^{s}\left( 
\mathbb{R}
^{n}\right) .
\end{equation*}

$\left( \mathtt{b}\right) $ If $u\in \mathcal{H}_{\mathtt{ul}}^{s}\left( 
\mathbb{R}
^{n}\right) $, then $\overline{u\left( 
\mathbb{R}
^{n}\right) }$ is the spectrum of the element $u$.
\end{corollary}

\begin{corollary}
Suppose that $s>\max \left\{ n/2,3/4\right\} $. If $u=\left(
u_{1},...,u_{d}\right) \in \mathcal{H}_{\mathtt{ul}}^{s}\left( 
\mathbb{R}
^{n}\right) ^{d}$, then 
\begin{equation*}
\sigma _{\mathcal{H}_{\mathtt{ul}}^{s}}\left( u_{1},...,u_{d}\right) =%
\overline{u\left( 
\mathbb{R}
^{n}\right) },
\end{equation*}%
where $\sigma _{\mathcal{H}_{\mathtt{ul}}^{s}}\left( u_{1},...,u_{d}\right) $
is the joint spectrum of the elements $u_{1},...,u_{d}\in \mathcal{H}_{%
\mathtt{ul}}^{s}\left( 
\mathbb{R}
^{n}\right) $.
\end{corollary}

\begin{proof}
Since 
\begin{equation*}
\delta _{x}:\mathcal{H}_{\mathtt{ul}}^{s^{\prime }}\left( 
\mathbb{R}
^{n}\right) \subset \mathcal{BC}\left( 
\mathbb{R}
^{n}\right) \rightarrow 
\mathbb{C}
,\quad w\rightarrow w\left( x\right) ,
\end{equation*}%
is a multiplicative linear functional, Theorem 3.1.14 of \cite{Hormander 2}
implies the inclusion $\overline{u\left( 
\mathbb{R}
^{n}\right) }\subset \sigma _{\mathcal{H}_{\mathtt{ul}}^{s}}\left(
u_{1},...,u_{d}\right) $. On the other hand, if $\lambda =\left( \lambda
_{1},...,\lambda _{d}\right) \notin \overline{u\left( 
\mathbb{R}
^{n}\right) }$, then 
\begin{equation*}
u_{\lambda }=\overline{\left( u_{1}-\lambda _{1}\right) }\left(
u_{1}-\lambda _{1}\right) +...+\overline{\left( u_{d}-\lambda _{d}\right) }%
\left( u_{d}-\lambda _{d}\right) \in \mathcal{H}_{\mathtt{ul}}^{s}\left( 
\mathbb{R}
^{n}\right)
\end{equation*}%
satisfies the condition 
\begin{equation*}
u_{\lambda }\left( x\right) \geq c>0,\quad x\in 
\mathbb{R}
^{n}.
\end{equation*}%
It follows that 
\begin{equation*}
\frac{1}{u_{\lambda }}\in \mathcal{H}_{\mathtt{ul}}^{s}\left( 
\mathbb{R}
^{n}\right)
\end{equation*}%
and%
\begin{equation*}
v_{1}\left( u_{1}-\lambda _{1}\right) +...+v_{d}\left( u_{d}-\lambda
_{d}\right) =1
\end{equation*}%
with $v_{1}=\overline{\left( u_{1}-\lambda _{1}\right) }/u_{\lambda
},...,v_{d}=\overline{\left( u_{d}-\lambda _{d}\right) }/u_{\lambda }\in 
\mathcal{H}_{\mathtt{ul}}^{s}\left( 
\mathbb{R}
^{n}\right) $. The last equality expresses precisely that $\lambda \notin
\sigma _{\mathcal{H}_{\mathtt{ul}}^{s}}\left( u_{1},...,u_{d}\right) $.
\end{proof}

\begin{corollary}
Suppose that $s>\max \left\{ n/2,3/4\right\} $. Let $\mathit{\Omega =%
\mathring{\Omega}}\subset 
\mathbb{C}
^{d}$ and $\mathit{\Phi }:\mathit{\Omega }\rightarrow 
\mathbb{C}
$ a holomorphic function.

$\left( \mathtt{a}\right) $ Let $1\leq p<\infty $. If $u=\left(
u_{1},...,u_{d}\right) \in \mathcal{K}_{p}^{\mathbf{s}}\left( 
\mathbb{R}
^{n}\right) ^{d}$ satisfies the condition $\overline{u_{1}\left( 
\mathbb{R}
^{n}\right) }\times ...\times \overline{u_{d}\left( 
\mathbb{R}
^{n}\right) }\subset \mathit{\Omega }$ and if $\mathit{\Phi }\left( 0\right)
=0$, then $\mathit{\Phi }\left( u\right) \in \mathcal{K}_{p}^{\mathbf{s}%
}\left( 
\mathbb{R}
^{n}\right) $.

$\left( \mathtt{b}\right) $ If $u=\left( u_{1},...,u_{d}\right) \in \mathcal{%
H}^{s}\left( 
\mathbb{R}
^{n}\right) ^{d}$ satisfies the condition $\overline{u_{1}\left( 
\mathbb{R}
^{n}\right) }\times ...\times \overline{u_{d}\left( 
\mathbb{R}
^{n}\right) }\subset \mathit{\Omega }$ and if $\mathit{\Phi }\left( 0\right)
=0$, then $\mathit{\Phi }\left( u\right) \in \mathcal{H}^{s}\left( 
\mathbb{R}
^{n}\right) $.
\end{corollary}

\begin{proof}
$\left( \mathtt{a}\right) $ Since $\mathcal{K}_{p}^{\mathbf{s}}\left( 
\mathbb{R}
^{n}\right) $ is an ideal in the algebra $\mathcal{H}_{\mathtt{ul}%
}^{s}\left( 
\mathbb{R}
^{n}\right) $, it follows that $0$ belongs to the spectrum of any element of 
$\mathcal{K}_{p}^{\mathbf{s}}\left( 
\mathbb{R}
^{n}\right) $. Hence $0\in \overline{u_{1}\left( 
\mathbb{R}
^{n}\right) }\times ...\times \overline{u_{d}\left( 
\mathbb{R}
^{n}\right) }\subset \mathit{\Omega }$. Shrinking $\mathit{\Omega }$ if
necessary, we can assume that $\mathit{\Omega }=\mathit{\Omega }_{1}\times
...\times \mathit{\Omega }_{d}$ with $\overline{u_{k}\left( 
\mathbb{R}
^{n}\right) }\subset \mathit{\Omega }_{k}$, $k=1,...,d$. Now we continue by
induction on $d$.

Let $F:\mathit{\Omega }\rightarrow 
\mathbb{C}
$ be the holomorphic function defined by 
\begin{equation*}
F\left( z_{1},...,z_{d}\right) =\left\{ 
\begin{array}{ccc}
\frac{\mathit{\Phi }\left( z_{1},...,z_{d}\right) -\mathit{\Phi }\left(
0,...,z_{d}\right) }{z_{1}} & if & z_{1}\neq 0, \\ 
\frac{\partial \mathit{\Phi }}{\partial z_{1}}\left( 0,...,z_{d}\right) & if
& z_{1}=0.%
\end{array}%
\right.
\end{equation*}%
Then $\mathit{\Phi }\left( z_{1},...,z_{d}\right) =z_{1}F\left(
z_{1},...,z_{d}\right) +\mathit{\Phi }\left( 0,...,z_{d}\right) $, so 
\begin{equation*}
\mathit{\Phi }\left( u\right) =u_{1}F\left( u\right) +\mathit{\Phi }\left(
0,...,u_{d}\right) \in \mathcal{K}_{p}^{\mathbf{s}}\left( 
\mathbb{R}
^{n}\right)
\end{equation*}%
because $u_{1}F\left( u\right) \in \mathcal{K}_{p}^{\mathbf{s}}\left( 
\mathbb{R}
^{n}\right) \cdot \mathcal{H}_{\mathtt{ul}}^{s}\left( 
\mathbb{R}
^{n}\right) \subset \mathcal{K}_{p}^{\mathbf{s}}\left( 
\mathbb{R}
^{n}\right) $ and $\mathit{\Phi }\left( 0,...,u_{d}\right) \in \mathcal{K}%
_{p}^{\mathbf{s}}\left( 
\mathbb{R}
^{n}\right) $ by inductive hypothesis.

$\left( \mathtt{b}\right) $ is a consequence of $\left( \mathtt{a}\right) $.
\end{proof}

\begin{corollary}[A division lemma]
Suppose that $s>\max \left\{ n/2,3/4\right\} $. Let $t\in 
\mathbb{R}
$ such that $s+t>n/2$. Let $u\in \mathcal{H}^{t}\left( 
\mathbb{R}
^{n}\right) \cap \mathcal{E}^{\prime }\left( 
\mathbb{R}
^{n}\right) $ and $v\in \mathcal{H}_{\mathtt{ul}}^{s}\left( 
\mathbb{R}
^{n}\right) $. If $v$ satisfies the condition%
\begin{equation*}
\left\vert v\left( x\right) \right\vert \geq c>0,\quad x\in \mathtt{supp}u,
\end{equation*}%
then 
\begin{equation*}
\frac{u}{v}\in \mathcal{H}^{\min \left\{ s,t\right\} }\left( 
\mathbb{R}
^{n}\right) .
\end{equation*}
\end{corollary}

\begin{proof}
Let $\varphi \in \mathcal{C}_{0}^{\infty }\left( 
\mathbb{R}
^{n}\right) $, $0\leq \varphi \leq 1$, $\varphi =1$ on \texttt{supp}$u$, be
such that%
\begin{equation*}
\left\vert v\left( x\right) \right\vert \geq c/2>0,\quad x\in \text{\texttt{%
supp}}\varphi .
\end{equation*}%
Then $w=\varphi \left\vert v\right\vert ^{2}+c^{2}\left( 1-\varphi \right)
/4\in \mathcal{H}_{\mathtt{ul}}^{s}\left( 
\mathbb{R}
^{n}\right) $ satisfies $w\geq \left( \varphi +\left( 1-\varphi \right)
\right) c^{2}/4=c^{2}/4$. If $\delta $ satisfies $0<\delta $ $<\min \left\{
s+t-n/2,s-n/2\right\} $, then 
\begin{equation*}
\min \left\{ s,t\right\} =\min \left\{ s,t,s+t-n/2-\delta \right\} .
\end{equation*}%
By using Corollary \ref{ks13} and Corollary \ref{ks14} we obtain%
\begin{equation*}
\frac{u}{\left\vert v\right\vert ^{2}}=\frac{u}{w}\in \mathcal{H}^{t}\left( 
\mathbb{R}
^{n}\right) \cdot \mathcal{H}_{\mathtt{ul}}^{s}\left( 
\mathbb{R}
^{n}\right) \subset \mathcal{H}^{\min \left\{ s,t\right\} }\left( 
\mathbb{R}
^{n}\right)
\end{equation*}%
This proves the lemma since 
\begin{equation*}
\frac{u}{v}=\overline{v}\cdot \frac{u}{\left\vert v\right\vert ^{2}}\in 
\mathcal{H}_{\mathtt{ul}}^{s}\left( 
\mathbb{R}
^{n}\right) \cdot \mathcal{H}^{\min \left\{ s,t\right\} }\left( 
\mathbb{R}
^{n}\right) \subset \mathcal{H}^{\min \left\{ s,t\right\} }\left( 
\mathbb{R}
^{n}\right) .
\end{equation*}
\end{proof}

\section{Kato-Sobolev spaces and Schatten-von Neumann class properties for
pseudo-differential operators}

We begin this section with some interpolation results of Kato-Sobolev spaces.

For $\mathbf{s}=\left( s_{1},...,s_{j}\right) \in 
\mathbb{R}
^{j}$, if 
\begin{gather*}
F_{\mathbf{s}}=\left\{ v\in \mathcal{S}^{\prime }\left( 
\mathbb{R}
^{n}\right) :\left\langle \left\langle \cdot \right\rangle \right\rangle ^{%
\mathbf{s}}v\in L^{2}\left( 
\mathbb{R}
^{n}\right) \right\} , \\
\left\Vert v\right\Vert _{\mathcal{H}^{\mathbf{s}}}=\left\Vert \left\langle
\left\langle \cdot \right\rangle \right\rangle ^{\mathbf{s}}v\right\Vert
_{L^{2}},\quad v\in F_{\mathbf{s}},
\end{gather*}%
then the Fourier transform $\mathcal{F}$ is an isometry (up to a constant
factor) from $\mathcal{H}^{\mathbf{s}}\left( 
\mathbb{R}
^{n}\right) $ onto $F_{\mathbf{s}}$ and the inverse Fourier transform $%
\mathcal{F}^{-1}$ is an isometry (up to a constant factor) from $F_{\mathbf{s%
}}$ onto $\mathcal{H}^{\mathbf{s}}\left( 
\mathbb{R}
^{n}\right) $. The interpolation property implies then that $\mathcal{F}$
maps continuously $\left[ \mathcal{H}^{\mathbf{s}}\left( 
\mathbb{R}
^{n}\right) ,\mathcal{H}^{\mathbf{t}}\left( 
\mathbb{R}
^{n}\right) \right] _{\theta }$ into $\left[ F_{\mathbf{s}},F_{\mathbf{t}}%
\right] _{\theta }$ and $\mathcal{F}^{-1}$ maps continuously $\left[ F_{%
\mathbf{s}},F_{\mathbf{t}}\right] _{\theta }$ into $\left[ \mathcal{H}^{%
\mathbf{s}}\left( 
\mathbb{R}
^{n}\right) ,\mathcal{H}^{\mathbf{t}}\left( 
\mathbb{R}
^{n}\right) \right] _{\theta }$, so that $\left[ \mathcal{H}^{\mathbf{s}%
}\left( 
\mathbb{R}
^{n}\right) ,\mathcal{H}^{\mathbf{t}}\left( 
\mathbb{R}
^{n}\right) \right] _{\theta }$ coincides with the tempered distributions
whose Fourier transform belongs to $\left[ F_{\mathbf{s}},F_{\mathbf{t}}%
\right] _{\theta }$ (and one deduces in the same way that it is an isometry
if one uses the corresponding norms). Identifying interpolation spaces
between Sobolev spaces $\mathcal{H}^{\mathbf{s}}\left( 
\mathbb{R}
^{n}\right) $ is then the same question as interpolating between some $L^{2}$
spaces with weights. The following lemma is a consequence of this remark and
Theorem 1.18.5 in \cite{Triebel}.

\begin{lemma}
If $\mathbf{s},\mathbf{t}\in 
\mathbb{R}
^{j}$ and $0<\theta <1$ then 
\begin{equation*}
\left[ \mathcal{H}^{\mathbf{s}}\left( 
\mathbb{R}
^{n}\right) ,\mathcal{H}^{\mathbf{t}}\left( 
\mathbb{R}
^{n}\right) \right] _{\theta }=\mathcal{H}^{\left( 1-\theta \right) \mathbf{s%
}+\theta \mathbf{t}}\left( 
\mathbb{R}
^{n}\right) .
\end{equation*}
\end{lemma}

Using the results of \cite{Triebel} Subsection 1.18.1 we obtain the
following corollary.

\begin{corollary}
Let $\mathbf{s},\mathbf{t}\in 
\mathbb{R}
^{j}$, $1\leq p_{0}<\infty $, $1\leq p_{1}\leq \infty $, $0<\theta <1$ and 
\begin{equation*}
\frac{1}{p}=\frac{1-\theta }{p_{0}}+\frac{\theta }{p_{1}},\quad \mathbf{%
\sigma }=\left( 1-\theta \right) \mathbf{s}+\theta \mathbf{t}.
\end{equation*}%
Then 
\begin{equation*}
\left[ l^{p_{0}}\left( 
\mathbb{Z}
^{n},\mathcal{H}^{\mathbf{s}}\left( 
\mathbb{R}
^{n}\right) \right) ,l^{p_{1}}\left( 
\mathbb{Z}
^{n},\mathcal{H}^{\mathbf{t}}\left( 
\mathbb{R}
^{n}\right) \right) \right] _{\theta }=l^{p}\left( 
\mathbb{Z}
^{n},\mathcal{H}^{\mathbf{\sigma }}\left( 
\mathbb{R}
^{n}\right) \right) .
\end{equation*}
\end{corollary}

We pass now to the Kato-Sobolev spaces. We choose $\chi _{%
\mathbb{Z}
^{n}}\in \mathcal{C}_{0}^{\infty }\left( \mathbb{R}^{n}\right) $ so that $%
\sum_{k\in 
\mathbb{Z}
^{n}}\chi _{%
\mathbb{Z}
^{n}}\left( \cdot -k\right) =1$. For $k\in 
\mathbb{Z}
^{n}$ we define the operator%
\begin{equation*}
S_{k}:\mathcal{D}^{\prime }\left( 
\mathbb{R}
^{n}\right) \rightarrow \mathcal{D}^{\prime }\left( 
\mathbb{R}
^{n}\right) ,\quad S_{k}u=\left( \tau _{k}\chi _{%
\mathbb{Z}
^{n}}\right) u.
\end{equation*}%
Now from the definition of $\mathcal{K}_{p}^{\mathbf{s}}\left( 
\mathbb{R}
^{n}\right) $ it follows that the linear operator 
\begin{equation*}
S:\mathcal{K}_{p}^{\mathbf{s}}\left( 
\mathbb{R}
^{n}\right) \rightarrow l^{p}\left( 
\mathbb{Z}
^{n},\mathcal{H}^{\mathbf{s}}\left( 
\mathbb{R}
^{n}\right) \right) ,\quad Su=\left( S_{k}u\right) _{k\in 
\mathbb{Z}
^{n}}
\end{equation*}%
is well defined and continuous.

On the other hand, for any $\chi \in \mathcal{C}_{0}^{\infty }\left( \mathbb{%
R}^{n}\right) $ the operator 
\begin{gather*}
R_{\chi }:l^{p}\left( 
\mathbb{Z}
^{n},\mathcal{H}^{\mathbf{s}}\left( 
\mathbb{R}
^{n}\right) \right) \rightarrow \mathcal{K}_{p}^{\mathbf{s}}\left( 
\mathbb{R}
^{n}\right) , \\
R_{\chi }\left( \left( u_{k}\right) _{k\in 
\mathbb{Z}
^{n}}\right) =\sum_{k\in 
\mathbb{Z}
^{n}}\left( \tau _{k}\chi \right) u_{k}
\end{gather*}%
is well defined and continuous.

Let $\underline{\mathbf{u}}=\left( u_{k}\right) _{k\in 
\mathbb{Z}
^{n}}\in l^{p}\left( 
\mathbb{Z}
^{n},\mathcal{H}^{\mathbf{s}}\left( 
\mathbb{R}
^{n}\right) \right) $. Using Proposition \ref{ks4} we get 
\begin{equation*}
\left\Vert \left( \tau _{k^{\prime }}\chi _{%
\mathbb{Z}
^{n}}\right) \left( \tau _{k}\chi \right) u_{k}\right\Vert _{\mathcal{H}^{%
\mathbf{s}}}\leq Cst\sup_{\left\vert \alpha +\beta \right\vert \leq m_{%
\mathbf{s}}}\left\vert \left( \left( \tau _{k^{\prime }}\partial ^{\alpha
}\chi _{%
\mathbb{Z}
^{n}}\right) \left( \tau _{k}\partial ^{\beta }\chi \right) \right)
\right\vert \left\Vert u_{k}\right\Vert _{\mathcal{H}^{\mathbf{s}}}.
\end{equation*}%
where $m_{\mathbf{s}}=\left[ \left\vert \mathbf{s}\right\vert _{1}+\frac{n+1%
}{2}\right] $ $+1$. Now for some continuous seminorm $p=p_{n,\mathbf{s}}$ on 
$\mathcal{S}\left( \mathbb{R}^{n}\right) $ we have 
\begin{eqnarray*}
\left\vert \left( \left( \tau _{k^{\prime }}\partial ^{\alpha }\chi _{%
\mathbb{Z}
^{n}}\right) \left( \tau _{k}\partial ^{\beta }\chi \right) \right) \left(
x\right) \right\vert &\leq &p\left( \chi _{%
\mathbb{Z}
^{n}}\right) p\left( \chi \right) \left\langle x-k^{\prime }\right\rangle
^{-2\left( n+1\right) }\left\langle x-k\right\rangle ^{-2\left( n+1\right) }
\\
&\leq &2^{n+1}p\left( \chi _{%
\mathbb{Z}
^{n}}\right) p\left( \chi \right) \left\langle 2x-k^{\prime }-k\right\rangle
^{-n-1}\left\langle k^{\prime }-k\right\rangle ^{-n-1} \\
&\leq &2^{n+1}p\left( \chi _{%
\mathbb{Z}
^{n}}\right) p\left( \chi \right) \left\langle k^{\prime }-k\right\rangle
^{-n-1},\quad \left\vert \alpha +\beta \right\vert \leq m_{\mathbf{s}}.
\end{eqnarray*}%
Hence 
\begin{gather*}
\sup_{\left\vert \alpha +\beta \right\vert \leq m_{\mathbf{s}}}\left\vert
\left( \left( \tau _{k^{\prime }}\partial ^{\alpha }\chi _{%
\mathbb{Z}
^{n}}\right) \left( \tau _{k}\partial ^{\beta }\chi \right) \right)
\right\vert \leq 2^{n+1}p\left( \chi _{%
\mathbb{Z}
^{n}}\right) p\left( \chi \right) \left\langle k^{\prime }-k\right\rangle
^{-n-1}, \\
\left\Vert \left( \tau _{k^{\prime }}\chi _{%
\mathbb{Z}
^{n}}\right) \left( \tau _{k}\chi \right) u_{k}\right\Vert _{\mathcal{H}^{%
\mathbf{s}}}\leq C\left( n,\mathbf{s,}\chi _{%
\mathbb{Z}
^{n}},\chi \right) \left\langle k^{\prime }-k\right\rangle ^{-n-1}\left\Vert
u_{k}\right\Vert _{\mathcal{H}^{\mathbf{s}}}.
\end{gather*}%
The last estimate implies that 
\begin{equation*}
\left\Vert \left( \tau _{k^{\prime }}\chi _{%
\mathbb{Z}
^{n}}\right) R_{\chi }\left( \underline{\mathbf{u}}\right) \right\Vert _{%
\mathcal{H}^{\mathbf{s}}}\leq C\left( n,\mathbf{s,}\chi _{%
\mathbb{Z}
^{n}},\chi \right) \sum_{k\in 
\mathbb{Z}
^{n}}\left\langle k^{\prime }-k\right\rangle ^{-n-1}\left\Vert
u_{k}\right\Vert _{\mathcal{H}^{\mathbf{s}}}.
\end{equation*}%
Now Schur's lemma implies the result%
\begin{equation*}
\left( \sum_{k^{\prime }\in 
\mathbb{Z}
^{n}}\left\Vert \left( \tau _{k^{\prime }}\chi _{%
\mathbb{Z}
^{n}}\right) R_{\chi }\left( \underline{\mathbf{u}}\right) \right\Vert _{%
\mathcal{H}^{\mathbf{s}}}^{p}\right) ^{\frac{1}{p}}\leq C^{\prime }\left( n,%
\mathbf{s,}\chi _{%
\mathbb{Z}
^{n}},\chi \right) \left\Vert \left\langle \cdot \right\rangle
^{-n-1}\right\Vert _{L^{1}}\left( \sum_{k\in 
\mathbb{Z}
^{n}}\left\Vert u_{k}\right\Vert _{\mathcal{H}^{\mathbf{s}}}^{p}\right) ^{%
\frac{1}{p}}.
\end{equation*}

If $\chi =1$ on a neighborhood of $\mathtt{supp}\chi _{%
\mathbb{Z}
^{n}}$, then $\chi \chi _{%
\mathbb{Z}
^{n}}=\chi _{%
\mathbb{Z}
^{n}}$ and as a consequence $R_{\chi }S=\mathtt{Id}_{S_{w}^{p}\left( \mathbb{%
R}^{n}\right) }$:%
\begin{eqnarray*}
R_{\chi }Su &=&\sum_{k\in 
\mathbb{Z}
^{n}}\left( \tau _{k}\chi \right) S_{k}u=\sum_{k\in 
\mathbb{Z}
^{n}}\left( \tau _{k}\chi \right) \left( \tau _{k}\chi _{%
\mathbb{Z}
^{n}}\right) u \\
&=&\sum_{k\in 
\mathbb{Z}
^{n}}\left( \tau _{k}\chi _{%
\mathbb{Z}
^{n}}\right) u=u.
\end{eqnarray*}%
Thus we proved the following result.

\begin{proposition}
Under the above conditions, the operator $R_{\chi }:l^{p}\left( 
\mathbb{Z}
^{n},\mathcal{H}^{\mathbf{s}}\right) \rightarrow \mathcal{K}_{p}^{\mathbf{s}%
} $ is a retraction and the operator $S:\mathcal{K}_{p}^{\mathbf{s}%
}\rightarrow l^{p}\left( 
\mathbb{Z}
^{n},\mathcal{H}^{\mathbf{s}}\right) $ is a coretraction.
\end{proposition}

\begin{corollary}
\label{ks17}Let $\mathbf{s},\mathbf{t}\in 
\mathbb{R}
^{j}$, $1\leq p_{0}<\infty $, $1\leq p_{1}\leq \infty $, $0<\theta <1$ and 
\begin{equation*}
\frac{1}{p}=\frac{1-\theta }{p_{0}}+\frac{\theta }{p_{1}},\quad \mathbf{%
\sigma }=\left( 1-\theta \right) \mathbf{s}+\theta \mathbf{t}.
\end{equation*}%
Then 
\begin{equation*}
\left[ \mathcal{K}_{p_{0}}^{\mathbf{s}}\left( 
\mathbb{R}
^{n}\right) ,\mathcal{K}_{p_{1}}^{\mathbf{t}}\left( 
\mathbb{R}
^{n}\right) \right] _{\theta }=\mathcal{K}_{p}^{\mathbf{\sigma }}\left( 
\mathbb{R}
^{n}\right) .
\end{equation*}
\end{corollary}

\begin{proof}
The last part of Theorem \ref{ks15} $\left( \mathtt{d}\right) $ shows that $%
\left\{ \mathcal{K}_{p_{0}}^{\mathbf{s}}\left( 
\mathbb{R}
^{n}\right) ,\mathcal{K}_{p_{1}}^{\mathbf{t}}\left( 
\mathbb{R}
^{n}\right) \right\} $ is an interpolation couple (in the sense of the
notations of Subsection 1.2.1 of \cite{Triebel} one can choose $\mathcal{A}=%
\mathcal{S}^{\prime }\left( 
\mathbb{R}
^{n}\right) $). If $F$ is an interpolation functor, then one obtains by
Theorem 1.2.4 of \cite{Triebel} that%
\begin{equation*}
\left\Vert u\right\Vert _{F\left( \left\{ \mathcal{K}_{p_{0}}^{\mathbf{s}%
}\left( 
\mathbb{R}
^{n}\right) ,\mathcal{K}_{p_{1}}^{\mathbf{t}}\left( 
\mathbb{R}
^{n}\right) \right\} \right) }\sim \left\Vert \left( S_{k}u\right) _{k\in 
\mathbb{Z}
^{n}}\right\Vert _{F\left( \left\{ l^{p_{0}}\left( 
\mathbb{Z}
^{n},\mathcal{H}^{\mathbf{s}}\left( 
\mathbb{R}
^{n}\right) \right) ,l^{p_{1}}\left( 
\mathbb{Z}
^{n},\mathcal{H}^{\mathbf{t}}\left( 
\mathbb{R}
^{n}\right) \right) \right\} \right) }
\end{equation*}%
By specialization we obtain 
\begin{eqnarray*}
\left\Vert u\right\Vert _{\left[ \mathcal{K}_{p_{0}}^{\mathbf{s}}\left( 
\mathbb{R}
^{n}\right) ,\mathcal{K}_{p_{1}}^{\mathbf{t}}\left( 
\mathbb{R}
^{n}\right) \right] _{\theta }} &\sim &\left\Vert \left( S_{k}u\right)
_{k\in 
\mathbb{Z}
^{n}}\right\Vert _{\left[ l^{p_{0}}\left( 
\mathbb{Z}
^{n},\mathcal{H}^{\mathbf{s}}\left( 
\mathbb{R}
^{n}\right) \right) ,l^{p_{1}}\left( 
\mathbb{Z}
^{n},\mathcal{H}^{\mathbf{t}}\left( 
\mathbb{R}
^{n}\right) \right) \right] _{\theta }} \\
&\sim &\left\Vert \left( S_{k}u\right) _{k\in 
\mathbb{Z}
^{n}}\right\Vert _{l^{p}\left( 
\mathbb{Z}
^{n},\mathcal{H}^{\mathbf{\sigma }}\left( 
\mathbb{R}
^{n}\right) \right) } \\
&\sim &\left\Vert u\right\Vert _{\mathcal{K}_{p}^{\mathbf{\sigma }}\left( 
\mathbb{R}
^{n}\right) }
\end{eqnarray*}
\end{proof}

In addition to the above interpolation results we need an embedding theorem
which we shall prove below. First we shall recall the definition of spaces
that appear in this theorem.

\begin{definition}
Let $1\leq p\leq \infty $. We say that a distribution $u\in \mathcal{D}%
^{\prime }\left( \mathbb{R}^{n}\right) $ belongs to $S_{w}^{p}\left( \mathbb{%
R}^{n}\right) $ if there is $\chi \in \mathcal{C}_{0}^{\infty }\left( 
\mathbb{R}^{n}\right) \smallsetminus 0$ such that the measurable function%
\begin{gather*}
U_{\chi ,p}:\mathbb{R}^{n}\rightarrow \left[ 0,+\infty \right) , \\
U_{\chi ,p}\left( \xi \right) =\left\{ 
\begin{array}{ccc}
\sup_{y\in \mathbb{R}^{n}}\left\vert \widehat{u\tau _{y}\chi }\left( \xi
\right) \right\vert & \text{\textit{if}} & p=\infty \\ 
\left( \int \left\vert \widehat{u\tau _{y}\chi }\left( \xi \right)
\right\vert ^{p}\mathtt{d}y\right) ^{1/p} & \text{\textit{if}} & 1\leq
p<\infty%
\end{array}%
\right. , \\
\widehat{u\tau _{y}\chi }\left( \xi \right) =\left\langle u,\mathtt{e}^{-%
\mathtt{i}\left\langle \cdot ,\xi \right\rangle }\chi \left( \cdot -y\right)
\right\rangle .
\end{gather*}%
belongs to $L^{1}\left( \mathbb{R}^{n}\right) $.
\end{definition}

These spaces are special cases of modulation spaces. They were used by many
authors (Boulkhemair, Gr\"{o}chenig, Heil, Sj\"{o}strand, Toft ...) in the
analysis of pseudo-differential operators defined by symbols more general
than usual.

We now list some properties of these spaces.

\begin{proposition}
$(\mathtt{a})$ Let $u\in S_{w}^{p}\left( \mathbb{R}^{n}\right) $ and let $%
\chi \in \mathcal{C}_{0}^{\infty }\left( \mathbb{R}^{n}\right) $. Then the
measurable function%
\begin{gather*}
U_{\chi ,p}:\mathbb{R}^{n}\rightarrow \left[ 0,+\infty \right) , \\
U_{\chi ,p}\left( \xi \right) =\left\{ 
\begin{array}{ccc}
\sup_{y\in \mathbb{R}^{n}}\left\vert \widehat{u\tau _{y}\chi }\left( \xi
\right) \right\vert & \text{\textit{if}} & p=\infty \\ 
\left( \int \left\vert \widehat{u\tau _{y}\chi }\left( \xi \right)
\right\vert ^{p}\mathtt{d}y\right) ^{1/p} & \text{\textit{if}} & 1\leq
p<\infty%
\end{array}%
\right. , \\
\widehat{u\tau _{y}\chi }\left( \xi \right) =\left\langle u,\mathtt{e}^{-%
\mathtt{i}\left\langle \cdot ,\xi \right\rangle }\chi \left( \cdot -y\right)
\right\rangle .
\end{gather*}%
belongs to $L^{1}\left( \mathbb{R}^{n}\right) $.

$(\mathtt{b})$ If we fix $\chi \in \mathcal{C}_{0}^{\infty }\left( \mathbb{R}%
^{n}\right) \smallsetminus 0$ and if we put%
\begin{equation*}
\left\Vert u\right\Vert _{S_{w}^{p},\chi }=\int U_{\chi ,p}\left( \xi
\right) \mathtt{d}\xi =\left\Vert U_{\chi ,p}\right\Vert _{L^{1}},\quad u\in
S_{w}\left( \mathbb{R}^{n}\right) ,
\end{equation*}%
then $\left\Vert \cdot \right\Vert _{S_{w}^{p},\chi }$ is a norm on $%
S_{w}^{p}\left( \mathbb{R}^{n}\right) $ and the topology that defines does
not depend on the choice of the function $\chi \in \mathcal{C}_{0}^{\infty
}\left( \mathbb{R}^{n}\right) \smallsetminus 0$.

$(\mathtt{c})$ Let $1\leq p\leq q\leq \infty $. Then 
\begin{equation*}
S_{w}^{1}\left( \mathbb{R}^{n}\right) \subset S_{w}^{p}\left( \mathbb{R}%
^{n}\right) \subset S_{w}^{q}\left( \mathbb{R}^{n}\right) \subset
S_{w}^{\infty }\left( \mathbb{R}^{n}\right) =S_{w}\left( \mathbb{R}%
^{n}\right) \subset \mathcal{BC}\left( \mathbb{R}^{n}\right) \subset 
\mathcal{S}^{\prime }\left( \mathbb{R}^{n}\right) .
\end{equation*}

$(\mathtt{d})$ If $\lambda \in \mathtt{End}_{%
\mathbb{R}
}\left( 
\mathbb{R}
^{n}\right) $ is invertible and $u\in S_{w}^{p}\left( \mathbb{R}^{n}\right) $%
, then $u_{\lambda }=u\circ \lambda \in S_{w}^{p}\left( \mathbb{R}%
^{n}\right) $ and there is $C\in \left( 0,+\infty \right) $ independent of $%
u $ and $\lambda $ such that%
\begin{equation*}
\left\Vert u_{\lambda }\right\Vert _{S_{w}^{p}}\leq C\left\vert \det \lambda
\right\vert ^{-n/p}\left( 1+\left\Vert \lambda \right\Vert \right)
^{n}\left\Vert u\right\Vert _{S_{w}^{p}}.
\end{equation*}
\end{proposition}

A proof of this proposition can be found for instance in \cite{Arsu}.

\begin{lemma}
Suppose that $\mathbb{R}^{n}=\mathbb{R}^{n_{1}}\times ...\times \mathbb{R}%
^{n_{j}}$. If $s_{1}>n_{1},...,s_{j}>n_{j}$ and $1\leq p\leq \infty $, then $%
\mathcal{K}_{p}^{\mathbf{s}}\left( 
\mathbb{R}
^{n}\right) \hookrightarrow S_{w}^{p}\left( \mathbb{R}^{n}\right) $.
\end{lemma}

\begin{proof}
Let $u\in \mathcal{K}_{p}^{\mathbf{s}}\left( 
\mathbb{R}
^{n}\right) $. Let $\chi ,\widetilde{\chi }\in \mathcal{C}_{0}^{\infty
}\left( \mathbb{R}^{n}\right) \smallsetminus 0$ be such that $\widetilde{%
\chi }=1$ on $\mathtt{supp}\chi $. For $y\in 
\mathbb{R}
^{n}$ we have 
\begin{equation*}
u\tau _{y}\chi =\left( u\tau _{y}\widetilde{\chi }\right) \left( \tau
_{y}\chi \right) \text{ }\Rightarrow \text{ }\widehat{u\tau _{y}\chi }%
=\left( 2\pi \right) ^{-n}\widehat{u\tau _{y}\widetilde{\chi }}\ast \widehat{%
\tau _{y}\chi }.
\end{equation*}%
Then by Peetre's inequality and by Cauchy-Schwarz inequality, we obtain 
\begin{eqnarray*}
\left\langle \left\langle \xi \right\rangle \right\rangle ^{\mathbf{s}%
}\left\vert \widehat{u\tau _{y}\chi }\left( \xi \right) \right\vert &\leq
&2^{\left\vert \mathbf{s}\right\vert _{1}/2}\left( 2\pi \right) ^{-n}\int
\left\langle \left\langle \eta \right\rangle \right\rangle ^{\mathbf{s}%
}\left\vert \widehat{u\tau _{y}\widetilde{\chi }}\left( \eta \right)
\right\vert \left\langle \left\langle \xi -\eta \right\rangle \right\rangle
^{\mathbf{s}}\left\vert \widehat{\tau _{y}\chi }\left( \xi -\eta \right)
\right\vert \mathtt{d}\xi \\
&\leq &C_{s,n}\left\Vert \left\langle \left\langle \cdot \right\rangle
\right\rangle ^{\mathbf{s}}\widehat{u\tau _{y}\widetilde{\chi }}\right\Vert
_{L^{2}}\left\Vert \left\langle \left\langle \cdot \right\rangle
\right\rangle ^{\mathbf{s}}\widehat{\tau _{y}\chi }\right\Vert _{L^{2}} \\
&=&C_{s,n}^{\prime }\left\Vert u\tau _{y}\widetilde{\chi }\right\Vert _{%
\mathcal{H}^{\mathbf{s}}}\left\Vert \chi \right\Vert _{\mathcal{H}^{\mathbf{s%
}}},
\end{eqnarray*}%
hence 
\begin{equation*}
\left\vert \widehat{u\tau _{y}\chi }\left( \xi \right) \right\vert \leq
C_{s,n}^{\prime }\left\Vert u\tau _{y}\widetilde{\chi }\right\Vert _{%
\mathcal{H}^{\mathbf{s}}}\left\Vert \chi \right\Vert _{\mathcal{H}^{\mathbf{s%
}}}\left\langle \left\langle \xi \right\rangle \right\rangle ^{-\mathbf{s}}.
\end{equation*}%
It follows that%
\begin{gather*}
U_{\chi ,p}\left( \xi \right) \leq C_{s,n}^{\prime }\left\Vert \chi
\right\Vert _{\mathcal{H}^{\mathbf{s}}}\left\langle \left\langle \xi
\right\rangle \right\rangle ^{-\mathbf{s}}\left\{ 
\begin{array}{ccc}
\sup_{y\in \mathbb{R}^{n}}\left\Vert u\tau _{y}\widetilde{\chi }\right\Vert
_{\mathcal{H}^{\mathbf{s}}} & \text{\textit{if}} & p=\infty \\ 
\left( \int \left\Vert u\tau _{y}\widetilde{\chi }\right\Vert _{\mathcal{H}^{%
\mathbf{s}}}^{p}\mathtt{d}y\right) ^{\frac{1}{p}} & \text{\textit{if}} & 
1\leq p<\infty%
\end{array}%
\right. \\
\leq C_{s,n}^{\prime }\left\Vert \chi \right\Vert _{\mathcal{H}^{\mathbf{s}%
}}\left\langle \left\langle \xi \right\rangle \right\rangle ^{-\mathbf{s}%
}\left\{ 
\begin{array}{ccc}
\left\Vert u\right\Vert _{\mathbf{s},\infty ,\widetilde{\chi }} & \text{%
\textit{if}} & p=\infty \\ 
\left\Vert u\right\Vert _{\mathbf{s},p,\widetilde{\chi }} & \text{\textit{if}%
} & 1\leq p<\infty%
\end{array}%
\right. \\
\leq C_{s,n}^{\prime }\left\Vert \chi \right\Vert _{\mathcal{H}^{\mathbf{s}%
}}\left\Vert u\right\Vert _{\mathbf{s},p,\widetilde{\chi }}\left\langle
\left\langle \xi \right\rangle \right\rangle ^{-\mathbf{s}},
\end{gather*}%
which implies that%
\begin{equation*}
\left\Vert u\right\Vert _{S_{w}^{p},\chi }=\left\Vert U_{\chi ,p}\right\Vert
_{L^{1}}\leq C_{s,n}^{\prime }\left\Vert \chi \right\Vert _{\mathcal{H}^{%
\mathbf{s}}}\left\Vert \left\langle \left\langle \cdot \right\rangle
\right\rangle ^{-\mathbf{s}}\right\Vert _{L^{1}}\left\Vert u\right\Vert _{%
\mathbf{s},p,\widetilde{\chi }},\quad u\in \mathcal{K}_{p}^{\mathbf{s}%
}\left( 
\mathbb{R}
^{n}\right) .
\end{equation*}
\end{proof}

In fact, the result proved in the particular orthogonal decomposition $%
\mathbb{R}^{n}=\mathbb{R}^{n_{1}}\times ...\times \mathbb{R}^{n_{j}}$ is
true for arbitrary orthogonal decomposition of $\mathbb{R}^{n}$.

First we shall associate to an orthogonal decomposition a family of Banach
spaces such as those introduced in Section 2 and Section 3. Let $j\in
\left\{ 1,...,n\right\} $, $1\leq p\leq \infty $ and $\mathbf{s}=\left(
s_{1},...,s_{j}\right) \in 
\mathbb{R}
^{j}$. Let $\mathbb{V}=\left( V_{1},...,V_{j}\right) $ denote an orthogonal
decomposition, i.e.%
\begin{equation*}
\mathbb{R}^{n}=V_{1}\oplus ...\oplus V_{j}.
\end{equation*}%
We introduce the Banach space $\mathcal{H}_{\mathbb{V}}^{\mathbf{s}}\left( 
\mathbb{R}
^{n}\right) =\mathcal{H}_{V_{1},...,V_{j}}^{s_{1},...,s_{j}}\left( 
\mathbb{R}
^{n}\right) $ defined by%
\begin{eqnarray*}
\mathcal{H}_{\mathbb{V}}^{\mathbf{s}}\left( 
\mathbb{R}
^{n}\right) &=&\left\{ u\in \mathcal{S}^{\prime }\left( 
\mathbb{R}
^{n}\right) :\left( 1-\triangle _{V_{1}}\right) ^{s_{1}/2}\otimes ...\otimes
\left( 1-\triangle _{V_{k}}\right) ^{s_{j}/2}u\in L^{2}\left( \mathbb{R}%
^{n}\right) \right\} , \\
\left\Vert u\right\Vert _{\mathcal{H}_{\mathbb{V}}^{\mathbf{s}}}
&=&\left\Vert \left( 1-\triangle _{V_{1}}\right) ^{s_{1}/2}\otimes
...\otimes \left( 1-\triangle _{V_{k}}\right) ^{s_{j}/2}u\right\Vert
_{L^{2}},\quad u\in \mathcal{H}_{\mathbb{V}}^{\mathbf{s}}.
\end{eqnarray*}%
We recall that if $\lambda :%
\mathbb{R}
^{n}\rightarrow V$ is an isometric linear isomorphism, where $V$ is an
euclidean space, then $\left( 1-\triangle _{V}\right) ^{\gamma /2}u=\left(
1-\triangle \right) ^{\gamma /2}\left( u\circ \lambda \right) \circ \lambda
^{-1}$ (see appendix \ref{ks16}).

\begin{definition}
Let $1\leq p\leq \infty $, $\mathbf{s}\in 
\mathbb{R}
^{j}$ and $u\in \mathcal{D}^{\prime }\left( 
\mathbb{R}
^{n}\right) $. We say that $u$ belongs to $\mathcal{K}_{p,\mathbb{V}}^{%
\mathbf{s}}\left( 
\mathbb{R}
^{n}\right) $ if there is $\chi \in \mathcal{C}_{0}^{\infty }\left( \mathbb{R%
}^{n}\right) \smallsetminus 0$ such that the measurable function $%
\mathbb{R}
^{n}\ni y\rightarrow \left\Vert u\tau _{y}\chi \right\Vert _{\mathcal{H}_{%
\mathbb{V}}^{\mathbf{s}}}\in 
\mathbb{R}
$ belongs to $L^{p}\left( 
\mathbb{R}
^{n}\right) $. We put%
\begin{eqnarray*}
\left\Vert u\right\Vert _{\mathcal{K}_{p,\mathbb{V}}^{\mathbf{s}},\chi }
&=&\left( \int \left\Vert u\tau _{y}\chi \right\Vert _{\mathcal{H}_{\mathbb{V%
}}^{\mathbf{s}}}^{p}\mathtt{d}y\right) ^{\frac{1}{p}},\quad 1\leq p<\infty ,
\\
\left\Vert u\right\Vert _{\mathcal{K}_{\infty ,\mathbb{V}}^{\mathbf{s}},\chi
} &\equiv &\left\Vert u\right\Vert _{\mathbf{s},\mathtt{ul},\chi
}=\sup_{y}\left\Vert u\tau _{y}\chi \right\Vert _{\mathcal{H}^{\mathbf{s}}}.
\end{eqnarray*}
\end{definition}

If $\lambda :%
\mathbb{R}
^{n}\rightarrow \mathbb{R}^{n}$ is a orthogonal transformation satisfying $%
\lambda \left( \mathbb{R}^{n_{1}}\right) =V_{1}$, ..., $\lambda \left( 
\mathbb{R}^{n_{k}}\right) =V_{k}$, then the operators%
\begin{eqnarray*}
\mathcal{H}_{\mathbb{V}}^{\mathbf{s}}\left( 
\mathbb{R}
^{n}\right) &\ni &u\rightarrow u\circ \lambda \in \mathcal{H}^{\mathbf{s}%
}\left( 
\mathbb{R}
^{n}\right) , \\
\mathcal{K}_{p,\mathbb{V}}^{\mathbf{s}}\left( 
\mathbb{R}
^{n}\right) &\ni &u\rightarrow u\circ \lambda \in \mathcal{K}_{p}^{\mathbf{s}%
}\left( 
\mathbb{R}
^{n}\right) ,
\end{eqnarray*}%
are linear isometric isomorphisms (see appendix \ref{ks16}).

A consequence of Corollay \ref{ks17} is is the following remark.

\begin{remark}
Let $\mathbf{s},\mathbf{t}\in 
\mathbb{R}
^{j}$, $1\leq p_{0}<\infty $, $1\leq p_{1}\leq \infty $, $0<\theta <1$ and 
\begin{equation*}
\frac{1}{p}=\frac{1-\theta }{p_{0}}+\frac{\theta }{p_{1}},\quad \mathbf{%
\sigma }=\left( 1-\theta \right) \mathbf{s}+\theta \mathbf{t}.
\end{equation*}%
Then 
\begin{equation*}
\left[ \mathcal{K}_{p_{0},\mathbb{V}}^{\mathbf{s}}\left( 
\mathbb{R}
^{n}\right) ,\mathcal{K}_{p_{1},\mathbb{V}}^{\mathbf{t}}\left( 
\mathbb{R}
^{n}\right) \right] _{\theta }=\mathcal{K}_{p,\mathbb{V}}^{\mathbf{\sigma }%
}\left( 
\mathbb{R}
^{n}\right) .
\end{equation*}
\end{remark}

The previous result and the invariance of the ideals $S_{w}^{p}$ to the
composition with $\lambda \in \mathtt{GL}\left( n,%
\mathbb{R}
\right) $ imply the following corollary.

\begin{corollary}
Suppose that $\mathbb{R}^{n}=V_{1}\oplus ...\oplus V_{j}$. If $s_{1}>\dim
V_{1},...,s_{j}>\dim V_{j}$ and $1\leq p\leq \infty $, then $\mathcal{K}_{p,%
\mathbb{V}}^{\mathbf{s}}\left( 
\mathbb{R}
^{n}\right) \hookrightarrow S_{w}^{p}\left( \mathbb{R}^{n}\right) $.
\end{corollary}

The embedding theorem allows us to deal with Schatten-von Neumann class
properties of pseudo-differential operators.

Let $\tau \in \mathtt{End}_{%
\mathbb{R}
}\left( 
\mathbb{R}
^{n}\right) \equiv M_{n\times n}\left( 
\mathbb{R}
\right) $, $a\in \mathcal{S}\left( \mathbb{R}^{n}\times \mathbb{R}%
^{n}\right) $, $v\in \mathcal{S}\left( \mathbb{R}^{n}\times \mathbb{R}%
^{n}\right) $. We define 
\begin{eqnarray*}
\mathtt{Op}_{\tau }\left( a\right) v\left( x\right) &=&a^{\tau }\left(
X,D\right) v\left( x\right) \\
&=&\left( 2\pi \right) ^{-n}\iint \mathtt{e}^{\mathtt{i}\left\langle
x-y,\eta \right\rangle }a\left( \left( 1-\tau \right) x+\tau y,\eta \right)
v\left( y\right) \mathtt{d}y\mathtt{d}\eta .
\end{eqnarray*}%
If $u,v\in \mathcal{S}\left( \mathbb{R}^{n}\right) $, then%
\begin{eqnarray*}
\left\langle \mathtt{Op}_{\tau }\left( a\right) v,u\right\rangle &=&\left(
2\pi \right) ^{-n}\iiint \mathtt{e}^{\mathtt{i}\left\langle x-y,\eta
\right\rangle }a\left( \left( 1-\tau \right) x+\tau y,\eta \right) u\left(
x\right) v\left( y\right) \mathtt{d}x\mathtt{d}y\mathtt{d}\eta \\
&=&\left\langle \left( \left( 1\otimes \mathcal{F}^{-1}\right) a\right)
\circ \mathtt{C}_{\tau },u\otimes v\right\rangle ,
\end{eqnarray*}%
where 
\begin{equation*}
\mathtt{C}_{\tau }:\mathbb{R}^{n}\times \mathbb{R}^{n}\rightarrow \mathbb{R}%
^{n}\times \mathbb{R}^{n},\quad \mathtt{C}_{\tau }\left( x,y\right) =\left(
\left( 1-\tau \right) x+\tau y,x-y\right) .
\end{equation*}%
We can define $\mathtt{Op}_{\tau }\left( a\right) $ as an operator in $%
\mathcal{B}\left( \mathcal{S},\mathcal{S}^{\prime }\right) $ for any $a\in 
\mathcal{S}^{\prime }\left( \mathbb{R}^{n}\times \mathbb{R}^{n}\right) $ by%
\begin{eqnarray*}
\left\langle \mathtt{Op}_{\tau }\left( a\right) v,u\right\rangle _{\mathcal{S%
},\mathcal{S}^{\prime }} &=&\left\langle \mathcal{K}_{\mathtt{Op}_{\tau
}\left( a\right) },u\otimes v\right\rangle , \\
\mathcal{K}_{\mathtt{Op}_{\tau }\left( a\right) } &=&\left( \left( 1\otimes 
\mathcal{F}^{-1}\right) a\right) \circ \mathtt{C}_{\tau }
\end{eqnarray*}

\begin{theorem}
Suppose that $\mathbb{R}^{n}\times \mathbb{R}^{n}=V_{1}\oplus ...\oplus
V_{j} $.

$(\mathtt{a})$ Let $1\leq p<\infty $, $\tau \in \mathtt{End}_{%
\mathbb{R}
}\left( 
\mathbb{R}
^{n}\right) \equiv M_{n\times n}\left( 
\mathbb{R}
\right) $ and $a\in \mathcal{K}_{p,\mathbb{V}}^{\mathbf{s}}\left( \mathbb{R}%
^{n}\times \mathbb{R}^{n}\right) $. If $s_{1}>\dim V_{1},...,s_{j}>\dim
V_{j} $, then 
\begin{equation*}
\mathtt{Op}_{\tau }\left( a\right) =a^{\tau }\left( X,D\right) \in \mathcal{B%
}_{p}\left( L^{2}\left( 
\mathbb{R}
^{n}\right) \right)
\end{equation*}%
where $\mathcal{B}_{p}\left( L^{2}\left( 
\mathbb{R}
^{n}\right) \right) $ denote the Schatten ideal of compact operators whose
singular values lie in $l^{p}$. We have 
\begin{equation*}
\left\Vert \mathtt{Op}_{\tau }\left( a\right) \right\Vert _{\mathcal{B}%
_{p}}\leq Cst\left\Vert a\right\Vert \left\Vert u\right\Vert _{\mathcal{K}%
_{p,\mathbb{V}}^{\mathbf{s}}}.
\end{equation*}%
Moreover, the mapping%
\begin{equation*}
\mathtt{End}_{%
\mathbb{R}
}\left( 
\mathbb{R}
^{n}\right) \ni \tau \rightarrow \mathtt{Op}_{\tau }\left( a\right) =a^{\tau
}\left( X,D\right) \in \mathcal{B}_{p}\left( L^{2}\left( 
\mathbb{R}
^{n}\right) \right)
\end{equation*}%
is continuous.

$(\mathtt{b})$ Let $\tau \in \mathtt{End}_{%
\mathbb{R}
}\left( 
\mathbb{R}
^{n}\right) \equiv M_{n\times n}\left( 
\mathbb{R}
\right) $ and $a\in \mathcal{K}_{\infty ,\mathbb{V}}^{\mathbf{s}}\left( 
\mathbb{R}^{n}\times \mathbb{R}^{n}\right) $. If $s_{1}>\dim V_{1}$, $...$, $%
s_{j}>\dim V_{j}$, then 
\begin{equation*}
\mathtt{Op}_{\tau }\left( a\right) =a^{\tau }\left( X,D\right) \in \mathcal{B%
}\left( L^{2}\left( 
\mathbb{R}
^{n}\right) \right)
\end{equation*}%
We have 
\begin{equation*}
\left\Vert \mathtt{Op}_{\tau }\left( a\right) \right\Vert _{\mathcal{B}%
\left( L^{2}\left( 
\mathbb{R}
^{n}\right) \right) }\leq Cst\left\Vert a\right\Vert _{\mathcal{K}_{\infty ,%
\mathbb{V}}^{\mathbf{s}}}.
\end{equation*}%
Moreover, the mapping%
\begin{equation*}
\mathtt{End}_{%
\mathbb{R}
}\left( 
\mathbb{R}
^{n}\right) \ni \tau \rightarrow \mathtt{Op}_{\tau }\left( a\right) =a^{\tau
}\left( X,D\right) \in \mathcal{B}\left( L^{2}\left( 
\mathbb{R}
^{n}\right) \right)
\end{equation*}%
is continuous.
\end{theorem}

\begin{proof}
This theorem is a consequence of the previous corollary and the fact that it
is true for pseudo-differential operators with symbols in $S_{w}^{p}\left( 
\mathbb{R}^{n}\right) $ (see for instance \cite{Arsu} for $1\leq p<\infty $
and \cite{Boulkhemair 2} for $p=\infty $).
\end{proof}

\begin{theorem}
Suppose that $\mathbb{R}^{n}\times \mathbb{R}^{n}=V_{1}\oplus ...\oplus
V_{j} $. Let $\mathbf{n}=\left( \dim V_{1},...,\dim V_{j}\right) $, $\mu >1$
and $1\leq p<\infty $. If $\tau \in \mathtt{End}_{%
\mathbb{R}
}\left( 
\mathbb{R}
^{n}\right) \equiv M_{n\times n}\left( 
\mathbb{R}
\right) $ and $a\in \mathcal{K}_{p,\mathbb{V}}^{\mu \mathbf{n}\left\vert
1-2/p\right\vert }\left( \mathbb{R}^{n}\times \mathbb{R}^{n}\right) $ then 
\begin{equation*}
\mathtt{Op}_{\tau }\left( a\right) =a^{\tau }\left( X,D\right) \in \mathcal{B%
}_{p}\left( L^{2}\left( 
\mathbb{R}
^{n}\right) \right) .
\end{equation*}%
Moreover, the mapping%
\begin{equation*}
\mathtt{End}_{%
\mathbb{R}
}\left( 
\mathbb{R}
^{n}\right) \ni \tau \rightarrow \mathtt{Op}_{\tau }\left( a\right) =a^{\tau
}\left( X,D\right) \in \mathcal{B}_{p}\left( L^{2}\left( 
\mathbb{R}
^{n}\right) \right)
\end{equation*}%
is continuous.
\end{theorem}

\begin{proof}
The Schwartz kernel of the operator $\mathtt{Op}_{\tau }\left( a\right) $ is 
$\left( \left( 1\otimes \mathcal{F}^{-1}\right) a\right) \circ \mathtt{C}%
_{\tau }$. Therefore, $a\in \mathcal{K}_{2,\mathbb{V}}^{\mathbf{0}}\left( 
\mathbb{R}^{n}\times \mathbb{R}^{n}\right) \equiv L^{2}\left( \mathbb{R}%
^{n}\times \mathbb{R}^{n}\right) $ implies that $\mathtt{Op}_{\tau }\left(
a\right) \in \mathcal{B}_{2}\left( L^{2}\left( 
\mathbb{R}
^{n}\right) \right) $. Next we use the interpolation properties of
Kato-Sobolev spaces $\mathcal{K}_{p,\mathbb{V}}^{\mathbf{s}}\left( 
\mathbb{R}
^{n}\right) $ and of the Schatten ideals $\mathcal{B}_{p}\left( L^{2}\left( 
\mathbb{R}
^{n}\right) \right) $ to finish the theorem.%
\begin{equation*}
\begin{array}{c}
\left[ \mathcal{K}_{2,\mathbb{V}}^{\mathbf{0}}\left( 
\mathbb{R}
^{n}\times \mathbb{R}^{n}\right) ,\mathcal{K}_{1,\mathbb{V}}^{\mu \mathbf{n}%
}\left( 
\mathbb{R}
^{n}\times \mathbb{R}^{n}\right) \right] _{\frac{2}{p}-1}=\mathcal{K}_{p,%
\mathbb{V}}^{\mu \mathbf{n}\left( 2/p-1\right) }\left( \mathbb{R}^{n}\times 
\mathbb{R}^{n}\right)  \\ 
\left[ \mathcal{B}_{2}\left( L^{2}\left( 
\mathbb{R}
^{n}\right) \right) ,\mathcal{B}_{1}\left( L^{2}\left( 
\mathbb{R}
^{n}\right) \right) \right] _{\frac{2}{p}-1}=\mathcal{B}_{p}\left(
L^{2}\left( 
\mathbb{R}
^{n}\right) \right) 
\end{array}%
,\quad 1\leq p\leq 2,
\end{equation*}%
\begin{equation*}
\begin{array}{c}
\left[ \mathcal{K}_{2,\mathbb{V}}^{\mathbf{0}}\left( 
\mathbb{R}
^{n}\times \mathbb{R}^{n}\right) ,\mathcal{K}_{\infty ,\mathbb{V}}^{\mu 
\mathbf{n}}\left( 
\mathbb{R}
^{n}\times \mathbb{R}^{n}\right) \right] _{1-\frac{2}{p}}=\mathcal{K}_{p,%
\mathbb{V}}^{\mu \mathbf{n}\left( 2/p-1\right) }\left( \mathbb{R}^{n}\times 
\mathbb{R}^{n}\right)  \\ 
\left[ \mathcal{B}_{2}\left( L^{2}\left( 
\mathbb{R}
^{n}\right) \right) ,\mathcal{B}\left( L^{2}\left( 
\mathbb{R}
^{n}\right) \right) \right] _{1-\frac{2}{p}}=\mathcal{B}_{p}\left(
L^{2}\left( 
\mathbb{R}
^{n}\right) \right) 
\end{array}%
,\quad 2\leq p<\infty .
\end{equation*}
\end{proof}

\appendix{}

\section{Convolution operators\label{ks16}}

Let $V$ be an $n$ dimensional real vector, $V^{\prime }$ its dual and $%
\left\langle \cdot ,\cdot \right\rangle _{V,V^{\prime }}:V\times V^{\prime
}\rightarrow 
\mathbb{R}
$ the duality form. We define the (Fourier) transforms 
\begin{eqnarray*}
\mathcal{F}_{V},\ \overline{\mathcal{F}}_{V} &:&\mathcal{S}^{\prime }\left(
V\right) \rightarrow \mathcal{S}^{\prime }\left( V^{\prime }\right) , \\
\mathcal{F}_{V^{\prime }},\ \overline{\mathcal{F}}_{V^{\prime }} &:&\mathcal{%
S}^{\prime }\left( V^{\prime }\right) \rightarrow \mathcal{S}^{\prime
}\left( V\right) ,
\end{eqnarray*}%
by 
\begin{eqnarray*}
(\mathcal{F}_{V}u)(\xi ) &=&\int_{V}\mathtt{e}^{-\mathtt{i}\left\langle
x,\xi \right\rangle _{V,V^{\prime }}}u(x)\mathtt{d}\mu \left( x\right) , \\
(\overline{\mathcal{F}}_{V}u)(\xi ) &=&\int_{V}\mathtt{e}^{+\mathtt{i}%
\left\langle x,\xi \right\rangle _{V,V^{\prime }}}u(x)\mathtt{d}\mu \left(
x\right) , \\
(\mathcal{F}_{V^{\prime }}v)(x) &=&\int_{V^{\prime }}\mathtt{e}^{-\mathtt{i}%
\left\langle x,\xi \right\rangle _{V,V^{\prime }}}v(\xi )\mathtt{d}\mu
^{\prime }\left( \xi \right) , \\
(\overline{\mathcal{F}}_{V^{\prime }}v)(x) &=&\int_{V^{\prime }}\mathtt{e}^{+%
\mathtt{i}\left\langle x,\xi \right\rangle _{V,V^{\prime }}}v(p\xi )\mathtt{d%
}\mu ^{\prime }\left( \xi \right) .
\end{eqnarray*}%
Here $\mu $ is a Lebesgue (Haar) measure in $V$ and $\mu ^{\prime }$ is the
dual one in $V^{\prime }$ such that Fourier's inversion formulas hold: $%
\overline{\mathcal{F}}_{V^{\prime }}\circ \mathcal{F}_{V}=\mathtt{1}_{%
\mathcal{S}^{\prime }\left( V\right) }$, $\mathcal{F}_{V}\circ \overline{%
\mathcal{F}}_{V^{\prime }}=\mathtt{1}_{\mathcal{S}^{\prime }\left( V^{\prime
}\right) }$. Replacing $\mu $ by $c\mu $ one must change $\mu ^{\prime }$ to 
$c^{-1}\mu ^{\prime }$. These (Fourier) transforms defined above are unitary
operators $\mathcal{F}_{V},\overline{\mathcal{F}}_{V}:L^{2}\left( V\right)
\rightarrow L^{2}\left( V^{\prime }\right) $ and $\mathcal{F}_{V^{\prime }},%
\overline{\mathcal{F}}_{V^{\prime }}:L^{2}\left( V^{\prime }\right)
\rightarrow L^{2}\left( V\right) $.

If $V=%
\mathbb{R}
^{n}$ we use $\mu =m$ -Lebesgue measure in $V$ with the dual $\mu ^{\prime
}=\left( 2\pi \right) ^{-n}m$ in $V^{\prime }$, so that $\mathcal{F}_{V}=%
\mathcal{F}$ and $\overline{\mathcal{F}}_{V^{\prime }}=\mathcal{F}^{-1}$,
where $\mathcal{F}$ is the usual Fourier transform.

Let $\lambda :%
\mathbb{R}
^{n}\rightarrow V$ be a linear isomorphism. Then there is $c_{\lambda }\in 
\mathbb{C}
^{\ast }$ such that for $u\in \mathcal{S}\left( V\right) $ we have 
\begin{equation*}
\mathcal{F}\left( u\circ \lambda \right) =c_{\lambda }\left( \mathcal{F}%
_{V}u\right) \circ {}^{t}\lambda ^{-1}.
\end{equation*}%
Indeed, the measure $m\circ \lambda ^{-1}$ is a Haar measure in $V$ and
therefore there is $c_{\lambda }\in 
\mathbb{C}
^{\ast }$ such that $m\circ \lambda ^{-1}=c_{\lambda }\mu $. Then 
\begin{eqnarray*}
\mathcal{F}\left( u\circ \lambda \right) &=&\int_{%
\mathbb{R}
^{n}}\mathtt{e}^{-\mathtt{i}\left\langle y,\cdot \right\rangle }\left(
u\circ \lambda \right) \left( y\right) \mathtt{d}y=\int_{V}\mathtt{e}^{-%
\mathtt{i}\left\langle \lambda ^{-1}x,\cdot \right\rangle }u\left( x\right) 
\mathtt{d}\left( m\circ \lambda ^{-1}\right) \left( x\right) \\
&=&c_{\lambda }\int_{V}\mathtt{e}^{-\mathtt{i}\left\langle x,^{t}\lambda
^{-1}\cdot \right\rangle _{V,V^{\prime }}}u\left( x\right) \mathtt{d}\mu
\left( x\right) =c_{\lambda }\left( \mathcal{F}_{V}u\right) \circ
{}^{t}\lambda ^{-1}.
\end{eqnarray*}%
Similarly it is shown that there is $c_{\lambda }^{\prime }\in 
\mathbb{C}
^{\ast }$ such that for $w\in \mathcal{S}\left( V^{\prime }\right) $ we have%
\begin{equation*}
\mathcal{F}^{-1}\left( w\circ {}^{t}\lambda ^{-1}\right) =c_{\lambda
}^{\prime }\left( \overline{\mathcal{F}}_{V^{\prime }}w\right) \circ \lambda
.
\end{equation*}%
Let us note that the above formulas hold for elements in $L^{2}$, $\mathcal{S%
}^{\prime }$, .... Since the measures $\mu $ and $\mu ^{\prime }$ are dual,
we obtain that $c_{\lambda }c_{\lambda }^{\prime }=1$.%
\begin{eqnarray*}
u\circ \lambda &=&\mathcal{F}^{-1}\mathcal{F}\left( u\circ \lambda \right)
=c_{\lambda }\mathcal{F}^{-1}\left( \left( \mathcal{F}_{V}u\right) \circ
{}^{t}\lambda ^{-1}\right) \\
&=&c_{\lambda }c_{\lambda }^{\prime }\left( \overline{\mathcal{F}}%
_{V^{\prime }}\mathcal{F}_{V}u\right) \circ \lambda =c_{\lambda }c_{\lambda
}^{\prime }\left( u\circ \lambda \right) \\
&\Rightarrow &c_{\lambda }c_{\lambda }^{\prime }=1.
\end{eqnarray*}

Let $a:V^{\prime }\rightarrow 
\mathbb{C}
$ be a measurable function. We introduce the densely defined operator $%
a\left( D_{V}\right) $ defined by 
\begin{equation*}
a\left( D_{V}\right) =\overline{\mathcal{F}}_{V^{\prime }}M_{a}\mathcal{F}%
_{V}.
\end{equation*}

\begin{lemma}
Let $\lambda :%
\mathbb{R}
^{n}\rightarrow V$ be a linear isomorphism. Then

$\left( \mathtt{a}\right) $ $u\in \mathcal{D}\left( a\left( D_{V}\right)
\right) $ if and only if $u\circ \lambda \in \mathcal{D}\left( \left( a\circ
{}^{t}\lambda ^{-1}\right) \left( D\right) \right) $.

$\left( \mathtt{b}\right) $ For any $u\in \mathcal{D}\left( a\left(
D_{V}\right) \right) $ 
\begin{equation*}
\left( a\circ {}^{t}\lambda ^{-1}\right) \left( D\right) \left( u\circ
\lambda \right) =a\left( D_{V}\right) u\circ \lambda .
\end{equation*}
\end{lemma}

\begin{proof}
$\left( \mathtt{a}\right) $ The following statements are equivalent:

$\left( \mathtt{i}\right) $ $u\in \mathcal{D}\left( a\left( D_{V}\right)
\right) .$

$\left( \mathtt{ii}\right) $ $a\mathcal{F}_{V}u\in L^{2}\left( V\right) .$

$\left( \mathtt{iii}\right) $ $\left( a\mathcal{F}_{V}u\right) \circ
{}^{t}\lambda ^{-1}\in L^{2}\left( 
\mathbb{R}
^{n}\right) .$

$\left( \mathtt{iv}\right) $ $\left( a\circ {}^{t}\lambda ^{-1}\right)
\left( \mathcal{F}_{V}u\circ {}^{t}\lambda ^{-1}\right) \in L^{2}\left( 
\mathbb{R}
^{n}\right) .$

$\left( \mathtt{v}\right) $ $\left( a\circ {}^{t}\lambda ^{-1}\right) 
\mathcal{F}\left( u\circ \lambda \right) \in L^{2}\left( 
\mathbb{R}
^{n}\right) .$

$\left( \mathtt{vi}\right) $ $u\circ \lambda \in \mathcal{D}\left( \left(
a\circ {}^{t}\lambda ^{-1}\right) \left( D\right) \right) .$

$\left( \mathtt{b}\right) $ Let $u\in \mathcal{D}\left( a\left( D_{V}\right)
\right) $. Then 
\begin{eqnarray*}
\left( a\circ {}^{t}\lambda ^{-1}\right) \left( D\right) \left( u\circ
\lambda \right) &=&\mathcal{F}^{-1}M_{a\circ {}^{t}\lambda ^{-1}}\mathcal{F}%
\left( u\circ \lambda \right) =c_{\lambda }\mathcal{F}^{-1}\left[ \left(
M_{a}\mathcal{F}_{V}u\right) \circ {}^{t}\lambda ^{-1}\right] \\
&=&c_{\lambda }c_{\lambda }^{\prime }\left( \overline{\mathcal{F}}%
_{V^{\prime }}M_{a}\mathcal{F}_{V}u\right) \circ \lambda =a\left(
D_{V}\right) u\circ \lambda .
\end{eqnarray*}
\end{proof}

If $a\in \mathcal{C}_{\mathtt{pol}}^{\infty }\left( V^{\prime }\right) $
then the operator $a\left( D_{V}\right) =\overline{\mathcal{F}}_{V^{\prime
}}M_{a}\mathcal{F}_{V}$ belongs to $\mathcal{B}\left( \mathcal{S}^{\prime
}\left( V\right) \right) $. In this case we have the following result.

\begin{lemma}
Let $\lambda :%
\mathbb{R}
^{n}\rightarrow V$ be a linear isomorphism. If $a\in \mathcal{C}_{\mathtt{pol%
}}^{\infty }\left( V^{\prime }\right) $ then for any $u\in \mathcal{S}%
^{\prime }\left( V\right) $ 
\begin{equation*}
\left( a\circ {}^{t}\lambda ^{-1}\right) \left( D\right) \left( u\circ
\lambda \right) =a\left( D_{V}\right) u\circ \lambda .
\end{equation*}
\end{lemma}

\begin{proof}
By continuity and density, it suffice to prove the equality for $u\in 
\mathcal{S}\left( V\right) $. But this is done in previous lemma.
\end{proof}

Suppose that $\left( V,\left\vert \cdot \right\vert _{V}\right) $ is an
euclidian space and that $a$ is symbol of the form 
\begin{equation*}
a=b\left( \left\vert \cdot \right\vert _{V^{\prime }}^{2}\right) ,
\end{equation*}%
where $b\in \mathcal{C}_{\mathtt{pol}}^{\infty }\left( 
\mathbb{R}
\right) $ and $\left\vert \cdot \right\vert _{V^{\prime }}$ is the dual norm
of $\left\vert \cdot \right\vert _{V}$.

Let $\lambda :%
\mathbb{R}
^{n}\rightarrow V$ be a linear isometric isomorphism. Then $^{t}\lambda
^{-1}:%
\mathbb{R}
^{n}\rightarrow V^{\prime }$ is a linear isometric isomorphism so that $%
\left\vert \cdot \right\vert _{V^{\prime }}\circ {}^{t}\lambda
^{-1}=\left\vert \cdot \right\vert $, where $\left\vert \cdot \right\vert $
is the euclidian in $%
\mathbb{R}
^{n}$. It follows that 
\begin{equation*}
b\left( \left\vert \cdot \right\vert _{V^{\prime }}^{2}\right) \circ
{}^{t}\lambda ^{-1}=b\left( \left\vert \cdot \right\vert ^{2}\right) .
\end{equation*}%
Using the previous lemma we obtain the following result.

\begin{corollary}
Let $\lambda :%
\mathbb{R}
^{n}\rightarrow V$ be a linear isometric isomorphism. If $b\in \mathcal{C}_{%
\mathtt{pol}}^{\infty }\left( 
\mathbb{R}
\right) $ then for any $u\in \mathcal{S}^{\prime }\left( V\right) $ 
\begin{equation*}
b\left( \left\vert D\right\vert ^{2}\right) \left( u\circ \lambda \right)
=b\left( \left\vert D_{V}\right\vert ^{2}\right) u\circ \lambda .
\end{equation*}
\end{corollary}

\end{document}